\documentclass{article}

\usepackage{arxiv}

\usepackage[utf8]{inputenc} 
\usepackage[T1]{fontenc}    
\usepackage{hyperref}       
\usepackage{url}            
\usepackage{booktabs}       
\usepackage{amsfonts}       
\usepackage{nicefrac}       
\usepackage{microtype}      
\usepackage{lipsum}
\usepackage{graphicx}
\graphicspath{ {./images/} }
\usepackage{amsmath}
\usepackage[numbers,sort&compress]{natbib}
\usepackage{enumitem}
\usepackage{xcolor}
\usepackage{cleveref}
\usepackage{subfigure}
\usepackage{algorithm}
\usepackage{algpseudocodex}
\usepackage{enumitem}

\usepackage{amsthm}
\newtheorem{proposition}{Proposition}

\title{GoRINNs: Godunov-Riemann Informed Neural Networks for Learning Hyperbolic Conservation Laws}

\author{
\textbf{Dimitrios G. Patsatzis\textcolor{blue}{$^{1}$}, Mario di Bernardo\textcolor{blue}{$^{1,2}$}, Lucia Russo\textcolor{blue}{$^{3}$}, Constantinos Siettos\textcolor{blue}{$^{4,}$}\thanks{Corresponding author, email: \texttt{constantinos.siettos@unina.it}}}
{}\\
\textcolor{blue}{$^{(1)}$}Modelling Engineering Risk and Complexity, \emph{Scuola Superiore Meridionale}, Naples, Italy \\
\textcolor{blue}{$^{(2)}$}Dept. of Electrical Engineering and Information Technology, \emph{University of Naples Federico II}, Naples, Italy\\
\textcolor{blue}{$^{(3)}$}Institute of Science and Technology for Energy and Sustainable Mobility,\\ \emph{Consiglio Nazionale delle Ricerche}, Naples, Italy\\
\textcolor{blue}{$^{(4)}$}Dept. of Mathematics and Applications ``Renato Caccioppoli'', \emph{University of Naples Federico II}, Naples, Italy\\
}

\begin{document}
\maketitle
\begin{abstract}
We present GoRINNs:~numerical analysis-informed (shallow) neural networks for the solution of inverse problems of non-linear systems of conservation laws.~GoRINNs is a hybrid/blended machine learning scheme based on high-resolution Godunov schemes for the solution of the Riemann problem in hyperbolic Partial Differential Equations (PDEs).~In contrast to other existing machine learning methods that learn the numerical fluxes or just parameters of conservative Finite Volume methods, relying on deep neural networks (that may lead to poor approximations due to the computational complexity involved in their training), GoRINNs learn the closures of the conservation laws per se based on ``intelligently`` numerical-assisted shallow neural networks.~Due to their structure, IN particular, GoRINNs provide explainable, conservative schemes, that solve the inverse problem for hyperbolic PDEs,
on the basis of approximate Riemann solvers that satisfy the Rankine-Hugoniot condition.~The performance of GoRINNs is assessed via four benchmark problems, namely the Burger's, the Shallow Water, the Lighthill-Whitham-Richards and the Payne-Whitham traffic flow models.~The solution profiles of these PDEs exhibit shock waves, rarefactions and/or contact discontinuities at finite times.~We demonstrate that GoRINNs provide a very high accuracy both in the smooth and discontinuous regions.     
\end{abstract}


\section{Introduction}
\label{sec:Intro}
Conservation laws (of mass, momentum, and energy) are fundamental physical laws present in many problems including fluid mechanics including multiphase problems \cite{batchelor2000introduction,gibou2007level}, traffic flow and crowd dynamics \cite{treiber2013traffic,bellomo2023human},  electromagnetism \cite{milton2024classical}, collisionless shocks for astrophysical phenomena\cite{treumann2009fundamentals} among others, in which a physical quantity is conserved in time over an isolated spatial domain \cite{leveque2002finite,dafermos2005hyperbolic}.~The conserved quantities are physically determined  or can arise in complex systems at the macroscopic level from the microscopic interactions (particles, molecules, vehicles, individuals, etc); see e.g., \cite{chapman1990mathematical}.~Conservation laws are often modelled with hyperbolic, (systems of) Partial Differential Equations (PDEs) \cite{leveque2002finite}, which can exhibit the challenging feature of emerging discontinuities, such as shock waves, at finite times 
even when the initial conditions are smooth \cite{dafermos2005hyperbolic}.

As only a few (simple) hyperbolic PDEs exhibit analytic solutions, for real-world applications approximate solutions are sought via numerical methods that focus on accurately capturing the emerging discontinuities.~For preserving the underlying conservation laws, the vast majority of numerical methods involve approximations of the numerical fluxes obtained by Finite Volume (FV) schemes \cite{leveque2002finite}, such as the Godunov and the Lax-Friedrichs schemes \cite{godunov1959finite,leveque2002finite}.~The former allows for the employment of high-resolution methods such as flux-limiter methods \cite{sweby1984high,clain2011high}, which avoid non-physical solutions and peculiar oscillatory behavior around the discontinuities when suitable flux-limiters are used \cite{leveque2002finite}.~Most importantly, these schemes are conservative, under specific conditions related to the Riemann solver, such as the \textit{Rankine-Hugoniot} (RH) condition \cite{toro2013riemann,leveque2002finite}.~Other methods include high-resolution (weighted) essential non-oscillatory (W)ENO schemes \cite{harten1997uniformly,jiang1996efficient}, 
discontinuous Galerkin \cite{lomtev1999discontinuous,girfoglio2021pod} and spectral methods \cite{wang2007high}.~For a comprehensive presentation of FV numerical methods for conservation laws, see \citep{dafermos2005hyperbolic,leveque2002finite,ketcheson2020riemann,bressan2000hyperbolic} and references therein. 

The solution of the so-called \textit{forward} problem, has been the objective of many Scientific Machine Learning (SciML) algorithms.~One broad category consists of physics-informed deep neural networks (PINNs) \citep{raissi2019physics,pang2019fpinns,yang2020physics,meng2020composite} 
which however, when employed for hyperbolic PDEs, have a rather poor performance/low accuracy near the shock waves \cite{mao2020physics,patel2022thermodynamically}.~To tackle this challenge, several PINN-based schemes have been proposed for respecting the conservation law \cite{patel2022thermodynamically,zhang2022implicit,de2024wpinns,rajvanshi2024integral}. 
Finally, other methods include hybrid/ blended methods, which combine the classical numerical methods (to guarantee conservation) with neural networks (NNs) \cite{wang2019learning,kossaczka2021enhanced,kossaczka2023deep,magiera2020constraint,pakravan2021solving,ruggeri2022neural,wang2023fluxnet}.

For the solution of the \textit{inverse} problem (that is learning functional terms/closures/boundary conditions of unknown PDEs, as well as values of the parameters from available data), SciML methods typically include estimation of the time and spatial derivatives from time-series data, which are then provided in sparse regression models \cite{schaeffer2017learning,rudy2017data,kang2021ident}, DNNs \cite{lee2023learning,galaris2022numerical} and Gaussian Process Regressors \cite{lee2020coarse,raissi2018hidden}.~However, such schemes, relying mostly on DNNs, face significant challenges when employed to find closures of hyperbolic PDEs, since they do not preserve explicitly the conservation laws, thus providing rather moderate accuracy due to the computational complexity involved in the training of DNNs.~In fact near shocks, the finite difference schemes for the estimation of the derivatives cannot accurately learn the true differential operator, and NN-based differentiation schemes may be required \cite{bar2019learning}.~In addition, even when the derivative estimation is accurate, the above methods may introduce artificial differential terms in the related closures (see e.g. inviscid Burgers' equation in \cite{schaeffer2017learning}).~Approaches, relying to physics-informed neural networks (PINNs) and their conservative flavors \cite{pang2019fpinns,yang2020physics,meng2020composite,mao2020physics,jagtap2020conservative,patel2022thermodynamically,ahmadi2024ai}. Neural Operators (NOs) such as DeepONets \cite{lu2021learning}, Fourier Neural Operators \cite{li2020fourier,kovachki2023neural,thodi2024fourier}, and RandONets \cite{fabiani2024randonet}, have been also introduced for learning the evolution operator, i.e. the right-hand side of PDEs and/or the solution operator.~However, NOs are also prone to discontinuities, and only few works are dedicated to hyperbolic PDEs \citep{wen2022u,kovachki2023neural} whle the problem of computational complexity remains.~Recently, RiemannONets \cite{peyvan2024riemannonets}, a special variation of DeepONets \cite{lu2021learning}, were introduced to approximate with high accuracy discontinuous solutions of the Riemann problem for hyperbolic equations. For a review on the solution for the inverse problem for PDEs see also in \cite{pakravan2021solving}.

While such schemes have been used for parameter estimation tasks, they do not explicitly provide conservative schemes, nor have not so far been employed, to the best of our knowledge, for learning the functional closures in PDEs of conservation laws. For learning the flux functionals/closures, hybrid/blended methods have been also developed, combining SciML algorithms with numerical methods \cite{pakravan2021solving,morand2024deep,chen2024learning,kim2024approximating}.
The above hybrid methods are especially crafted for respecting conservation laws. However, they, learn the numerical approximation of the fluxes/closures rather than the physical closures per se. This may introduce significant inaccuracy for coarsely discretized data.~For learning the physical, parameterized, flux closure, in \cite{li2023identification}, a special network architecture, resulting to rationals of polynomial functions, was coupled with the Rusanov (local Lax-Friedrichs) numerical FV scheme.~Leaving aside the numerical approximation accuracy of the selected FV scheme (which is usually of first-order), all the aforementioned hybrid methods have been developed for non-linear scalar hyperbolic PDEs, in which the discontinuities (shocks) usually travel along one characteristic.~Clearly, this is not the case for systems of hyperbolic PDEs. 

In this work, we present a hybrid/blended approach for the solution of the inverse problem for both scalar and systems of hyperbolic PDEs in 1-dim. spatial domains. In particular, we couple shallow NNs with high-resolution Godunov-type FV schemes for the solution of the Riemann problem. We name this scheme as \textit{Godunov-Riemann Informed Neural Networks (GoRINNs)}.\par
In contrast to existing hybrid methods that discover parameters or numerical fluxes \cite{morand2024deep,chen2024learning,kim2024approximating}, based on DNNs, here we use ``intelligently`` designed numerical analysis informed shallow NNs to approximate directly the physical flux functional terms, which are missing from the flux function/closure of the unknown system of PDEs.~To take into account partial information that may  be available, GoRINNs can be also employed in the presence of known physics such as known fluxes or source terms in the system of PDEs.~For obtaining the numerical solution of the unknown system, we perform operator splitting and solve the homogeneous part of the PDEs with high-resolution Godunov-type methods with flux-limiters and the inhomogeneous part (if any) with stiff integration schemes.~These numerical schemes provide high-order accuracy both far and near the emerging discontinuities.~However, since we deal with non-linear systems of hyperbolic PDEs, the related Riemann problem is solved by approximate Riemann solvers \cite{leveque2002finite}, which need to satisfy the \textit{Rankine-Hugoniot} (RH) condition in order for the numerical scheme to be conservative \cite{toro2013riemann,leveque2002finite}.~While NNs have been used to obtain RH condition-satisfying approximations for the solution of the Riemann problem via regression \cite{magiera2020constraint,ruggeri2022neural,wang2023fluxnet}, here we solve the inverse problem with unknown flux terms/closures; hence the target solutions are simply not available.~To deal with this problem, we construct Riemann solvers that are dependent on the unknown flux function/closure and constrain them to satisfy the RH condition; the latter is implemented as a soft constraint.

In this way, GoRINNs not only learn the physical unknown flux term/closures, but also the related approximate Riemann solver that renders the scheme conservative.~This enables the direct implementation of GoRINNs into software packages for the solution of hyperbolic PDEs, as Clawpack \cite{clawpack}, which require both the functional and the Riemann solver.~This is the first time that such a hybrid approach for systems of conservation laws, is presented, allowing also the inclusion of source terms.

The performance of GoRINNs is assessed via four benchmark problems of non-linear conservation laws arising in fluid dynamics and traffic flow.~First, we consider the scalar conservation laws of the inviscid Burgers' equation \cite{burgers1948mathematical} and the Lighthill-Whitham-Richards (LWR) equation \cite{lighthill1955kinematic,richards1956shock}, for which we learn the flux function and the velocity closure, respectively.~Next, we study two systems of two conservation laws, the Shallow Water (SW) and the Payne-Whitham (PW) equations \cite{payne1971model,whitham2011linear}, for both of which we learn the pressure closures in the flux function.~We highlight here, that the Burgers', LWR and SW equations are homogeneous hyperbolic PDEs that exhibit analytic solutions for the Riemann problem, while for the PW equations not only the Riemann problem cannot be solved analytically, but also source terms are included.~For all the benchmark problems, the approximate Riemann solvers are derived in \ref{app:RS4bp}.~We underline here that for all problems under study, GoRINNs due to their structure do not require extensive training data sets for achieving high approximation accuracy.~This is because, a low number of complete (in space), time observations results in a high number of residuals to be minimized for the optimization problem, even for coarse spatial discretizations.

The reminder of the paper is organized as follows.~In \cref{sec:set}, we concisely describe the problem definition.~In \cref{sec:meth}, we briefly present the high-resolution Godunov-type FV numerical scheme and introduce in detail GoRINNs for the solution of the inverse problem.~In \cref{sec:NR}, we present the numerical results for the four benchmark problems considered (which are presented in \ref{app:Bench} along with details about the forward problem solution required for the data acquisition) and assess the convergence and accuracy of GoRINNs.~Finally, in \cref{sec:Con}, we conclude, discussing limitations and perspectives of the present work.

\section{Problem Definition}
\label{sec:set}

We consider the general form of a system of $D$ conservation laws, with the possible inclusion of source terms in an open domain $\Omega$, and its boundary $\vartheta \Omega$.~The evolution of conservative quantities $\mathbf{u} = \mathbf{u}(t,x) \in \mathbb{R}^D$ can be described by a non-homogeneous system of PDEs with prescribed initial and boundary conditions, in the form:
\begin{align}
    \partial_t \mathbf{u} + \partial_x \mathbf{f}(\mathbf{u}) = \mathbf{s}(\mathbf{u}) \, \text{in}\,  \Omega, \quad \mathbf{u}(0,x)=\mathbf{u}^0(x), \quad \mathcal{B}(\mathbf{u}(t,\vartheta \Omega))=0  \, \mbox{on} \,  \vartheta \Omega, \label{eq:genPDE}
\end{align}
where $x\in \Omega$, and $\mathbf{f}(\mathbf{u}),\mathbf{s}(\mathbf{u}): \mathbb{R}^D \rightarrow \mathbb{R}^D$ are continuously differentiable, non-linear functions, denoting the flux function and the source term, respectively.~The system of PDEs in \cref{eq:genPDE} is hyperbolic iff the Jacobian matrix $\mathbf{J}(\mathbf{u}) = \partial_{\mathbf{u}} \mathbf{f}( \mathbf{u})$ is diagonalizable with real eigenvalues, for any $\mathbf{u}\in \mathbb{R}^D$.

Assuming the flux function $\mathbf{f}(\mathbf{u})$ and source term $\mathbf{s}(\mathbf{u})$ analytically available, deriving the solution of the system in \cref{eq:genPDE} is  challenging due to the frequent occurrence of discontinuities (e.g., shock waves) emerging at finite times, even when the initial conditions are smooth functions \citep{dafermos2005hyperbolic,leveque2002finite}.~As discussed in \cref{sec:Intro}, classical numerical FV schemes, novel SciML methods, such as PINNs, NOs, or hybrid methods, can be used for deriving approximate solutions of the forward problem.

In this work, we aim at the solution of  the inverse problem and effectively learn flux functionals/closures of the unknown PDE, given data of the state variables and partial information of about the system.

We consider the general form of the system in \cref{eq:genPDE} in the form: 
\begin{align}
    \partial_t \mathbf{u} + \partial_x \left(\mathbf{f}_K(\mathbf{u}) + \mathbf{f}_U(\mathbf{u}, \mathcal{N}(\mathbf{u})) \right) = \mathbf{s}(\mathbf{u}), \label{eq:unPDE}
\end{align}
where the flux term $\mathbf{f}(\mathbf{u})=\mathbf{f}_K(\mathbf{u}) + \mathbf{f}_U(\mathbf{u}, \mathcal{N}(\mathbf{u}))$ consists of two continuously differentiable non-linear functions, the known $\mathbf{f}_K(\mathbf{u}): \mathbb{R}^D \rightarrow \mathbb{R}^D$ and the unknown $\mathbf{f}_U(\mathbf{u}, \mathcal{N}(\mathbf{u})): \mathbb{R}^D \rightarrow \mathbb{R}^D$.~The unknown flux term includes a completely unknown functional of the state variables $\mathcal{N}(\mathbf{u}): \mathbb{R}^D \rightarrow \mathbb{R}^D$, but may also include known functionals of $\mathbf{u}$.~This particular form allows for the approximation of unknown flux functions for which partial (or desired) information may be available (e.g., when $\mathbf{f}_U(\mathbf{u}, \mathcal{N}(\mathbf{u})) = \mathbf{g}(\mathbf{u}) \cdot \mathcal{N}(\mathbf{u})$ with $\mathbf{g}(\mathbf{u})$ known).

As discussed in \cref{sec:Intro}, we shall approximate the unknown functional $\mathcal{N}(\mathbf{u})$ via shallow NNs, thus learning the physical/analytical flux $\mathbf{f}_U(\mathbf{u}, \mathcal{N}(\mathbf{u}))$, in contrast to existing hybrid methods that learn the numerical fluxes \cite{morand2024deep,chen2024learning,kim2024approximating}.~For learning $\mathcal{N}(\mathbf{u})$, we shall solve the forward problem for the system of the unknown system of PDEs in \cref{eq:unPDE}, the solution of which generates the available data of the state variables.

\section{Methodology: The GoRINNs}
\label{sec:meth}
We present \textit{Godunov-Riemann Informed Neural Networks} GoRINNs, a hybrid numerical analysis informed SciML scheme, for solving the inverse problem for hyperbolic systems of PDEs in the form of \cref{eq:unPDE}, particularly for learning unknown flux terms.~Our approach relies on the solution of the forward problem via high-resolution Godunov-type FV schemes, which, for the sake of completeness, we briefly describe below.

\subsection{High-resolution Godunov-type methods}
\label{sb:FP_god}

Consider the system in \cref{eq:genPDE}, where the flux and source functionals are analytically known.~For the solution of \cref{eq:genPDE}, the numerical schemes require operator splitting for the homogeneous and inhomogeneous parts (if source terms are included). Typically, the homogeneous part is solved by FV schemes and then, the inhomogeneous part is integrated by implicit or explicit schemes, depending on the nature of $\mathbf{s}(\mathbf{u})$ which can result in stiff problems.~To numerically solve the homogeneous part, we employ high-resolution Godunov-type methods with flux-limiters, which allow for high-order accuracy near discontinuities and do not introduce non-physical solutions.~More importantly, they are also conservative under conditions that will be discussed next.~In fact, these methods are used in many Computational Fluid Dynamics software packages, such as COMSOL Multiphysics\textsuperscript{\textregistered} \cite{comsol}, Clawpack \cite{clawpack}, etc.
For a comprehensive presentation of FV numerical methods for conservation laws, see \citep{dafermos2005hyperbolic,leveque2002finite,ketcheson2020riemann}.

For our illustration of employing a FV method, we will assume a 1-dim. spatial domain and its discretization in $i=1,\ldots,N$ FV cells, where the $i$-th cell occupies the space interval $[x_{i-1/2}, x_{i+1/2}]$, and assign the values of the state variables $\mathbf{u}$ in \cref{eq:genPDE} as cell volumes, which we denote by $\mathbf{Q}_i^n$ for the $i$-th cell at the $n$-th time step.~Typically, the update of each $\mathbf{Q}_i^n$ requires the solution of a \textit{Riemann problem} with a contact discontinuity between a left and a right state, say $\mathbf{q}_l\in\mathbb{R}^D$ and $\mathbf{q}_r\in\mathbb{R}^D$ \citep{riemann1860fortpflanzung}.

~Here, since we use a Godunov's scheme \citep{godunov1959finite} to solve the Riemann problem at each cell interface, the left and right states are given by adjacent cells at the $n$-th time step; i.e., $\mathbf{q}_l=\mathbf{Q}^n_{i-1}$ and $\mathbf{q}_r=\mathbf{Q}^n_i$.~Then, the solution of the Riemann problem provides a set of $M_w$ propagating waves, say $\mathbf{W}^p_{i-1/2}\in \mathbb{R}^{D}$, so that the contact discontinuity can be resolved as:
\begin{equation}
    \Delta \mathbf{Q}_{i-1/2} \equiv \mathbf{Q}^n_i-\mathbf{Q}^n_{i-1}=\sum_{p=1}^{M_w}\mathbf{W}^p_{i-1/2}.
\end{equation} 
The waves travel through the left interface $i-1/2$ of the $i$-th cell with characteristic speeds $s_{i-1/2}^p\in \mathbb{R}$, which may be negative or positive, say $(s_{i-1/2}^p)^-$ and $(s_{i-1/2}^p)^+$.~A $p$-th negative/positive speed, indicates that the $p$-th wave propagates towards the left/right.~As a result, left- and right-going fluctuations, $\mathcal{A}^- \Delta \mathbf{Q}_{i-1/2}$ and $\mathcal{A}^+ \Delta \mathbf{Q}_{i-1/2}$ respectively, are developed on the interface $i-1/2$, which are given by: 
\begin{equation}
    \mathcal{A}^- \Delta \mathbf{Q}_{i-1/2} = \sum_{p=1}^{M_w} (s_{i-1/2}^p)^- \mathbf{W}_{i-1/2}^p, \quad \mathcal{A}^+ \Delta \mathbf{Q}_{i-1/2} = \sum_{p=1}^{M_w} (s_{i-1/2}^p)^+ \mathbf{W}_{i-1/2}^p, 
    \label{eq:HRupdate1}
\end{equation}
and typically satisfy:
\begin{equation}
    \mathcal{A}^- \Delta \mathbf{Q}_{i-1/2} + \mathcal{A}^+ \Delta \mathbf{Q}_{i-1/2} = \mathbf{f}(\mathbf{Q}_i^n) - \mathbf{f}(\mathbf{Q}_{i-1}^n),
    \label{eq:sumFlux}
\end{equation}
in order for the solution to the Riemann problem to be conservative.~With the definition of the fluctuations in \cref{eq:HRupdate1}, a general high-resolution Godunov-type method provides the update of the value of the state variable $\mathbf{Q}_i^n$ as:
\begin{equation}
    \mathbf{Q}^{*,n+1}_i = \mathbf{Q}^n_i -\dfrac{\Delta t}{\Delta x} (\mathcal{A}^+ \Delta \mathbf{Q}_{i-1/2} + \mathcal{A}^- \Delta \mathbf{Q}_{i+1/2} ) - \dfrac{\Delta t}{\Delta x} (\bar{\mathbf{F}}_{i+1/2} - \bar{\mathbf{F}}_{i-1/2}),
    \label{eq:HRupdate}
\end{equation}
where the high-resolution corrections are:
\begin{equation}
    \bar{\mathbf{F}}_{i-1/2} = \dfrac{1}{2}\sum_{p=1}^{M_w} \lvert (s_{i-1/2}^p)^\pm \rvert \left( 1 - \dfrac{\Delta t}{\Delta x} \lvert (s_{i-1/2}^p)^\pm \rvert \right) \tilde{\mathbf{W}}_{i-1/2}^p.
    \label{eq:HRupdate2}
\end{equation}
The limited version of the $p$-th wave, $\tilde{\mathbf{W}}_{i-1/2}^p\in\mathbb{R}^D$,
is computed via flux-limiter functions by comparing the wave $\mathbf{W}_{i-1/2}^p$ with the waves in adjacent cells; $\mathbf{W}_{i+1/2}^p$ for negative speeds $(s_{i-1/2}^p)^-$ or $\mathbf{W}_{i-3/2}^p$ for positive ones $(s_{i-1/2}^p)^+$.~The use of flux-limiter functions (such as the minmod, superbee, Van Leer limiters) provides second-order accuracy to the numerical scheme and avoids non-physical\ peculiar oscillations near discontinuities or steep gradients \cite{leveque2002finite}.

~Note that a necessary condition for the convergence of the Godunov-type method in \cref{eq:HRupdate} is the satisfaction of the Courant-Friedrichs-Lewy (CFL) condition \citep{courant1928partiellen}, which essentially restricts the time step $\Delta t$ w.r.t. the cell length $\Delta x=x_{i+1/2}-x_{i-1/2}$ and the wave speeds, so that no wave is allowed to travel through more than one cell during the selected time step.~As is evident from \cref{eq:HRupdate}, the update for the $i$-th cell requires the values of the state variables at adjacent cells, typically the $i-2,\ldots,i+2$ cells.~For the update of the cell with $i=1,2,N-1,N$, the boundary conditions are taken into account, which are employed by augmenting the spatial domain, usually by two ghost cells per boundary; the values depend on the type of the boundary conditions, as discussed in detail in \cite{leveque2002finite}.

In the presence of source terms in \cref{eq:genPDE}, the update $\mathbf{Q}^{*,n+1}_i$ resulting from the Godunov-type method in \cref{eq:HRupdate}, is then used as initial condition for the integration of the inhomogeneous part of the PDE.~Thus, the value of the state variables at the next time step $\mathbf{Q}^{n+1}_i$ is obtained by numerical integration of $d\mathbf{u}/dt=\mathbf{s}(\mathbf{u})$ with suitable schemes depending on the nature of $\mathbf{s}(\mathbf{u})$.~A detailed discussion for the homogeneous/inhomogeneous operator splitting and the numerical integration schemes, is included on \cite{leveque2002finite}.~Naturally, $\mathbf{Q}^{n+1}_i=\mathbf{Q}^{*,n+1}_i$ in the absence of source terms.

\paragraph{Speeds and waves computation} Given the numerical scheme in \cref{eq:HRupdate}, the solution of the forward problem, requires the computation of the speeds $s_{i-1/2}^p$ and the waves $\mathbf{W}_{i-1/2}^p$ from the known flux term $\mathbf{f}(\mathbf{u})$.~The speeds and waves are provided by the solution of the Riemann problem, for which the so-called \textit{Rankine-Hugoniot} (RH) jump condition should be satisfied \cite{leveque2002finite} (determining the wave speeds and directions near jump discontinuities).~For linear systems of hyperbolic PDEs, the RH condition implies that the speeds $s_{i-1/2}^p$ and the waves $\mathbf{W}_{i-1/2}^p$ can be computed by the $p$-th eigenvalue and the corresponding right eigenvector of the constant matrix $\mathbf{A} =\partial_{\mathbf{u}} \mathbf{f}( \mathbf{u})$.~However, in a non-linear system the Riemann problem cannot be usually solved exactly, and thus approximate Riemann solvers are employed.~An obvious choice is the linearized  Riemann solver, for which it is assumed that the homogeneous system in \cref{eq:genPDE} can be locally approximated by: 
\begin{equation}
\mathbf{u}_t + \hat{\mathbf{A}}_{i-1/2} \mathbf{u}_x = \mathbf{0}, \label{eq:linRiem}
\end{equation}
where $\hat{\mathbf{A}}_{i-1/2}$ is an approximation of the Jacobian $\mathbf{J}(\mathbf{u}) = \partial_{\mathbf{u}} \mathbf{f}( \mathbf{u})$ in the neighborhood of the $(i-1)$-th and $i$-th cells.~Hence, the speeds and waves computed via the eigenvalues and the eigenvectors of the linearized matrix $\hat{\mathbf{A}}_{i-1/2}$, can be provided into the numerical scheme of \cref{eq:HRupdate}.~Note that $\hat{\mathbf{A}}_{i-1/2} \equiv \hat{\mathbf{A}}_{i-1/2}(\mathbf{q}_l,\mathbf{q}_r)$ and usually, as for the problems considered here, we are able to derive closed-form expressions for any $\mathbf{q}_l$ and $\mathbf{q}_r$ of the Riemann problem, hence accelerating computations.

In order for the linearized Riemann solver $\hat{\mathbf{A}}_{i-1/2}$ to guarantee that the numerical scheme in \cref{eq:HRupdate} is conservative, and that the approximate Riemann solver is consistent with the exact one, the matrix $\hat{\mathbf{A}}_{i-1/2}(\mathbf{q}_l,\mathbf{q}_r)$ should, for every $\mathbf{q}_l,\mathbf{q}_r$, satisfy the following properties:
\begin{enumerate}
    \item $\lim_{\mathbf{q}_l, \mathbf{q}_r\to\hat{\mathbf{q}}} \hat{\mathbf{A}}_{i-1/2}(\mathbf{q}_l,\mathbf{q}_r)  = \mathbf{J}(\hat{\mathbf{q}}) = \partial_{\mathbf{u}} \mathbf{f}( \hat{\mathbf{q}})$, so that the solution of the linearized system in \cref{eq:linRiem} with the numerical scheme in \cref{eq:HRupdate} is consistent with the solution of the non-linear conservation laws in \cref{eq:genPDE},
    \item $\hat{\mathbf{A}}_{i-1/2}(\mathbf{q}_l,\mathbf{q}_r)$ diagonalizable with real eigenvalues, so that the linearized system in \cref{eq:linRiem} is hyperbolic, and,
    \item $\hat{\mathbf{A}}_{i-1/2}(\mathbf{q}_l,\mathbf{q}_r)\cdot (\mathbf{q}_r-\mathbf{q}_l) = \mathbf{f}(\mathbf{q}_r) -\mathbf{f}(\mathbf{q}_l)$ (RH condition for systems), so that the solution of the linearized system in \cref{eq:linRiem} with the numerical scheme in \cref{eq:HRupdate} is conservative.
\end{enumerate}
The first two properties are essential for any approximate Riemann solver, while the third one is characteristic for a Roe solver \citep{roe1981approximate}; it ensures that speeds calculated from the RH condition near contact discontinuities are in fact eigenvalues of the matrix $\hat{\mathbf{A}}_{i-1/2}$ and satisfy \cref{eq:sumFlux}.

The derivation of the Roe linearized matrix follows the general approach, introduced in \cite{roe1981approximate}, which includes a non-linear transformation of the state variables for a closed-form computation of the arising integrals.~Alternatively, a ``reverse engineering'' approach \cite{ketcheson2020riemann} can be followed, in which one sets $\hat{\mathbf{A}}_{i-1/2}(\mathbf{q}_l,\mathbf{q}_r)=\partial_{\mathbf{u}} \mathbf{f}(\mathbf{\bar{q}})$ and seeks to find some \textit{average} state $\mathbf{\bar{q}}$ as a function of $\mathbf{q}_r$ and $\mathbf{q}_l$, such that the, related to RH condition, third property is satisfied.~These techniques usually lead to same Roe matrices for many systems, albeit they are not always applicable; e.g., the second technique is not applicable for the fourth benchmark problem of PW equations (see \Cref{app:RoePayne}).~This is because the flux function $\mathbf{f}(\mathbf{u})$ may not allow either the closed-form computation of the integrals involved in the first technique, or finding the average state $\mathbf{\bar{q}}$ involved in the second one; both techniques are presented in detail in \Cref{app:linRiem}.

Finally, we note that while in many cases the Roe matrix $\hat{\mathbf{A}}_{i-1/2}$ is successfully used (with entropy fixes) for solving the forward problem, such kind of linearization may fail.~In these cases, other approximate Riemann solvers should be used, e.g. the HLLE solver \citep{harten1983upstream,einfeldt1988godunov}, which effectively reduces to the Roe one when dealing with shock waves and does not require entropy fix when dealing with rarefactions.~The disadvantage of the HLLE solver is that it is modelled by only two waves (i.e., $M_w=2$ instead of $M_w=D$ in \cref{eq:HRupdate2}), which may affect the resolution of systems with more than two variables.~A detailed presentation of the HLLE approximate Riemann solvers is also provided in \Cref{app:linRiem}.

\subsection{Solution of the inverse problem with GoRINNs}
\label{sb:InvGNN}
Here, we present \textit{GoRINNs} for the solution of the inverse problem for scalar or systems of conservation laws described by hyperbolic PDEs.~As already discussed, this is a hybrid approach combining the use of NNs with the high-resolution Godunov-type numerical methods presented in \Cref{sb:FP_god}.

Let us first revisit the system of hyperbolic PDEs in \cref{eq:unPDE} with the unknown flux functionals term $\mathbf{f}_U(\mathbf{u}, \mathcal{N}(\mathbf{u}))$.~For the approximation of the unknown functional, we consider a shallow NN, $\mathcal{N}(\mathbf{u}):=\mathcal{N}(\mathbf{u},\mathbf{P};\mathbf{H}_1):\mathbb{R}^D\rightarrow\mathbb{R}^D$, with network parameters $\mathbf{P}$ (weights and biases of each layer) and hyper-parameters $\mathbf{H}_1$ (e.g., activation function, learning rate, the number of epochs, etc.).~To solve the inverse problem with GoRINNs, we seek to find the parameters $\mathbf{P}$ of the NN $\mathcal{N}(\mathbf{u},\mathbf{P};\mathbf{H}_1)$, such that the employment of the high-resolution Godunov-type method of \cref{eq:HRupdate} to the system in \cref{eq:unPDE} generates the available data for the values of the state variables.

\paragraph{GoRINNs as conservative schemes} Recall that in order to employ a Godunov-type method, it is required to numerically compute the speeds and waves through an approximate Riemann solver of the linearized system in \cref{eq:linRiem}.~For guaranteeing a conservative and consistent with the non-linear system numerical scheme, we derive a Roe linearized matrix $\hat{\mathbf{A}}_{i-1/2}(\mathbf{q}_l,\mathbf{q}_r)$ for the unknown system of PDEs in \cref{eq:unPDE}, which is required to satisfy the three properties discussed in \Cref{sb:FP_god}.~For the derivation of the Roe matrix, the general approach of \cite{leveque2002finite} cannot be employed, since the integrals of the unknown flux term at transformed variables cannot be computed in a closed form.~Hence, we follow a ``reverse engineering'' approach as proposed in \cite{ketcheson2020riemann}, which results in the following proposition:
\begin{proposition}
Consider the homogenous system of non-linear PDEs of \cref{eq:unPDE} with the partially known flux function $\mathbf{f}(\mathbf{u})$.~Let $\mathbf{q}_l, \mathbf{q}_r \in \mathbb{R}^D$ be an arbitrary pair of left and right states of the related local Riemann problem and $\mathbf{\bar{q}}=\mathbf{h}(\mathbf{q}_l,\mathbf{q}_r)$ be an average state of them, such that $\mathbf{\hat{q}}=\mathbf{h}(\mathbf{\hat{q}},\mathbf{\hat{q}})$ for $\mathbf{\hat{q}}\in \mathbb{R}^D$.~Then, if the conditions:
\begin{align}
\partial_{\mathbf{u}} \mathbf{f}_K(\bar{\mathbf{q}}) \cdot (\mathbf{q}_r-\mathbf{q}_l) & = \mathbf{f}_K(\mathbf{q}_r) -\mathbf{f}_K(\mathbf{q}_l), \label{eq:Roe_NN1} \\
\left( \partial_{\mathbf{u}} \mathbf{f}_U(\bar{\mathbf{q}},\mathcal{N}(\bar{\mathbf{q}})) + \partial_{\mathcal{N}} \mathbf{f}_U(\bar{\mathbf{q}},\mathcal{N}(\bar{\mathbf{q}})) \cdot \partial_{\mathbf{u}}\mathcal{N}(\bar{\mathbf{q}}) \right) \cdot (\mathbf{q}_r-\mathbf{q}_l) & =  \mathbf{f}_U(\mathbf{q}_r,\mathcal{N}(\mathbf{q}_r)) -\mathbf{f}(\mathbf{q}_l,\mathcal{N}(\mathbf{q}_l)), \label{eq:Roe_NN2}
\end{align}
hold, the approximate Riemann solver defined by the Roe matrix $\hat{\mathbf{A}}_{i-1/2}(\mathbf{q}_l,\mathbf{q}_r)=\partial_{\mathbf{u}} \mathbf{f}(\bar{\mathbf{q}})$ satisfies the RH condition, and thus ensures that the Godunov-type numerical method in \cref{eq:HRupdate} is conservative; the speeds and waves are computed from $\hat{\mathbf{A}}_{i-1/2}(\mathbf{q}_l,\mathbf{q}_r)$. 
\label{prop1}
\end{proposition}
\begin{proof}
Using the definition of the known and unknown functions $\mathbf{f}(\mathbf{u})=\mathbf{f}_K(\mathbf{u}) + \mathbf{f}_U(\mathbf{u}, \mathcal{N}(\mathbf{u}))$ included in the non-linear system of PDEs in \cref{eq:unPDE}, the linearized matrix $\hat{\mathbf{A}}_{i-1/2}(\mathbf{q}_l,\mathbf{q}_r)=\partial_{\mathbf{u}} \mathbf{f}(\bar{\mathbf{q}})$ is written as:
\begin{equation}
    \hat{\mathbf{A}}_{i-1/2}(\mathbf{q}_l,\mathbf{q}_r) = \partial_{\mathbf{u}} \left(\mathbf{f}_K(\bar{\mathbf{q}}) + \mathbf{f}_U(\bar{\mathbf{q}},\mathcal{N}(\bar{\mathbf{q}})) \right) = \partial_{\mathbf{u}} \mathbf{f}_K(\bar{\mathbf{q}}) + \partial_{\mathbf{u}} \mathbf{f}_U(\bar{\mathbf{q}},\mathcal{N}(\bar{\mathbf{q}})) + \partial_{\mathcal{N}} \mathbf{f}_U(\bar{\mathbf{q}},\mathcal{N}(\bar{\mathbf{q}})) \cdot \partial_{\mathbf{u}}\mathcal{N}(\bar{\mathbf{q}}),
    \label{eq:RoeGNN}
\end{equation}
where the chain rule was employed at the last step.~According to \citep{roe1981approximate}, the sufficient condition for rendering the Godunov-type numerical method in \cref{eq:HRupdate} conservative is the RH condition: 
\begin{equation}
    \hat{\mathbf{A}}_{i-1/2}(\mathbf{q}_l,\mathbf{q}_r) \cdot (\mathbf{q}_r-\mathbf{q}_l) = \mathbf{f}(\mathbf{q}_r) -\mathbf{f}(\mathbf{q}_l),
\end{equation}
which, upon substitution of \cref{eq:RoeGNN} yields:
\begin{equation*}
    \left(\partial_{\mathbf{u}} \mathbf{f}_K(\bar{\mathbf{q}}) + \partial_{\mathbf{u}} \mathbf{f}_U(\bar{\mathbf{q}},\mathcal{N}(\bar{\mathbf{q}})) + \partial_{\mathcal{N}} \mathbf{f}_U(\bar{\mathbf{q}},\mathcal{N}(\bar{\mathbf{q}})) \cdot \partial_{\mathbf{u}}\mathcal{N}(\bar{\mathbf{q}})\right) \cdot (\mathbf{q}_r-\mathbf{q}_l) = \mathbf{f}_K(\mathbf{q}_r) -\mathbf{f}_K(\mathbf{q}_l) + \mathbf{f}_U(\mathbf{q}_r,\mathcal{N}(\mathbf{q}_r)) -\mathbf{f}(\mathbf{q}_l,\mathcal{N}(\mathbf{q}_l)).
\end{equation*}
The latter equation should hold for any set of parameters included in $\mathbf{f}(\mathbf{u})$, both for known and unknown terms (e.g., the network parameters found in $\mathcal{N}(\mathbf{u})$).~Its decomposition into known and unknown terms results in the sufficient for a conservative numerical method conditions of \cref{eq:Roe_NN1,eq:Roe_NN2}.
\end{proof}

The construction of the Roe matrix $\hat{\mathbf{A}}_{i-1/2}$ with the conditions of \Cref{prop1} not only guarantees a conservative numerical scheme, but also satisfies by definition the first property regarding consistency in \Cref{sb:FP_god}, since $\hat{\mathbf{A}}_{i-1/2}(\mathbf{q}_l,\mathbf{q}_r) \rightarrow \partial_{\mathbf{u}} \mathbf{f}( \hat{\mathbf{q}})$ as $\mathbf{q}_l, \mathbf{q}_r \rightarrow \hat{\mathbf{q}}$.~However, for the construction of $\hat{\mathbf{A}}_{i-1/2}$, one needs to determine $\mathbf{\bar{q}}=\mathbf{h}(\mathbf{q}_l,\mathbf{q}_r)$ through the conditions in \cref{eq:Roe_NN1,eq:Roe_NN2}.~In fact, the latter includes both the unknown flux term and the derivatives of the NN w.r.t. the state variables.~Hence, $\bar{\mathbf{q}}$ can be partially determined as a function of $\mathbf{q}_l$ and $\mathbf{q}_r$ for only some of its components by solving \cref{eq:Roe_NN1}.~For the rest of the components, we simply assume that they can be expressed by the arithmetic average $\bar{\mathbf{q}}=(\mathbf{q}_r+\mathbf{q}_l)/2$ and require \cref{eq:Roe_NN2} to hold for the assumed $\bar{\mathbf{q}}$.

Let's now assume a set of data points $\mathbf{Q}_i^n$ at the $n$-th time step, discretized all over the 1-dim. spatial domain $[x_l,x_r]$ in $i=1,\ldots,N$ FV cells.~Setting $\mathbf{q}_l=\mathbf{Q}^n_{i-1}$ and $\mathbf{q}_r=\mathbf{Q}^n_i$, one can estimate the averages $\bar{\mathbf{Q}}_i^n$ for every $i=1,\ldots,N$ by (i) computing the components of $\bar{\mathbf{q}}$ that can be determined from \cref{eq:Roe_NN1} and (ii) making the assumption of the arithmetic average for the remaining components, as $\bar{\mathbf{Q}}_i^n=(\mathbf{Q}^n_{i-1}+\mathbf{Q}^n_i)/2$.~Then, according to \Cref{prop1}, the linearized Riemann solver of the unknown system of PDEs in \cref{eq:unPDE} is conservative if the condition in \cref{eq:Roe_NN2} holds for every FV cell, that is:
\begin{equation}
    \left( \partial_{\mathbf{u}} \mathbf{f}_U(\bar{\mathbf{Q}}_i^n,\mathcal{N}(\bar{\mathbf{Q}}_i^n)) + \partial_{\mathcal{N}} \mathbf{f}_U(\bar{\mathbf{Q}}_i^n,\mathcal{N}(\bar{\mathbf{Q}}_i^n)) \cdot \partial_{\mathbf{u}}\mathcal{N}(\bar{\mathbf{Q}}_i^n) \right) \cdot (\mathbf{Q}^n_i-\mathbf{Q}^n_{i-1}) =  \mathbf{f}_U(\mathbf{Q}^n_i,\mathcal{N}(\mathbf{Q}^n_i)) -\mathbf{f}_U(\mathbf{Q}^n_{i-1},\mathcal{N}(\mathbf{Q}^n_{i-1})),
    \label{eq:RH_NN}
\end{equation}
where the NN function is computed at $\bar{\mathbf{Q}}_i^n$, $\mathbf{Q}^n_i$ and $\mathbf{Q}^n_{i-1}$.~Finding the networks parameters that satisfy  \cref{eq:RH_NN} ensures that the proposed GoRINNs  are conservative.

\paragraph{High-resolution GoRINNs} Given the linearized Riemann solver $\hat{\mathbf{A}}_{i-1/2}(\mathbf{q}_l,\mathbf{q}_r) = \partial_{\mathbf{u}} \mathbf{f}(\bar{\mathbf{q}})$, the speeds and waves for the Godunov method can be estimated; for a Roe solver these are given by the eigenvalues and eigenvectors of $\hat{\mathbf{A}}_{i-1/2}$, while for an HLLE solver, one needs to compute a middle state of the Riemann problem as shown in \Cref{app:HLLELin}.~In particular, for the high-resolution scheme in \cref{eq:HRupdate}, the update of $i$-th cell requires (i) the speeds $s_{i\pm1/2}^p$ and waves $\mathbf{W}^p_{i\pm1/2}$, which are functions of $\mathbf{Q}^n_{i-1}$, $\mathbf{Q}^n_i$ and $\mathbf{Q}^n_{i+1}$ and (ii) the flux-limited version of the waves, $\tilde{\mathbf{W}}^p_{i-1/2}$ and $\tilde{\mathbf{W}}^p_{i+1/2}$, which are functions of $\mathbf{Q}^n_{i+a}$ for $a=-2,\ldots,2$, since $\tilde{\mathbf{W}}^p_{i-1/2}$ is a function of $\mathbf{W}^p_{i+1/2}$ and $\mathbf{W}^p_{i-3/2}$, and $\tilde{\mathbf{W}}^p_{i+1/2}$ is a function of $\mathbf{W}^p_{i+3/2}$ and $\mathbf{W}^p_{i-1/2}$.~Hence, the FV high-resolution Godunov-type method in \cref{eq:HRupdate} can be summarized as an update $\hat{\mathbf{Q}}_i^{n+1}$ of the form:
\begin{equation}
    \hat{\mathbf{Q}}_i^{n+1} = FVG_{HR}\left(\mathbf{Q}^n_{i-2},\mathbf{Q}^n_{i-2},\mathbf{Q}^n_{i},\mathbf{Q}^n_{i+1},\mathbf{Q}^n_{i+2},\mathbf{P};\mathbf{H}_1,\mathbf{H}_2 \right),
    \label{eq:FVGupdate}
\end{equation}
which depends on the parameters $\mathbf{P}$ and hyperparameters $\mathbf{H}_1$ of the ANN (since the speeds and waves are computed via $\mathcal{N}(\mathbf{u},\mathbf{P};\mathbf{H}_1)$), but also on the  hyperparameters $\mathbf{H}_2$ of the Godunov scheme, such as the flux-limited function, the type of solver (Roe, HLLE or others), etc.~Note that the numerical scheme of \cref{eq:FVGupdate} requires information from the boundary conditions to update the $i=1,2,N-1,N$ nodes.~As already discussed, two ghost cells are added to the left and right boundaries, the values of which are determined by the type of boundary condition; here, we assume that such information is available.~Finally, note that the numerical scheme of \cref{eq:FVGupdate} also incorporates the numerical time integration scheme of the inhomogeneous splitted part, in cases where the system in \cref{eq:unPDE} includes source terms.~The pseudocode for performing the update of \cref{eq:FVGupdate} is presented in detailed in \Cref{algFVG}.
  
\paragraph{GoRINNs optimization problem} Let's now assume the availability of $n_t$ pairs of data points $\mathbf{Q}_i^n, \mathbf{Q}_i^{n+1}$ of the solution operator, over the entire 1-dim. spatial domain $[x_l,x_r]$; i.e., for $i=1,\ldots,N$.~Then, the solution of the inverse problem is provided by the functional $\mathbf{f}_U(\mathbf{u},\mathcal{N}(\mathbf{u}))$, for which the employment of the high-resolution Godunov-type update at $\mathbf{Q}_i^n$ approximates $\mathbf{Q}_i^{n+1}$, under a conservative linearized Riemann solver.~Hence, the numerical approximation of the unknown functional is provided by the parameters $\mathbf{P}$ of the NN $\mathcal{N}(\mathbf{u},\mathbf{P};\mathbf{H}_1)$, computed by solving the optimization problem:
\begin{equation}
    \min_{\mathbf{P}} \mathcal{L} (\mathbf{P}; \mathbf{H}_1,\mathbf{H}_2) := \sum_{n=1}^{n_t} \left( \lVert \mathcal{L}_{FVG}\left(\mathbf{Q}^{n+1},\mathbf{Q}^n, \mathbf{P};\mathbf{H}_1,\mathbf{H}_2 \right) \rVert^2 + \lVert  \mathcal{L}_{RH}\left(\mathbf{Q}^n,\mathbf{P};\mathbf{H}_1 \right) \rVert^2 \right),
    \label{eq:OptGen}
\end{equation}
where $\mathcal{L}_{FVG}(\cdot)$ is the data loss function of the FV high-resolution Godunov update in \cref{eq:FVGupdate}:
\begin{equation}
    \mathcal{L}_{FVG}\left(\mathbf{Q}^{n+1},\mathbf{Q}^n, \mathbf{P};\mathbf{H}_1,\mathbf{H}_2 \right) = \sum_{i=1}^N \left( \mathbf{Q}^{n+1}_i-FVG_{HR}\left(\mathbf{Q}^n_{i-2},\mathbf{Q}^n_{i-2},\mathbf{Q}^n_{i},\mathbf{Q}^n_{i+1},\mathbf{Q}^n_{i+2},\mathbf{P};\mathbf{H}_1,\mathbf{H}_2 \right)\right),
    \label{eq:Opt1}
\end{equation}
and $\mathcal{L}_{RH}(\cdot)$ is the loss function soft constraining the satisfaction of the RH condition in \cref{eq:RH_NN}, as:
\begin{multline}
    \mathcal{L}_{RH}\left(\mathbf{Q}^n, \mathbf{P};\mathbf{H}_1 \right) = \sum_{i=1}^N \Big( \left( \partial_{\mathbf{u}} \mathbf{f}_U(\bar{\mathbf{Q}}_i^n,\mathcal{N}(\bar{\mathbf{Q}}_i^n,\mathbf{P};\mathbf{H}_1)) + \partial_{\mathcal{N}} \mathbf{f}_U(\bar{\mathbf{Q}}_i^n,\mathcal{N}(\bar{\mathbf{Q}}_i^n,\mathbf{P};\mathbf{H}_1)) \cdot \partial_{\mathbf{u}}\mathcal{N}(\bar{\mathbf{Q}}_i^n,\mathbf{P};\mathbf{H}_1) \right) \cdot \\ \cdot (\mathbf{Q}^n_i-\mathbf{Q}^n_{i-1}) - \left(  \mathbf{f}_U(\mathbf{Q}^n_i,\mathcal{N}(\mathbf{Q}^n_i,\mathbf{P};\mathbf{H}_1)) -\mathbf{f}_U(\mathbf{Q}^n_{i-1},\mathcal{N}(\mathbf{Q}^n_{i-1},\mathbf{P};\mathbf{H}_1)) \right) \Big),
    \label{eq:Opt2}
\end{multline}
for ensuring a conservative scheme.~Here $\mathbf{Q}^n$ denotes the values of state variables ath the $n$-th time step for all $i=1,\ldots,N$ cells.~Note that for minimizing the loss function in \cref{eq:Opt2}, it is first required to compute the average states $\bar{\mathbf{Q}}_i^n$ from the data, as discussed after \cref{eq:Roe_NN2}.

\begin{algorithm}
\caption{GoRINNs 
implementation} \label{algFVG}
\begin{algorithmic}[1]
\Require Flux term $\mathbf{f}_K(\mathbf{u}) + \mathbf{f}_U(\mathbf{u}, \mathcal{N}(\mathbf{u},\mathbf{P};\mathbf{H}_1))$ and source term $\mathbf{s}(\mathbf{u})$ \Comment{Known and unknown terms in \cref{eq:unPDE}}
\Require State variables $\mathbf{Q}^n_{i-2},\mathbf{Q}^n_{i-1},\mathbf{Q}^n_{i},\mathbf{Q}^n_{i+1},\mathbf{Q}^n_{i+2}$, NN parameters $\mathbf{P}$ and hyperparameters $\mathbf{H}_1$, $\mathbf{H}_2$
\State Get $D$, $N$, $\Delta t$, $\Delta x$, flux-limiter function $\phi(\theta)$, and type of boundary conditions included in $\mathbf{H}_2$
\If{$i=1,2,N-1,N$}
\State Set values at required ghost cells $\mathbf{Q}^n_{-1},\mathbf{Q}^n_{0},\mathbf{Q}^n_{N+1},\mathbf{Q}^n_{N+2}$ according to boundary conditions 
\EndIf
\LComment{Get numerical fluxes and high-resolution corrections at left and right interfaces $i-1/2$ and $i+1/2$ of the $i$-th cell}
\State Get left $\mathcal{A}^+ \Delta \mathbf{Q}_{i-1/2},\bar{\mathbf{F}}_{i-1/2}$ $\gets$\textproc{NumerFlux$\&$HRCor}$(\mathbf{Q}^n_{i-2},\mathbf{Q}^n_{i-1},\mathbf{Q}^n_{i},\mathbf{Q}^n_{i+1},\mathbf{P},\mathbf{H}_1,\mathbf{H}_2)$
\State Get right $\mathcal{A}^- \Delta \mathbf{Q}_{i+1/2},\bar{\mathbf{F}}_{i+1/2}$ $\gets$\textproc{NumerFlux$\&$HRCor}$(\mathbf{Q}^n_{i-1},\mathbf{Q}^n_{i},\mathbf{Q}^n_{i+1},\mathbf{Q}^n_{i+2},\mathbf{P},\mathbf{H}_1,\mathbf{H}_2)$
\State Get update $\mathbf{Q}^{*,n+1}_i$ \Comment{Perform update in \cref{eq:HRupdate}}
\If{$\mathbf{s}(\mathbf{u})\neq 0$} 
\State Integrate $d\mathbf{u}/dt=\mathbf{s}(\mathbf{u})$ with initial condition $\mathbf{Q}^{*,n+1}_i$ for $\Delta t$ 
\State Set $\hat{\mathbf{Q}}^{n+1}_i \gets$ numerical solution \Comment{Update in \cref{eq:FVGupdate}}
\Else
\State Set $\hat{\mathbf{Q}}^{n+1}_i\gets \mathbf{Q}^{*,n+1}_i$ \Comment{Update in \cref{eq:FVGupdate}}
\EndIf
\end{algorithmic}
\begin{algorithmic}[0]
\State \textcolor{gray}{$\triangleright$~\textit{Procedure to compute the numerical fluxes and high-resolution correction at $i-1/2$ interface, for given variable states $\mathbf{Q}^n_{i-2},\mathbf{Q}^n_{i-1},\mathbf{Q}^n_{i},\mathbf{Q}^n_{i+1}$ at $i-2$, $i-1$, $i$ and $i+1$ cells, NN parameters $\mathbf{P}$ and hyperparameters $\mathbf{H}_1$, $\mathbf{H}_2$}}
\end{algorithmic}
\begin{algorithmic}[1]
\Procedure{NumerFlux$\&$HRCor}{$\mathbf{Q}^n_{i-2},\mathbf{Q}^n_{i-1},\mathbf{Q}^n_{i},\mathbf{Q}^n_{i+1},\mathbf{P},\mathbf{H}_1,\mathbf{H}_2$}
\State Set left and right states $\mathbf{q}_l\gets \mathbf{Q}^n_{i-1}$ and $\mathbf{q}_r\gets \mathbf{Q}^n_{i}$ \Comment{Numerical fluxes at $i-1/2$ interface}
\State Get speeds and waves $s^p_{i-1/2}$, $\mathbf{W}^p_{i-1/2} \gets$\textproc{ApproxRiemannSolver}$(\mathbf{q}_l,\mathbf{q}_r,\mathbf{P},\mathbf{H}_1,\mathbf{H}_2)$
\State Compute numerical fluxes $\mathcal{A}^- \Delta \mathbf{Q}_{i-1/2}$ and $\mathcal{A}^+ \Delta \mathbf{Q}_{i-1/2}$ \Comment{By \cref{eq:HRupdate1}}
\For{$p=1,\ldots,M_w$} \Comment{High-resolution correction with flux-limiters}
\If{$s^p_{i-1/2}<0$} \Comment{Decide adjacent $k$-th cell}
\State Set $k \gets i$, $k_l \gets k$ and $k_r \gets k+1$
\ElsIf{$s^p_{i-1/2}>0$}
\State Set $k \gets i-1$, $k_l \gets k-1$ and $k_r \gets k$
\EndIf
\State Set left and right states $\mathbf{q}_l\gets \mathbf{Q}^n_{k_l}$ and $\mathbf{q}_r\gets \mathbf{Q}^n_{k_r}$ \Comment{Waves at $k-1/2$ interface}
\State Get waves $\mathbf{W}^p_{k-1/2} \gets$\textproc{ApproxRiemannSolver}$(\mathbf{q}_l,\mathbf{q}_r,\mathbf{P},\mathbf{H}_1,\mathbf{H}_2)$
\State Compute angle of waves $\theta^p \gets \left(\big(\mathbf{W}^{p}_{k-1/2}\big)^\top \cdot \mathbf{W}^p_{i-1/2}\right)/\left(\big(\mathbf{W}^{p}_{i-1/2}\big)^\top\cdot \mathbf{W}^p_{i-1/2}\right)$.
\State Compute wave-limited version $\tilde{\mathbf{W}}_{i-1/2}^p\gets \phi(\theta^p) \mathbf{W}^p_{i-1/2}$ via the flux-limiter function $\phi(\theta)$
\EndFor
\State Compute high-resolution correction $\bar{\mathbf{F}}_{i-1/2}$ \Comment{By \cref{eq:HRupdate2}}
\State \textbf{return} $\mathcal{A}^- \Delta \mathbf{Q}_{i-1/2}$, $\mathcal{A}^+ \Delta \mathbf{Q}_{i-1/2}$ and $\bar{\mathbf{F}}_{i-1/2}$ for $p=1,\ldots,M_w$
\EndProcedure
\end{algorithmic}
\begin{algorithmic}[0]
\State \textcolor{gray}{$\triangleright$~\textit{Procedure to compute the speeds and the waves provided by the approximate Riemann solver of the GoRINNs, for given left and right states $\mathbf{q}_l,\mathbf{q}_r$}, NN parameters $\mathbf{P}$ and hyperparameters $\mathbf{H}_1$, $\mathbf{H}_2$}
\end{algorithmic}
\begin{algorithmic}[1]
\Procedure{ApproxRiemannSolver}{$\mathbf{q}_l,\mathbf{q}_r,\mathbf{P},\mathbf{H}_1,\mathbf{H}_2$}
\State Compute $\bar{\mathbf{q}}$ from $\mathbf{q}_l,\mathbf{q}_r$  \Comment{Partially by \cref{eq:Roe_NN1} and partially by $(\mathbf{q}_l+\mathbf{q}_r)/2$}
\State Compute $\partial_{\mathbf{u}}\mathcal{N}(\bar{\mathbf{q}},\mathbf{P};\mathbf{H}_1)$
\State Form linearized Roe matrix $\hat{\mathbf{A}}_{i-1/2}(\mathbf{q}_l,\mathbf{q}_r)$ \Comment{By \cref{eq:RoeGNN}}
\State Set f$\_$sol $\gets$ Roe or HLLE from $\mathbf{H}_2$ \Comment{Approximate Riemann solver type: Roe or HLLE}
\If{f$\_$sol=ROE}
\State Compute eigenvalues $\hat{\lambda}^p_{i-1/2}$ and eigenvectors $\hat{\mathbf{r}}^p_{i-1/2}$ of $\hat{\mathbf{A}}_{i-1/2}$ \Comment{for $p=1,\ldots,D$}
\State Set speeds $s^p_{i-1/2} \gets \hat{\lambda}^p_{i-1/2}$ and waves $\mathbf{W}^p_{i-1/2} \gets \hat{\mathbf{r}}^p_{i-1/2}$
\ElsIf{f$\_$sol=HLLE}
\State Compute $\mathcal{N}(\mathbf{q}_l)=\mathcal{N}(\mathbf{q}_l,\mathbf{P};\mathbf{H}_1)$ and $\mathcal{N}(\mathbf{q}_r)=\mathcal{N}(\mathbf{q}_r,\mathbf{P};\mathbf{H}_1)$ and derivatives w.r.t. $\mathbf{u}$
\State Form Jacobians $\mathbf{J}(\mathbf{q}_l)=\partial_\mathbf{u} \left(\mathbf{f}_K(\mathbf{q}_l) + \mathbf{f}_U(\mathbf{q}_l, \mathcal{N}(\mathbf{q}_l)) \right)$ and $\mathbf{J}(\mathbf{q}_r)=\partial_\mathbf{u} \left(\mathbf{f}_K(\mathbf{q}_r) + \mathbf{f}_U(\mathbf{q}_r, \mathcal{N}(\mathbf{q}_r)) \right)$
\State Compute eigenvalues $\hat{\lambda}^p_{i-1/2}$, $\lambda^p_{i-1}$ and $\lambda^p_i$ of $\hat{\mathbf{A}}_{i-1/2}$, $\mathbf{J}(\mathbf{q}_l)$ and $\mathbf{J}(\mathbf{q}_r)$
\State Compute speeds $s^p_{i-1/2}$ for $p=1,2$ \Comment{By \cref{eq:HLLEs}}
\State Compute middle state $\mathbf{q}_m$ using $\mathcal{N}(\mathbf{q}_l)$ and $\mathcal{N}(\mathbf{q}_r)$ \Comment{In \cref{eq:HLLEms}}
\State Compute waves $\mathbf{W}^p_{i-1/2}$ \Comment{By \cref{eq:HLLEwaves}}
\EndIf
\State \textbf{return} $s^p_{i-1/2}$ and $\mathbf{W}^p_{i-1/2}$ for $p=1,\ldots,M_w$
\EndProcedure
\end{algorithmic}
\end{algorithm}

\paragraph{GoRINNs implementation} For the solution of the inverse problem with the GoRINNs, we employed a single-layer feedforward NN architecture with $D$-dim. inputs and 1-dim. output, since for all benchmark problems examined, a scalar flux function term was assumed unknown.~For $L$ neurons in the hidden layer, the output of the NN is written as:
\begin{equation}
    \mathcal{N}(\mathbf{u},\mathbf{P};\mathbf{H}_1) 
    = \mathbf{w}^{o\top} \boldsymbol{\phi} \left( \mathbf{W} \mathbf{u}  + \mathbf{b} \right) ,
    \label{eq:SFLNNsumf}
\end{equation}
where the network parameters $\mathbf{P} = [\mathbf{w}^{o},\mathbf{W},\mathbf{b}]^\top\in\mathbb{R}^{L(D+2)}$ include (i) the output weights $\mathbf{w}^{o} \in \mathbb{R}^{L}$ of the neurons between the output and the hidden layer, (ii) the internal weights $\mathbf{W}\in\mathbb{R}^{L\times D}$ between the hidden and the input layer, and (iii) the internal biases  $\mathbf{b} \in\mathbb{R}^{L}$ of the neurons in the hidden layer.~Output biases were not included, since they are not contributing in neither of the loss functions in \cref{eq:Opt1,eq:Opt2}; the derivative of the network w.r.t. $\mathbf{u}$ appears in the former, and the difference between two NNs outputs in the latter.~The outputs of the activated neurons are included in the column vector $\boldsymbol{\phi}\left( \mathbf{W} \mathbf{u}  + \mathbf{b} \right)\in\mathbb{R}^L$, for which the logistic sigmoid function is used as activation function.~As already discussed, all hyperparameters are included in $\mathbf{H}_1$.

For the solution of the GoRINNs optimization problem in \cref{eq:OptGen}, the two loss functions in \cref{eq:Opt1,eq:Opt2} were expressed in an $(D\times N\times n_t)$-dim. and an $(N\times  n_t)$-dim. vector of non-linear residuals, say $\mathcal{F}_{FVG}(\mathbf{P})$ and $\mathcal{F}_{RH}(\mathbf{P})$, respectively.~To minimize these residuals w.r.t. the unknown parameters $\mathbf{P}$, we used the Levenberg-Marquardt (LM) algorithm \cite{hagan1994training}.~LM is a deterministic gradient-based optimization algorithm, which updates the unknown parameters according to a varying - at each iteration - damping factor $\lambda$.~It thus, requires at each iteration the derivatives $\partial_{\mathbf{P}} \mathcal{F}_{FVG}$ and $\partial_{\mathbf{P}} \mathcal{F}_{RH}$.~Since symbolic differentiation of $\partial_{\mathbf{P}} \mathcal{F}_{FVG}$ is not feasible, we used numerical differentiation through finite differences.~Note that the residuals corresponding to \cref{eq:Opt1} can be computed independently for each $n=1,\ldots,n_t$, a feature that allowed us to perform parallel computations to form $\mathcal{F}_{FVG}(\mathbf{P})$ and thus accelerate convergence of the LM algorithm.

We highlight here an important feature of GoRINNs, related to the optimization problem and the LM algorithm employed for its solution.~As discussed, GoRINNs are specifically designed with a linearized Riemann solver that renders the scheme conservative and consistent with the unknown non-linear system of conservation laws in \cref{eq:unPDE}.~However, there is no guarantee that the linearized system is hyperbolic, i.e., that the GoRINNs Roe matrix in \cref{eq:RoeGNN} is diagonalizable with real eigenvalues.~To ensure hyperbolicity throughout the optimization process, one may include it as a hard constraint in the GoRINNs optimization problem in \cref{eq:OptGen}, the solution of which would then require constrained optimization algorithms.~Here, to avoid such algorithms, we initialized the LM algorithm with a random guess of the parameters $\mathbf{P}$, such that the Roe matrix in \cref{eq:RoeGNN} is diagonalizable with real eigenvalues.~For ensuring hyperbolicity throughout the optimization process with the LM algorithm, a high value of the initial damping factor should be avoided for averting big steps in the parameter space, which may lead to loss of hyperbolicity. 

\subsection{Implementation and numerical assessment of GoRINNs}
\label{sb:imp}
In this section, we provide all the practical details for the implementation and numerical assessment of GoRINNs, as employed for learning the unknown flux terms in the four benchmark problems considered.

For acquiring the data upon which GoRINNs are trained/tested on, we solved the forward problem for the known hyperbolic systems in the form of \cref{eq:genPDE} with well-posed initial data and boundary conditions.~In particular, we consider two types of boundary conditions; either (i) periodic, in which the flow exits from the left/right boundary and enters from the right/left with the same characteristics, or (ii) outflow, in which the flow exits from the boundaries without reflections.~For the former, we set the ghost cells to $\mathbf{Q}^n_{-1}=\mathbf{Q}^n_{N-1}$, $\mathbf{Q}^n_{0}=\mathbf{Q}^n_{N}$,  $\mathbf{Q}^n_{N+1}=\mathbf{Q}^n_{1}$ and $\mathbf{Q}^n_{N+2}=\mathbf{Q}^n_{2}$ at every time step, while for the latter, we use zero-order extrapolation by setting $\mathbf{Q}^n_{-1}=\mathbf{Q}^n_{0}=\mathbf{Q}^n_{1}$ and  $\mathbf{Q}^n_{N+1}=\mathbf{Q}^n_{N+2}=\mathbf{Q}^n_{N}$.~For each problem, we consider smooth Gaussian or sinusoidal initial conditions $\mathbf{u}(0,x)$, which lead to shock waves and/or rarefactions along the time interval considered.~For employing the high-resolution Godunov-type method described in \Cref{sb:FP_god}, we compute the waves and speeds via the Roe linearized matrices derived for each problem, except from the fourth benchmark problem of the PW equations, for which the HLLE solver was used.~In addition, in all problems, we used the Van-Leer flux-limiter function to derive the limited-versions of the waves in \cref{eq:HRupdate2}.~Finally, for the PW equations, where a source term is included in \cref{eq:genPDE}, we used the stiff numerical integration scheme provided by \textit{ode15s} of MATLAB ODE suite \cite{shampine1997matlab} for integrating the inhomogeneous part, as described in \Cref{sb:FP_god}. 

For each benchmark problem, we solved the forward problem for $4$ different initial conditions in the $t\in[0,t_{end}]$ interval with different, for each problem, $dt$ satisfying the CFL condition.~From the collected data, we randomly split the 15-15-70\% of the data to form the training, validation and testing sets; for the PW equations the split was 7.5-7.5-85\% because the CFL-satisfying $dt$ was smaller.~The resulting sets consist of pairs of data points $\mathbf{Q}_i^n,\mathbf{Q}_i^{n+1}\in\mathbb{R}^D$ for all cells $i=1,\ldots,N$ and for randomly selected time steps $n=1,\ldots,n_t$.~Only the training set was used for solving the GoRINNs optimization problem in \cref{eq:OptGen}, which resulted in the formation of a total of $(D+1)\times N\times n_t$ residuals.~We emphasize here that a larger training data set is not required, since only a few, complete time observations $n_t$ are enough for forming many residuals, even with coarse spatial discretizations.~For each problem, we determined $n_t$ so that the ratio between the resulting number of residuals and the total number of trainable parameters in $\mathbf{P}$ is $\sim5000$; the number of neurons was set to $L=5$, resulting to $5(D+2)$ trainable parameters.

As already discussed, for the GoRINNs optimization problem we used the LM algorithm over the training set, initialized with a random guess of parameters satisfying the hyperbolicity property of the Roe matrix in \cref{eq:RoeGNN} and an initial damping factor $\lambda_0=0.01$.~The stopping criteria of the LM algorithm were (i) a maximum number of $500$ epochs, (ii) a relative loss function ($l^2$) tolerance of $10^{-9}$ and (iii) a validation error below $10^{-9}$ for 3 consecutive iterations; the latter criterion was evaluated over the validation set in a frequency of $20$ training epochs.~To evaluate the convergence of the GoRINNs optimization problem, we computed the maximum and mean $l^1$ errors $\lVert \hat{\mathbf{Q}}_i^{n+1}-\mathbf{Q}_i^{n+1}\rVert_1$ and the $MSE(\hat{\mathbf{Q}}_i^{n+1},\mathbf{Q}_i^{n+1})$, over all cells $i=1,\ldots,N$, time steps $n=1,\ldots,n_t$, and state variable dimensions $D$; $\hat{\mathbf{Q}}_i^{n+1}$ is the predicted value of the state variable at the next time step, as shown in \cref{eq:FVGupdate}.~An outline of the algorithm for the solution of the GoRINNs optimization problem with the LM iterative scheme is provided in \Cref{algOptLM}.

\begin{algorithm}
\caption{Outline of the solution of the GoRINNs optimization problem in \cref{eq:OptGen} with the LM algorithm} \label{algOptLM}
\begin{algorithmic}[1]
\Require Flux term $\mathbf{f}_K(\mathbf{u}) + \mathbf{f}_U(\mathbf{u}, \mathcal{N}(\mathbf{u})$ and source term $\mathbf{s}(\mathbf{u})$ \Comment{Known and unknown terms in \cref{eq:unPDE}}
\Require State variables $\mathbf{Q}^n$ and $\mathbf{Q}^{n+1}$ for all $n_t$ pairs and for all cells \Comment{Training/validation data sets} 
\State Set hyperparameters $\mathbf{H}_2$ of the forward problem: $D$, $N$, $\Delta t$, $\Delta x$, flux-limiter function $\phi(\theta)$, type of boundary conditions and type of Riemann solver (ROE or HLLE)
\State Set NN hyperparameters $\mathbf{H}_1$ and randomly initialize NN parameters $\mathbf{P}$.
\State Set $isHyperbolic\gets False$
\While{$isHyperbolic=False$} \Comment{Ensure hyperbolicity at initial parameters $\mathbf{P}$}
    \For{$i=1,\ldots N$ and $n=1,\ldots,n_t$}
    \State Set $\mathbf{q}_l \gets \mathbf{Q}_{i-1}^n$ and $\mathbf{q}_r \gets \mathbf{Q}_{i}^n$; adjust with ghost cells when $i=1,2,N-1,N$
    \State Compute middle state $\bar{\mathbf{q}}$ from $\mathbf{q}_l,\mathbf{q}_r$  \Comment{Partially by \cref{eq:Roe_NN1} and partially by $(\mathbf{q}_l+\mathbf{q}_r)/2$}
    \State  Compute $\partial_{\mathbf{u}}\mathcal{N}(\bar{\mathbf{q}},\mathbf{P};\mathbf{H}_1)$ and form linearized Roe matrix $\hat{\mathbf{A}}_{i-1/2}(\mathbf{q}_l,\mathbf{q}_r)$ \Comment{By \cref{eq:RoeGNN}}
    \If{eigenvalues of $\hat{\mathbf{A}}_{i-1/2}$  are not real} Redraw randomly the parameters $\mathbf{P}$ and go to line 3 \EndIf
    \EndFor
    \State Set $isHyperbolic\gets True$
\EndWhile
\Statex  \textcolor{gray}{$\triangleright$~\textit{Residual minimization with the LM algorithm}}
\State Set hyperparameters $\mathbf{H}_2$ of the LM algorithm: stopping criteria and initial damping factor $\lambda_0$
\Repeat
\For{$n=1,\ldots,n_t$ and $i=1,\ldots,N$}
\State Compute $\hat{\mathbf{Q}}^{n+1}_i \gets FVG_{HR}\left(\mathbf{Q}^n_{i-2},\mathbf{Q}^n_{i-1},\mathbf{Q}^n_{i},\mathbf{Q}^n_{i+1},\mathbf{Q}^n_{i+2},\mathbf{P};\mathbf{H}_1,\mathbf{H}_2 \right)$ \Comment{By \cref{eq:FVGupdate} and \cref{algFVG}}
\State Calculate loss $\mathbf{Q}^{n+1}_i-\hat{\mathbf{Q}}^{n+1}_i$ \Comment{Loss function in \cref{eq:Opt1}}
\EndFor
\State Form $(D\times N\times n_t)$-dim. vector of residuals $\mathcal{F}_{FVG}(\mathbf{P})$
\For{$n=1,\ldots,n_t$ and $i=1,\ldots,N$}
\State Set $\mathbf{q}_l \gets \mathbf{Q}_{i-1}^n$ and $\mathbf{q}_r \gets \mathbf{Q}_{i}^n$; adjust with ghost cells when $i=1,2,N-1,N$
\State Compute middle state $\bar{\mathbf{q}}$ from $\mathbf{q}_l,\mathbf{q}_r$  \Comment{Partially by \cref{eq:Roe_NN1} and partially by $(\mathbf{q}_l+\mathbf{q}_r)/2$}
\State Calculate loss $ \left( \partial_{\mathbf{u}} \mathbf{f}_U(\bar{\mathbf{q}},\mathcal{N}(\bar{\mathbf{q}},\mathbf{P};\mathbf{H}_1)) + \partial_{\mathcal{N}} \mathbf{f}_U(\bar{\mathbf{q}},\mathcal{N}(\bar{\mathbf{q}},\mathbf{P};\mathbf{H}_1)) \cdot \partial_{\mathbf{u}}\mathcal{N}(\bar{\mathbf{q}},\mathbf{P};\mathbf{H}_1) \right)  \cdot (\mathbf{q}_r-\mathbf{q}_l) - \\ \left(  \mathbf{f}_U(\mathbf{q}_r,\mathcal{N}(\mathbf{q}_r,\mathbf{P};\mathbf{H}_1)) -\mathbf{f}_U(\mathbf{q}_l,\mathcal{N}(\mathbf{q}_l,\mathbf{P};\mathbf{H}_1)) \right)$ \Comment{Loss function in \cref{eq:Opt2}}
\EndFor
\State Form $(N\times  n_t)$-dim. vector of residuals $\mathcal{F}_{RH}(\mathbf{P})$
\State Obtain $\partial_{\mathbf{P}} \mathcal{F}_{FVG} (\mathbf{P})$ and $\partial_{\mathbf{P}} \mathcal{F}_{HR} (\mathbf{P})$ with numerical differentiation to form Jacobian matrix 
\State Perform LM update \cite{hagan1994training} to get updated values of NN parameters $\mathbf{P}$ and damping factor $\lambda$
\Until{convergence}
\State Get the final NN parameter values $\mathbf{P}$
\State Compute maximum and mean $l^1$ errors $\lVert \hat{\mathbf{Q}}_i^{n+1}-\mathbf{Q}_i^{n+1}\rVert_1$ and the $MSE(\hat{\mathbf{Q}}_i^{n+1},\mathbf{Q}_i^{n+1})$, for all $i=1,\ldots,N$, $n=1,\ldots,n_t$, and $D$
\end{algorithmic}
\end{algorithm}

Finally, to assess the numerical accuracy of GoRINNs, we solve the forward problem using the learned flux function $\mathbf{f}_U(\mathbf{u},\mathcal{N}(\mathbf{u}))$ in \cref{eq:unPDE} and compare its solution $\hat{\mathbf{Q}}_i^{n+1}$ with the ground truth solution $\mathbf{Q}_i^{n+1}$ of the forward problem for the analytically known flux function in \cref{eq:genPDE}.~To quantify the numerical accuracy, we use the absolute errors $\lvert \hat{\mathbf{Q}}_i^{n+1}-\mathbf{Q}_i^{n+1}\rvert$ all over cells and time steps for all the components of the state variables $D$.~Note that this error is not the same with the training/validation/testing errors, since the solution of the forward problem with the learned flux function $\mathbf{f}_U(\mathbf{u},\mathcal{N}(\mathbf{u}))$ also takes into account the cumulative  error introduced by the numerical scheme.

All simulations were carried out with a CPU Intel(R) Core(TM) i7-13700H @ 2.40 GHz, RAM 32.0 GB using MATLAB R2022b with 14 parallel cores.

\section{Numerical Results}
\label{sec:NR}

In this section, we assess the performance of GoRINNs for the solution of the inverse problem for four benchmark problems of hyperbolic conservation laws.~The first is the inviscid Burgers' equation \cite{burgers1948mathematical}, a  benchmark problem in fluid mechanics end beyond (see e.g., \citep{ketcheson2020riemann,leveque2002finite,dafermos2005hyperbolic}).~The second problem, the Lighthill-Whitham-Richards (LWR) equation \cite{lighthill1955kinematic,richards1956shock}, comes from the field of traffic flow, and shares many similarities with the Burgers' equation.~Here, we consider it to demonstrate the use of GoRINNs on finding closures when partial information of the unknown flux function are available.~For the next two problems, we consider two systems of $D=2$ non-linear conservation laws; the Shallow Water (SW) equations and the Payne-Whitham (PW) model.~The SW equations can be derived by ``depth-averaging'' the Navier-Stokes equations for an incompressible and inviscid fluid under hydrostatic pressure.~The SW equations consist one of the easiest systems to step up from scalar to systems of conservation laws.~For all the above problems, the Riemann problem can be solved exactly \citep{leveque2002finite,ketcheson2020riemann}.~This is not the case for the last benchmark problem, the PW model \citep{payne1971model,whitham2011linear} of traffic flow, which describes a class of macroscopic second-order traffic models that take into account the individual behavior of the driver \cite{treiber2013traffic,helbing1998generalized}.~We consider it here to demonstrate that GoRINNs can effectively handle hyperbolic systems with source terms.

For training the GoRINNs, we followed the procedure described in \cref{sb:imp} to collect training/validation/testing data sets from the numerical solutions of the analytically known hyperbolic PDEs.~A presentation of the PDEs, and details on data acquisition, are provided in \Cref{app:Bench}.~Next, we assessed the numerical accuracy of GoRINNs, by comparing the numerical solutions, provided by the high-resolution Godunov-type scheme in \cref{sb:FP_god}, for the learned by GoRINNs, hyperbolic equation(s) and the analytically known equation(s) in \Cref{app:Bench}. 

\subsection{Learning the Burgers' equation}
Given the data set generated on the basis of the numerical solution of the Burgers' equation in \cref{sb:B_FP}, a subset of which is depicted in \cref{fig:solsBR_LWR}a, we assume that the conservation law, from which they are originating, is written as:
\begin{equation}
    \partial_t u + \partial_x \mathcal{N}(u) = 0, \label{eq:unBR}
\end{equation}
where the flux function is completely unknown, thus corresponding to $f_K(u)=0$, $f_U(u,\mathcal{N}(u))=\mathcal{N}(u)$ and $s(u)=0$ in \cref{eq:unPDE}.~To derive a Roe matrix for \cref{eq:unBR} that satisfies the conditions of \Cref{prop1},  \cref{eq:Roe_NN1} does not provide any information about the average state $\bar{q}$, since $f_K(u)=0$.~We thus take the simple arithmetic average $\bar{q}=(q_l+q_r)/2$ and require \cref{eq:Roe_NN2} to hold, yielding:
\begin{equation*}
    \partial_u \mathcal{N}(\bar{q}) (q_r-q_l) = \mathcal{N}(q_r) - \mathcal{N}(q_l),
\end{equation*}
for any right and left states $q_r,q_l$ in the data set.~Then, the loss function in \cref{eq:Opt2} rendering GoRINNs conservative, reduces to:
\begin{equation}
    \mathcal{L}_{RH}\left(\mathbf{Q}^n, \mathbf{P};\mathbf{H}_1 \right) = \sum_{i=1}^N \Big( \left( \partial_{\mathbf{u}}\mathcal{N}(\bar{\mathbf{Q}}_i^n,\mathbf{P};\mathbf{H}_1) \right) \cdot (\mathbf{Q}^n_i-\mathbf{Q}^n_{i-1}) - \left(  \mathcal{N}(\mathbf{Q}^n_i,\mathbf{P};\mathbf{H}_1) -\mathcal{N}(\mathbf{Q}^n_{i-1},\mathbf{P};\mathbf{H}_1) \right) \Big), \label{eq:Opt2_BR}
\end{equation}
which was used for solving the GoRINNs optimization problem in \cref{eq:OptGen}.

~As already discussed in \cref{sb:imp}, we considered $L=5$ neurons in the hidden layer and a logistic sigmoid activation function for learning the physical flux function $\mathcal{N}(u)$ with GoRINNs.~The training results of GoRINNs are shown in \cref{tb:Errors}, where the maximum and mean $\lVert \hat{\mathbf{Q}}_i^{n+1}-\mathbf{Q}_i^{n+1}\rVert_1$ errors and the MSE$(\hat{\mathbf{Q}}_i^{n+1},\mathbf{Q}_i^{n+1})$ are reported \textit{over all} $N=100$ cells and $n_t=360$/$360$/$1680$ (training/validation/testing) time steps.~Clearly, the GoRINNs optimization scheme converges, since the maximum $\lVert \hat{\mathbf{Q}}_i^{n+1}-\mathbf{Q}_i^{n+1}\rVert_1$ error is of the order $1E-06$.

\begin{table}[!h]
    \centering
    \resizebox{\textwidth}{!}{
    \begin{tabular}{l| c c c | c c c | c c c}
    \toprule
    \textbf{Problem} &	\multicolumn{3}{c}{\textbf{max  $\lVert \hat{\mathbf{Q}}_i^{n+1}-\mathbf{Q}_i^{n+1}\rVert_1$}}		&	\multicolumn{3}{c}{\textbf{mean  $\lVert \hat{\mathbf{Q}}_i^{n+1}-\mathbf{Q}_i^{n+1}\rVert_1$}}	&	\multicolumn{3}{c}{\textbf{MSE$(\hat{\mathbf{Q}}_i^{n+1},\mathbf{Q}_i^{n+1})$}}	\\[2pt]
    &	Train	&	Val	&	Test	&	Train	&	Val	&	Test	&	Train	&	Val	&	Test	\\
    \midrule
    Burgers'	&	6.62E$-$06	&	5.38E$-$06	&	6.51E$-$06	&	3.60E$-$08	&	3.68E$-$08	&	3.63E$-$08	&	2.23E$-$14	&	2.23E$-$14	&	2.10E$-$14	\\
    LWR	&	5.15E$-$07	&	5.11E$-$07	&	5.34E$-$07	&	4.14E$-$09	&	3.68E$-$09	&	3.81E$-$09	&	3.30E$-$16	&	2.52E$-$16	&	2.98E$-$16	\\
    SW	&	1.08E$-$06	&	1.05E$-$06	&	1.04E$-$06	&	1.29E$-$08	&	1.21E$-$08	&	1.27E$-$08	&	1.35E$-$15	&	9.30E$-$16	&	1.24E$-$15	\\
    PW, $P(\rho,q)$	&	1.78E$-$03	&	1.78E$-$03	&	2.42E$-$03	&	6.77E$-$05	&	6.64E$-$05	&	6.56E$-$05	&	3.24E$-$08	&	3.24E$-$08	&	3.23E$-$08	\\
    PW, $P(\rho)$	&	8.71E$-$06	&	8.59E$-$06	&	9.51E$-$06	&	2.44E$-$07	&	2.45E$-$07	&	2.53E$-$07	&	3.73E$-$13	&	3.78E$-$13	&	3.94E$-$13	\\
    \bottomrule
    \end{tabular}}
    \caption{Training, validation and testing errors of the GoRINNs optimization problem for all benchmark problems considered.~The maximum and absolute $\lVert \hat{\mathbf{Q}}_i^{n+1}-\mathbf{Q}_i^{n+1}\rVert_1$ errors and the mean squared error MSE$(\hat{\mathbf{Q}}_i^{n+1},\mathbf{Q}_i^{n+1})$ are reported over all cells $i=1,\ldots,N$, time steps $n=1,\ldots,n_t$, and state variable dimensions $D$}
    \label{tb:Errors}
\end{table}
The numerical accuracy of GoRINNs is visualized in \cref{fig:GNN_BR_LWR}a,c.~\Cref{fig:GNN_BR_LWR}a displays the absolute errors between the numerical solution, $\hat{u}(t,x)$, of \cref{eq:unBR} with the learned by the GoRINNs flux functional $\mathcal{N}(u)$, and the ground truth one, $u(t,x)$, of the Burgers' equation in \cref{eq:Burg}; the latter is the numerical solution displayed in \cref{fig:solsBR_LWR}a (in Appendix).~As shown, the numerical accuracy provided by the GoRINNs is very high, even near the shock wave regions where the error is at most $6E-05$.~In addition, we demonstrate in \cref{fig:GNN_BR_LWR}c that the flux functional learned by the GoRINNs is in excellent agreement with that of the Burgers' equation $u^2/2$.
\begin{figure}[!h]
    \centering
    \subfigure[Burgers' equation: $\lvert \hat{u}(t,x)-u(t,x)\rvert$]{\includegraphics[width=0.45\textwidth]{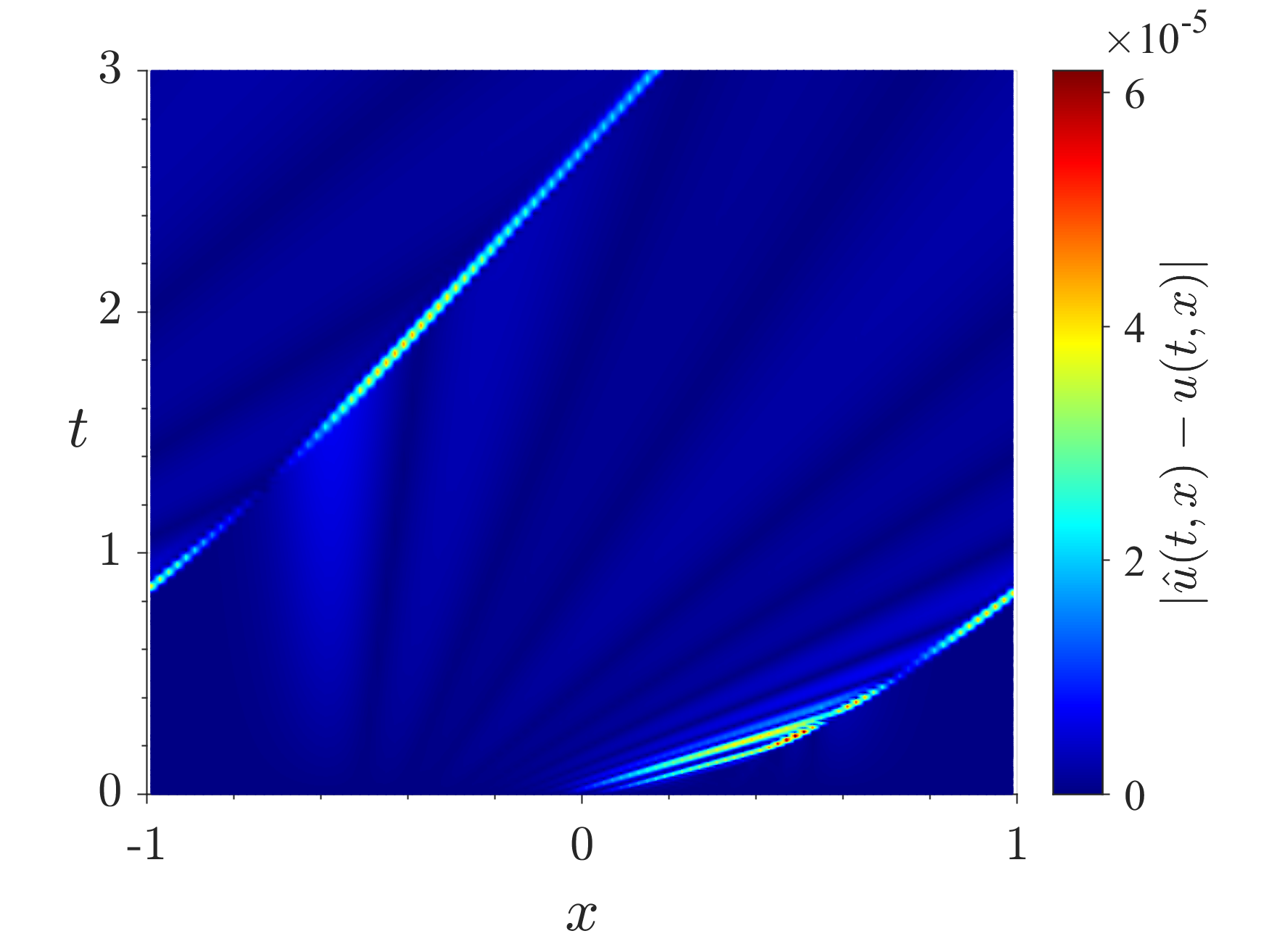}}
    \subfigure[LWR equation: $\lvert \hat{\rho}(t,x)-\rho(t,x)\rvert$]{\includegraphics[width=0.45\textwidth]{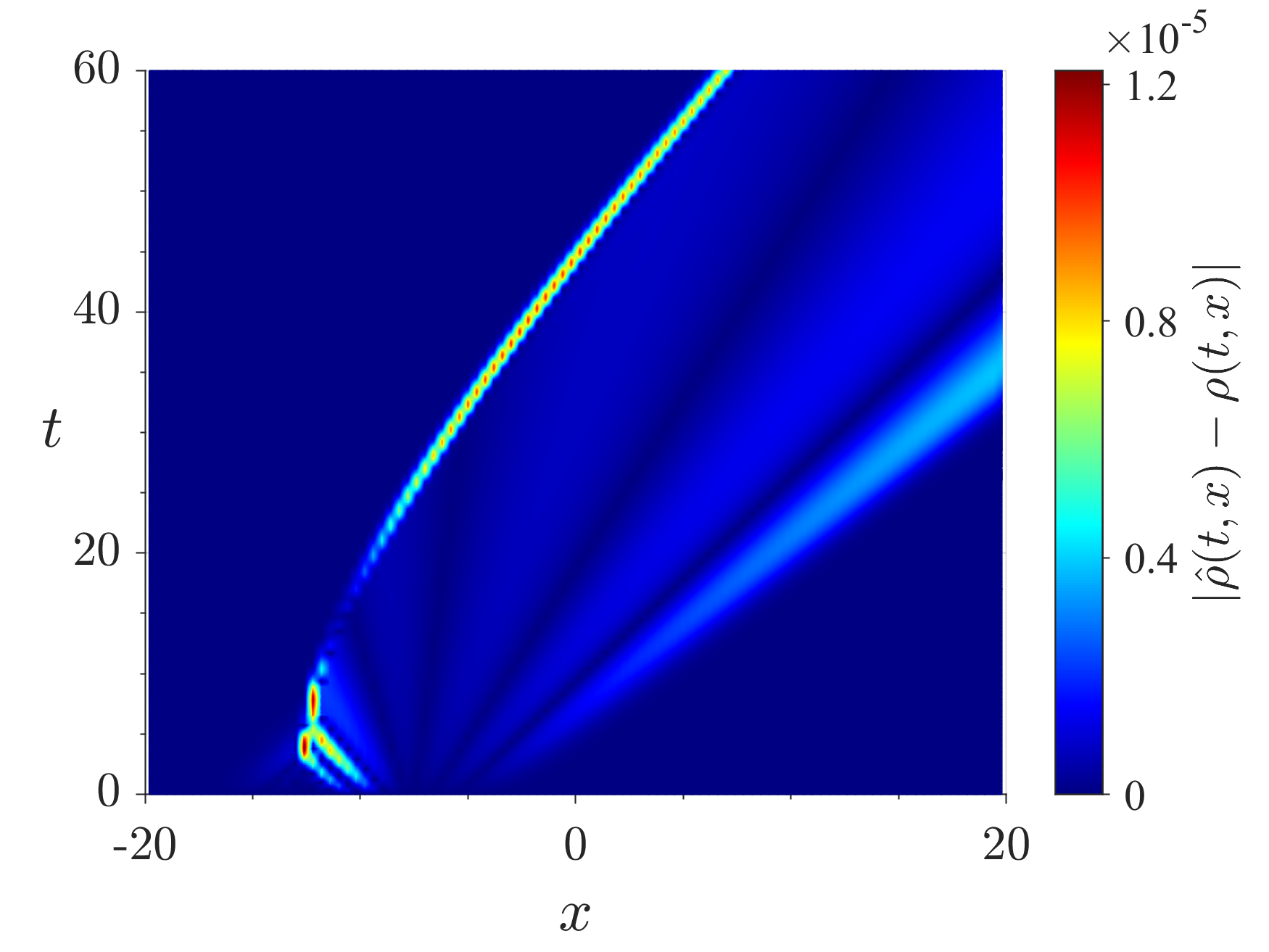}} \\
    \subfigure[Burgers' equation: $u^2/2$ vs GoRINNs $\mathcal{N}(u)$]{\includegraphics[width=0.45\textwidth]{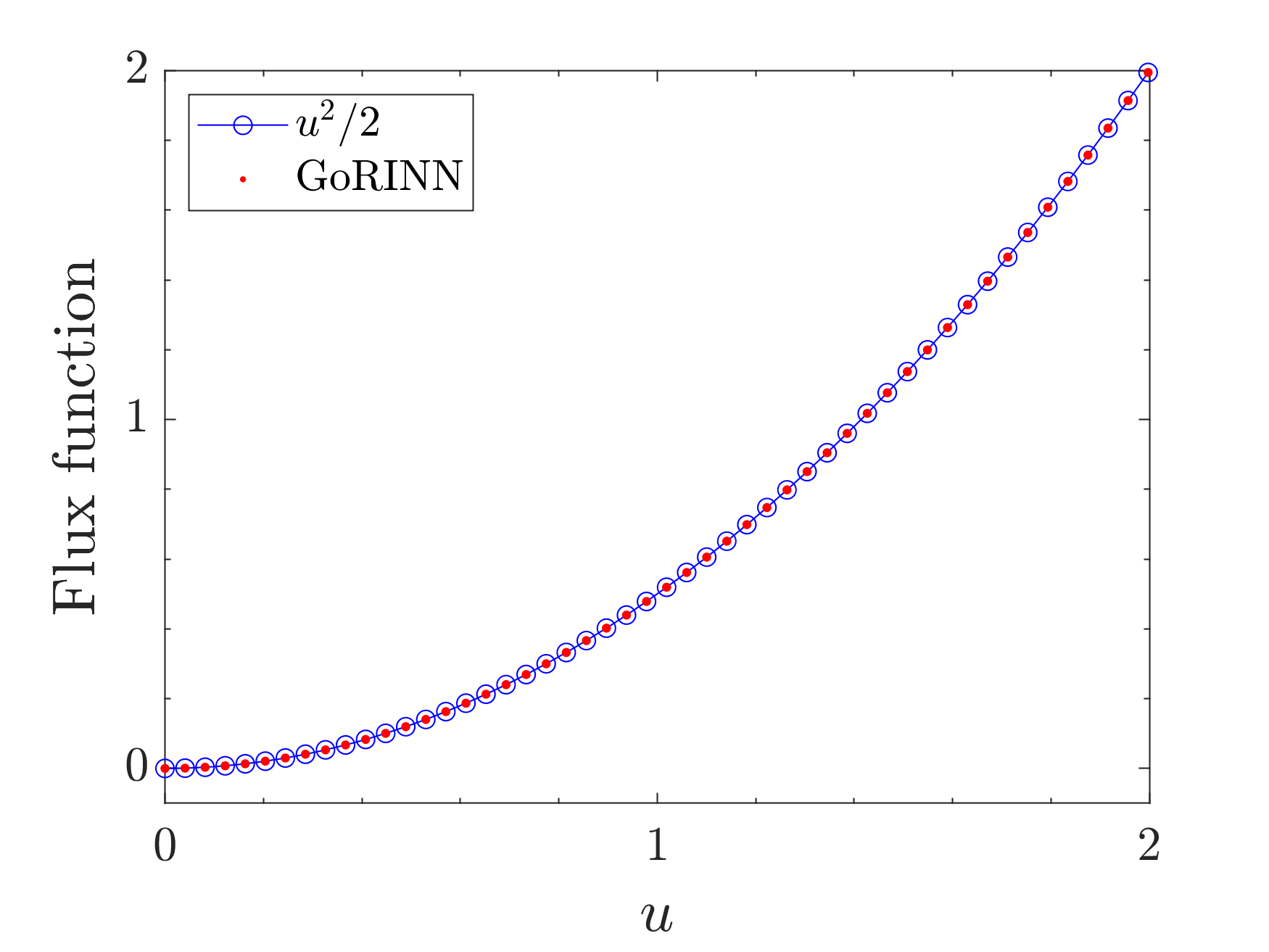}}
    \subfigure[LWR equation: $v(\rho)$ vs GoRINNs $\mathcal{N}(\rho)$]{\includegraphics[width=0.45\textwidth]{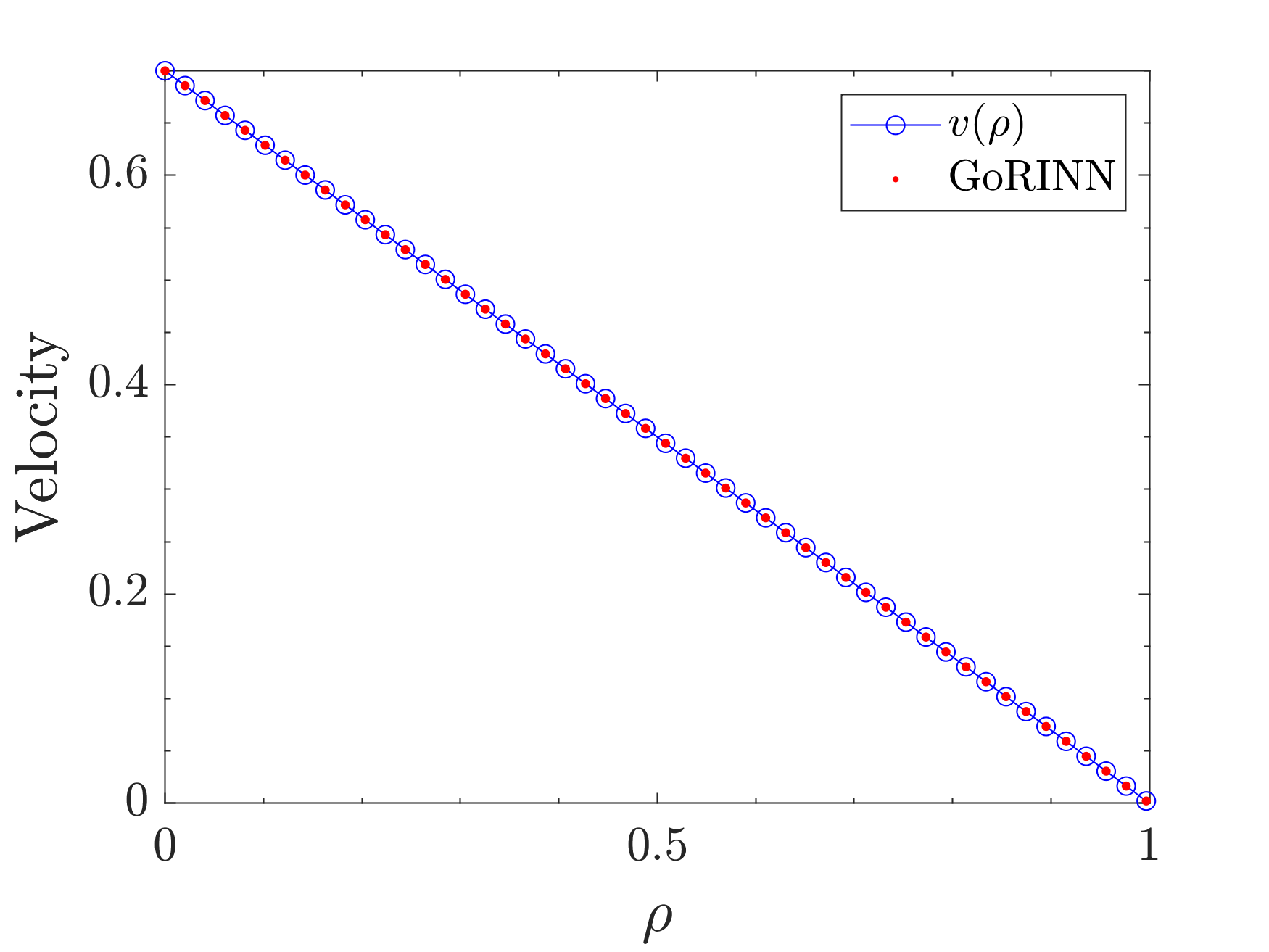}}
    \caption{Numerical accuracy of GoRINNs for the Burgers' and LWR equations.~(a,b) Absolute errors of the numerical solution provided by the GoRINNs learned equations in \cref{eq:unBR,eq:unLWR} ($\hat{u}$ and $\hat{\rho}$) vs the one provided by Burgers' and LWR equations in \cref{eq:Burg,eq:LWR} ($u$ and $\rho$).~(c,d) Analytically known flux function and velocity closure vs the $\mathcal{N}(u)$/$\mathcal{N}(\rho)$ functional learned by GoRINNs.}
    \label{fig:GNN_BR_LWR}
\end{figure}

\subsection{Learning the velocity closure of the LWR equation}
Here, we assume that partial information regarding the conservation law is available, but we miss the velocity closure.~In particular, given the set of data generated via the numerical solution of the LWR equation in \cref{sb:LWR_FP}, a subset of which is depicted in \cref{fig:solsBR_LWR}c, we assume that the originating equation is of the form:
\begin{equation}
    \partial_t \rho + \partial_x \left( \rho \mathcal{N}(\rho)\right) = 0, \label{eq:unLWR}
\end{equation}
in which we seek to approximate the unknown velocity closure $v(\rho)$ by the NN $\mathcal{N}(\rho)$.~Using the nomenclature in \cref{eq:unPDE}, the unknown function is $f_U(\rho,\mathcal{N}(\rho))=\rho \mathcal{N}(\rho)$ and $f_K(u)=s(u)=0$.~For deriving a Roe matrix of \cref{eq:unLWR} that satisfies the  conditions of \Cref{prop1}, \cref{eq:Roe_NN1} does not provide any information about the average state $\bar{\rho}$, since $f_K(u)=0$.~As in the case of \cref{eq:unBR}, here we again assume the arithmetic average $\bar{\rho}=(\rho_l+\rho_r)/2$ and require \cref{eq:Roe_NN2} to hold, which yields:
\begin{equation*}
    \left(  \mathcal{N}(\bar{\rho}) + \bar{\rho} \partial_{\rho} \mathcal{N}(\bar{\rho}) \right) (\rho_r-\rho_l) = \rho_r\mathcal{N}(\rho_r) - \rho_l\mathcal{N}(\rho_l),
\end{equation*}
for any right and left states $\rho_r,\rho_l$ provided by the data set.~Then, the loss function in \cref{eq:Opt2} rendering GoRINNs conservative, reduces to:
\begin{multline}
    \mathcal{L}_{RH}\left(\mathbf{Q}^n, \mathbf{P};\mathbf{H}_1 \right) = \sum_{i=1}^N \Big( \left( \mathcal{N}(\bar{\mathbf{Q}}_i^n,\mathbf{P};\mathbf{H}_1) + \bar{\mathbf{Q}}_i^n \cdot \partial_{\mathbf{u}}\mathcal{N}(\bar{\mathbf{Q}}_i^n,\mathbf{P};\mathbf{H}_1) \right) \cdot (\mathbf{Q}^n_i-\mathbf{Q}^n_{i-1}) - \\ \left(  \mathbf{Q}^n_i \cdot \mathcal{N}(\mathbf{Q}^n_i,\mathbf{P};\mathbf{H}_1) - \mathbf{Q}^n_{i-1} \cdot \mathcal{N}(\mathbf{Q}^n_{i-1},\mathbf{P};\mathbf{H}_1) \right) \Big), \label{eq:Opt2_LWR}
\end{multline}
which was used for solving the GoRINNs optimization problem in \cref{eq:OptGen}; similarly to the Burgers' equation, $L=5$ neurons in the hidden layer and a logistic sigmoid activation function are considered for approximating $\mathcal{N}(\rho)$.

~The training results are reported in \cref{tb:Errors}, again indicating that the GoRINNs optimization scheme converges.~In particular, the maximum $\lVert \hat{\mathbf{Q}}_i^{n+1}-\mathbf{Q}_i^{n+1}\rVert_1$ error \textit{over all} $N=100$ cells and $n_t=360$/$360$/$1680$ (training/validation/testing) time steps is around $5E-07$.

For the assessment of the numerical accuracy of the GoRINNs, we compare in \cref{fig:GNN_BR_LWR}b the numerical solution, $\hat{\rho}(t,x)$, of the unknown \cref{eq:unLWR} with the learned by the GoRINNs flux functional $\rho\mathcal{N}(\rho)$ and the ground true one, $\rho(t,x)$, of the LWR equation as shown in \cref{fig:solsBR_LWR}c.~The differences are minor, with the absolute error being very low all over the time-space domain, even near the shock wave regions (the maximum there is $1.2E-05$).~We further demonstrate in \cref{fig:GNN_BR_LWR}d that the velocity closure approximated by the GoRINNs is almost identical with the linear one of the LWR equation, $v(\rho)=v_{max}(1-\rho)$.

\subsection{Learning the pressure closure of the SW equations}

For this problem, we consider the set of $D=2$-dim. data constructed on the basis of the numerical solution of the SW equations in \cref{sb:SW_FP}, which consists of observations for $h$ and $q$ along the 1-dim. spatial domain $x\in[-5,5]$; a subset of the training set is shown in \cref{fig:solsSW_PW}a,b.~We now assume that the conservation of mass is known, while the one of momentum is missing the pressure closure $P(h)$, which we seek to approximate with a NN $\mathcal{N}(h,q)$.~We highlight here that the independence of the pressure by the momentum $q$ is not a-priori known.~Then, the unknown system of conservation laws is written in the form of \cref{eq:unPDE} as:
\begin{equation}
    \partial_t \begin{bmatrix}
        h \\ q
    \end{bmatrix} + \partial_x  \begin{bmatrix}
        q \\ \dfrac{q^2}{h} + \mathcal{N}(h,q)
    \end{bmatrix}  = \begin{bmatrix}
        0 \\ 0
    \end{bmatrix},
    \label{eq:unSW}
\end{equation}
where the known flux term is $\mathbf{f}_K([h,q]^\top)=[q,q^2/h]^\top$ and the unknown one is simply composed by the NN as $\mathbf{f}_U([h,q]^\top)=[0,\mathcal{N}(h,q)]^\top$.~To derive a Roe matrix for \cref{eq:unSW} that satisfies the conditions of \cref{prop1}, we first employ \cref{eq:Roe_NN1} to get information for the average state $[\bar{h},\bar{q}]^\top$.~This results to a $D=2$-dim. system of equations, where the non-trivial one reads:
\begin{equation}
    -\dfrac{\bar{q}^2}{\bar{h}^2} (h_r-h_l) + \dfrac{2\bar{q}}{\bar{h}} (q_r-q_l) = \dfrac{q_r^2}{h_r}-\dfrac{q_l^2}{h_l} \Rightarrow \bar{q}=\bar{h} \dfrac{q_l h_l^{-1/2} + q_r h_r^{-1/2}}{h_l^{1/2}+h_r^{1/2}}. \label{eq:SWutRoe1}
\end{equation}
The above expression of $\bar{q}$ agrees with the one of the SW equations in \cref{eq:SW_Roe1}.~Now, given $\bar{q}$, we use the simple arithmetic average $\bar{h}=(h_r+h_l)/2$ and require \cref{eq:Roe_NN2} to hold, which yields:
\begin{equation*}
    \partial_h \mathcal{N}(\bar{h},\bar{q}) (h_r-h_l) + \partial_q \mathcal{N}(\bar{h},\bar{q}) (q_r-q_l) = \mathcal{N}(h_r,q_r) - \mathcal{N}(h_l,q_l),
\end{equation*}
for any right and left states $(h_r,q_r)$, $(h_l,q_l)$ provided by the data set.~Then, the loss function in \cref{eq:Opt2} that renders GoRINNs conservative, reduces to:
\begin{multline}
    \mathcal{L}_{RH}\left(\mathbf{Q}^n, \mathbf{P};\mathbf{H}_1 \right) = \sum_{i=1}^N \Big(  \partial_{\mathbf{u}}\mathcal{N}(\bar{\mathbf{Q}}_i^n,\mathbf{P};\mathbf{H}_1) \cdot (\mathbf{Q}^n_i-\mathbf{Q}^n_{i-1}) - \left(  \mathcal{N}(\mathbf{Q}^n_i,\mathbf{P};\mathbf{H}_1) -\mathcal{N}(\mathbf{Q}^n_{i-1},\mathbf{P};\mathbf{H}_1) \right) \Big), \label{eq:Opt2_SW}
\end{multline} 
where, in this benchmark problem, $\bar{\mathbf{Q}}_i^n$ requires only the computation of $\bar{h}$ for every cell, since $\bar{q}$ is given by \cref{eq:SWutRoe1}.~We thus used the loss function in \cref{eq:Opt2_SW} to solve the GoRINNs optimization problem in \cref{eq:OptGen}.~The approximation of the physical flux function $\mathcal{N}(h,q)$ shares the same architecture as in previous benchmark problems.~The resulting training results are reported in \cref{tb:Errors}, where the small training/validation/testing errors indicate the convergence of the GoRINNs optimization scheme.~In fact, the maximum $\lVert \hat{\mathbf{Q}}_i^{n+1}-\mathbf{Q}_i^{n+1}\rVert_1$ error \textit{over all} $N=200$ cells and $n_t=180$/$180$/$840$ (training/validation/testing) time steps for both state variables $D=1,2$ is around $1E-06$.

The numerical accuracy of GoRINNs is visualized in \cref{fig:GNN_SW}.~\Cref{fig:GNN_SW}a,b display the absolute errors, for each variable of the system,  between the numerical solution of \cref{eq:unSW} with the learned by the GoRINNs flux functional $\mathcal{N}(h,q)$ and the ground truth one, as computed from the SW equations in \cref{eq:SWcon}; the latter is the numerical solution displayed in \cref{fig:solsSW_PW}a,b.~Again, the GoRINNs provide very high numerical accuracy for both $h$ and $q$ variables, even near the shock wave and rarefaction regions where the error is at most $1.5E-05$.~In addition, the excellent agreement of the analytically known pressure closure, $P(h)$, and the GoRINNs learned functional, $\mathcal{N}(h,q)$, is displayed in \cref{fig:GNN_SW}c.~In fact, GoRINNs are capable to learn the independence of the pressure closure to $q$; $\mathcal{N}(h,q)$ does not change along that $q$-axis.
\begin{figure}[!h]
    \centering
    \subfigure[$\lvert \hat{h}(t,x)-h(t,x)\rvert$]{\includegraphics[width=0.32\textwidth]{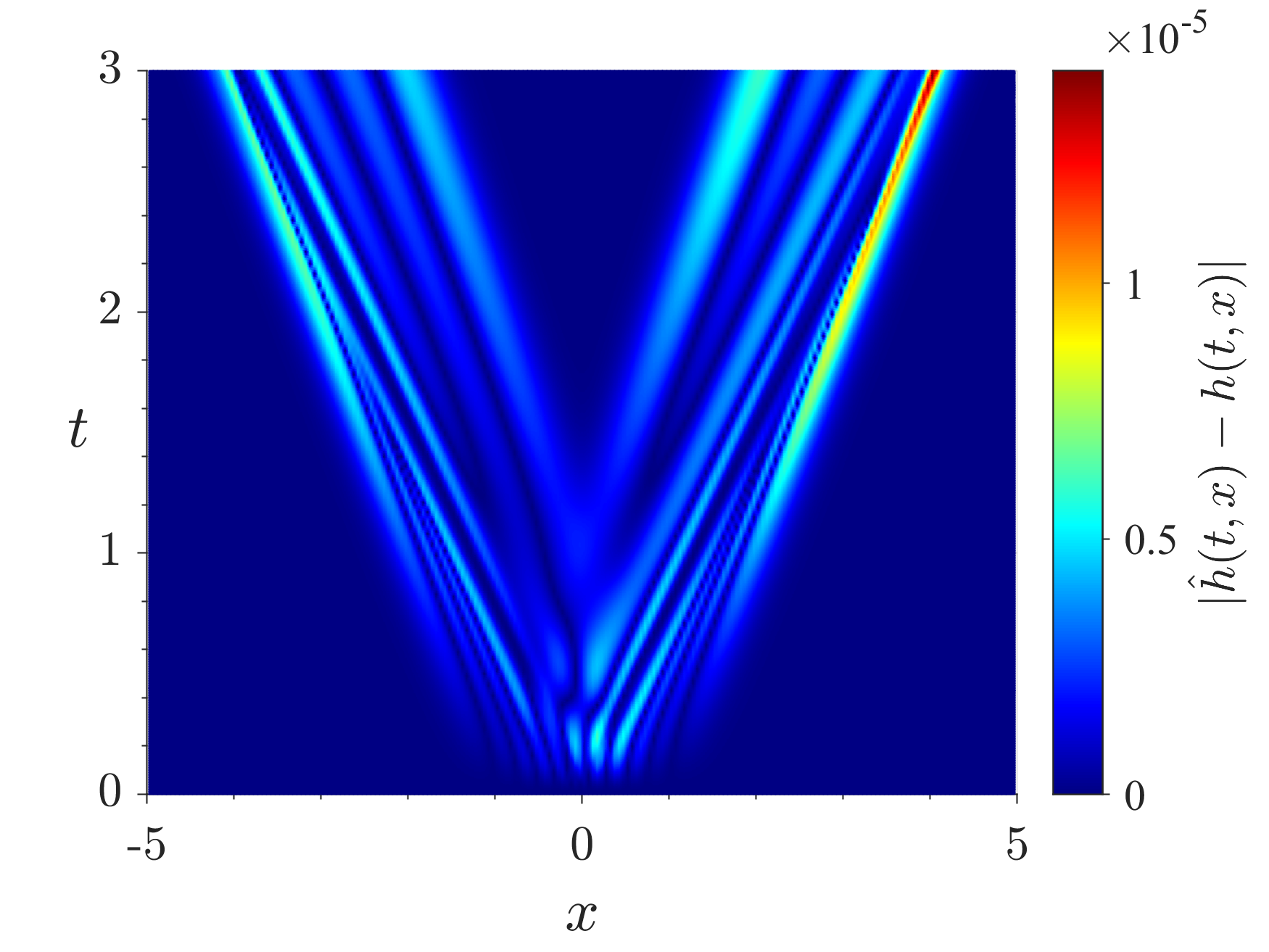}}
    \subfigure[$\lvert \hat{q}(t,x)-q(t,x)\rvert$]{\includegraphics[width=0.32\textwidth]{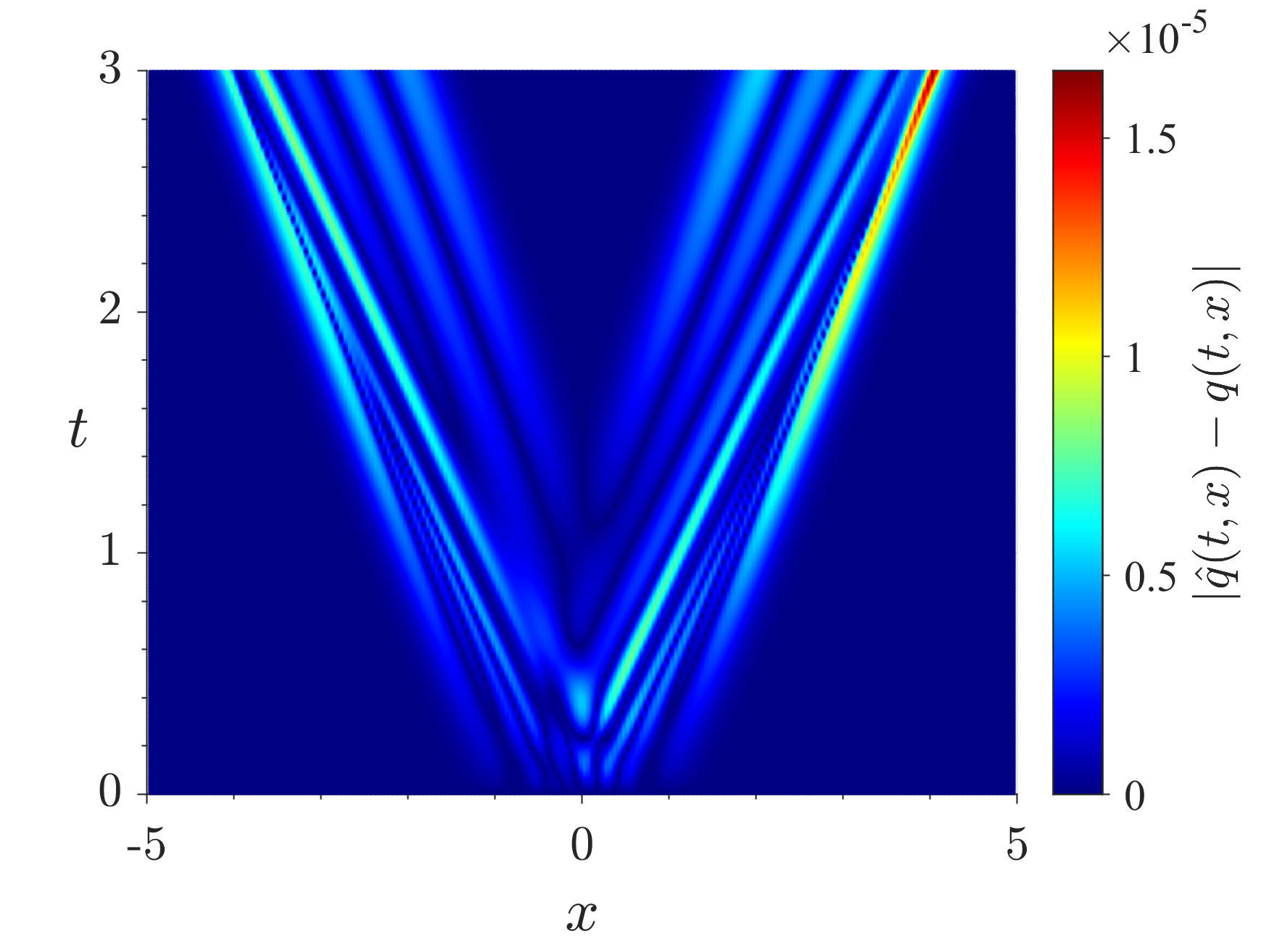}}
    \subfigure[$P(h)$ vs GoRINNs $\mathcal{N}(h,q)$]{\includegraphics[width=0.32\textwidth]{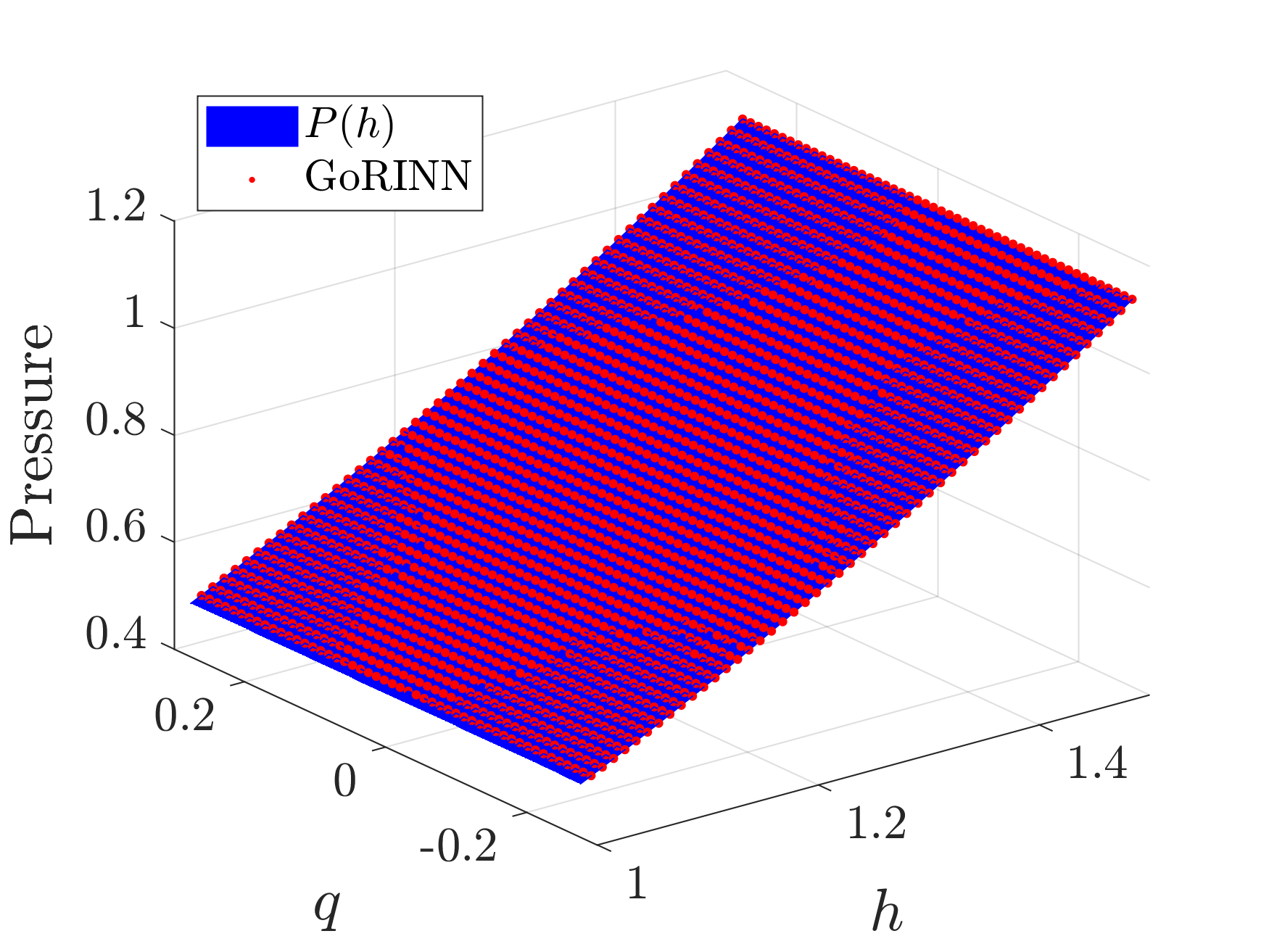}}
    \caption{Numerical accuracy of GoRINNs for the SW equations.~(a,b) Absolute errors of the numerical solution provided by the GoRINNs learned equations in \cref{eq:unSW} ($\hat{h}$ and $\hat{q}$) vs the one provided by SW equations in \cref{eq:SWcon} ($h$ and $q$).~(c) Analytically known pressure closure $P(h)$ vs the $\mathcal{N}(h,q)$ functional learned with GoRINNs.}
    \label{fig:GNN_SW}
\end{figure}

\subsection{Learning the pressure closure of the PW equations}
\label{sb:PW_IP}
In this problem, the data set that was constructed via the solution of the forward problem of the PW equations in \cref{sb:PW_FP}, consists of a set of $D=2$-dim. observations for the density $\rho$ and momentum $q$ along the 1-dim. spatial domain $x\in[0,800]$; a subset of these data is shown in \cref{fig:solsSW_PW}d,e.~As in the previous problem, the mass conservation is known, while the momentum conservation is lacking the pressure closure, which we seek to learn, without any a-priori knowledge on its variable dependencies, thus assuming that it can be approximated by a NN $\mathcal{N}(\rho,q)$.~However, in this problem, we have information about the source term that renders the momentum equation inhomogeneous.~With the above assumptions, the unknown system of conservation laws is written in the form of \cref{eq:unPDE} as:
\begin{equation}
    \partial_t \begin{bmatrix}
        \rho \\ q
    \end{bmatrix} + \partial_x  \begin{bmatrix}
        q \\ \dfrac{q^2}{\rho} + \mathcal{N}(\rho,q)
    \end{bmatrix}  = \begin{bmatrix}
        0 \\ \dfrac{\rho V_e(\rho)-q}{\tau}
    \end{bmatrix},
    \label{eq:unPW}
\end{equation}
where the known flux term is $\mathbf{f}_K([\rho,q]^\top)=[q,q^2/\rho]^\top$, the unknown one is $\mathbf{f}_U([\rho,q]^\top)=[0,\mathcal{N}(\rho,q)]^\top$ and the source term the same of the PW equations in \cref{eq:PWcon}.~Observe that the homogeneous system in \cref{eq:unPW} is exactly the same as the one in \cref{eq:unSW}.~Thus, similarly to \cref{eq:SWutRoe1}, for the Roe matrix we obtain by \cref{eq:Roe_NN1} an average state: 
\begin{equation}
    -\dfrac{\bar{q}^2}{\bar{\rho}^2} (\rho_r-\rho_l) + \dfrac{2\bar{q}}{\bar{\rho}} (q_r-q_l) = \dfrac{q_r^2}{\rho_r}-\dfrac{q_l^2}{\rho_l} \Rightarrow \bar{q}=\bar{\rho} \dfrac{q_l \rho_l^{-1/2} + q_r \rho_r^{-1/2}}{\rho_l^{1/2}+\rho_r^{1/2}}, \label{eq:PWutRoe1}
\end{equation}
which also agrees with the one of the known PW equations in \cref{eq:PW_Roe1}.~According to \Cref{prop1}, to render GoRINNs conservative, we again assume the arithmetic average $\bar{\rho}=(\rho_r+\rho_l)/2$ and require \cref{eq:Roe_NN2} to hold, which implies the same simplified loss function as the one in \cref{eq:Opt2_SW}.~We then use this loss function to solve the GoRINNs optimization problem in \cref{eq:OptGen}.

The training results are reported in \cref{tb:Errors}; see row denoted as ``PW, $P(\rho,q)$''.~It is clearly shown that the GoRINNs optimization problem does not converge in as low values as in the other problems.~In fact, the maximum $\lVert \hat{\mathbf{Q}}_i^{n+1}-\mathbf{Q}_i^{n+1}\rVert_1$ error \textit{over all} $N=100$ cells and $n_t=360$/$360$/$4080$ (training/validation/testing) time steps for both state variables $D=1,2$, is not negligible; at the order of $\mathcal{O}(10^{-3})$.~In fact, the pressure closure of the PW equations $P(\rho)$ is poorly approximated by the GoRINNs with $\mathcal{N}(\rho,q)$; see \cref{fig:GNN_PW2}d.~As a result, the numerical solution of \cref{eq:unPW} is not accurate, as shown in \cref{fig:GNN_PW2}a,b,c, especially near the shock wave where significant numerical errors accumulate.

This discrepancy occurs because the Roe linearized matrix $\hat{\mathbf{A}}_{i-1/2}(\mathbf{q}_l,\mathbf{q}_r)$ for the PW equations in \cref{eq:PW_Roe1} does not result from the Jacobian $\partial_\mathbf{u} \mathbf{f}(\bar{\mathbf{q}})$ computed at the average state $\bar{\mathbf{q}}=[\bar{\rho},\bar{q}]^\top$; see discussion in \Cref{app:RoePayne}.~In all other benchmark problems, the Roe linearized matrix $\hat{\mathbf{A}}_{i-1/2}(\mathbf{q}_l,\mathbf{q}_r)$ has this property and, thus, GoRINNs succeed to solve the inverse problem with a high accuracy.

\paragraph{When GoRINNs do not converge} Clearly, when no information is available about the hyperbolic system from which the data come from, such an issue can only be captured by the incapability of GoRINNs to converge.~The latter provides a hint on wrong predictor selection.~We thus attempt to find desired pressure closure in \cref{eq:unPW} as a function of only $\rho$ and not $q$; i.e., $\mathcal{N}(\rho)$.

By considering $\mathcal{N}(\rho)$ in \cref{eq:unPW}, GoRINNs are conservative when \cref{eq:Roe_NN1,eq:Roe_NN2} in \Cref{prop1} are satisfied.~Since no change is employed to the known functions, the average state $\bar{q}$ is again provided by \cref{eq:PWutRoe1}.~However, requiring \cref{eq:Roe_NN2} to hold for the new $\mathbf{f}_U([\rho,q]^\top)=[0,\mathcal{N}(\rho)]^\top$, implies:
\begin{equation}
    \partial_\rho \mathcal{N}(\bar{\rho}) (\rho_r-\rho_l) = \mathcal{N}(\rho_r) - \mathcal{N}(\rho_l) \Rightarrow \partial_\rho \mathcal{N}(\bar{\rho}) = \dfrac{\mathcal{N}(\rho_r) - \mathcal{N}(\rho_l)}{\rho_r-\rho_l}.
    \label{eq:PWxx}
\end{equation}
The latter equality is similar to the one of the PW equation in \cref{eq:PW_Roe1} and, in fact, allows us to substitute the above expression of $\partial_\rho \mathcal{N}(\bar{\rho})$ into the computation of the Roe linearized matrix of GoRINNs.

~As a result, the RH condition is explicitly satisfied, and thus the GoRINNs optimization problem in \cref{eq:OptGen} can be solved without including the loss function term $\mathcal{L}_{RH}$ in \cref{eq:Opt2}, and without using the arithmetic average $\bar{\rho}=(\rho_r+\rho_l)/2$.~The training results, derived by solving this simplified GoRINNs optimization problem, are reported in \cref{tb:Errors}; see row denoted as ``PW, $P(\rho)$''.~It is therein clearly shown that considering a density-dependent pressure closure makes the GoRINNs optimization problem converge to very small training/validation/testing errors.~In fact, all errors are of magnitude, as low as the errors for the other benchmark problems.

The numerical accuracy of GoRINNs is visualized in \cref{fig:GNN_PW1}.
\begin{figure}[!h]
    \centering
    \subfigure[$\lvert \hat{\rho}(t,x)-\rho(t,x)\rvert$]{\includegraphics[width=0.32\textwidth]{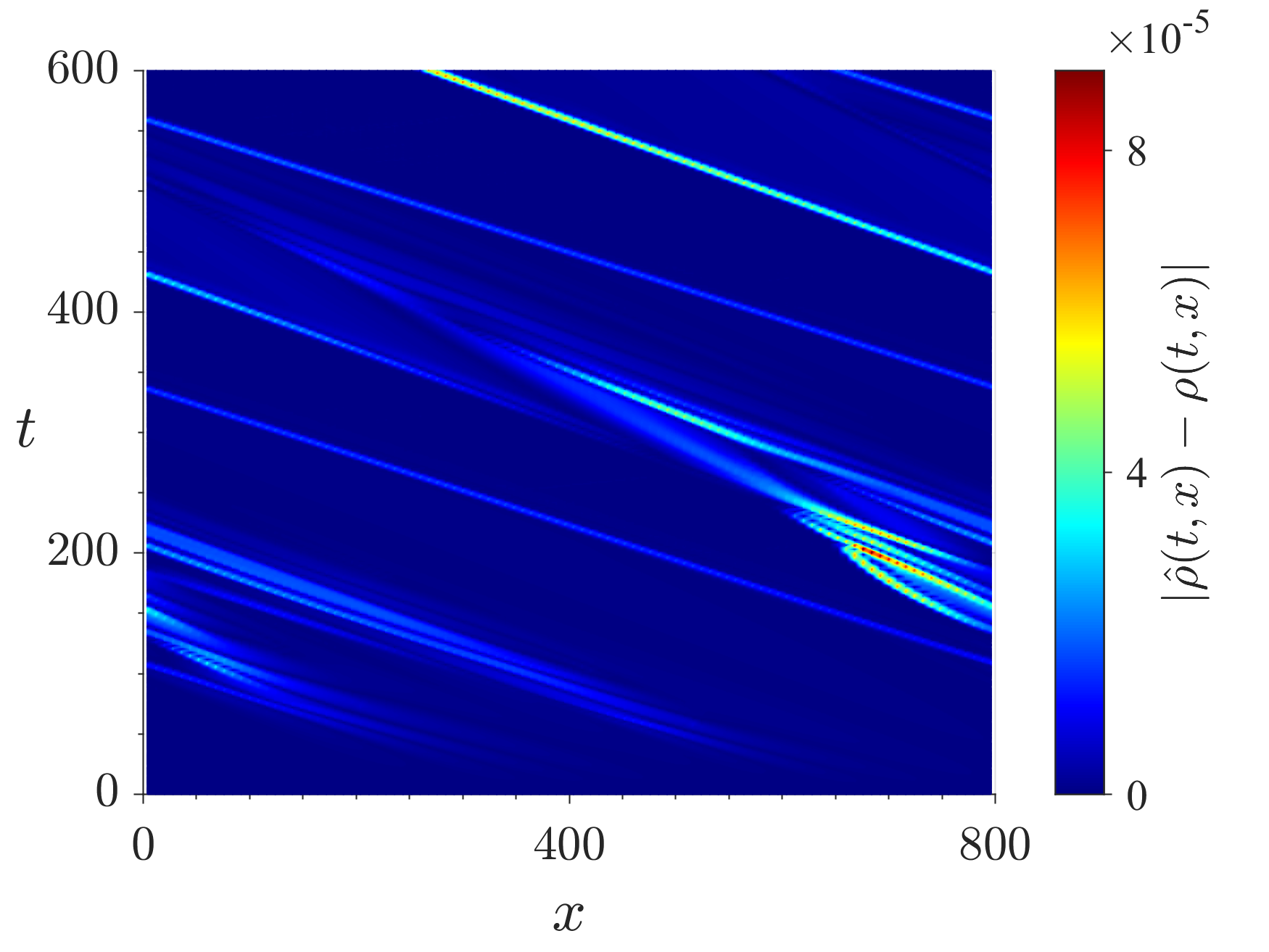}}
    \subfigure[$\lvert \hat{q}(t,x)-q(t,x)\rvert$]{\includegraphics[width=0.32\textwidth]{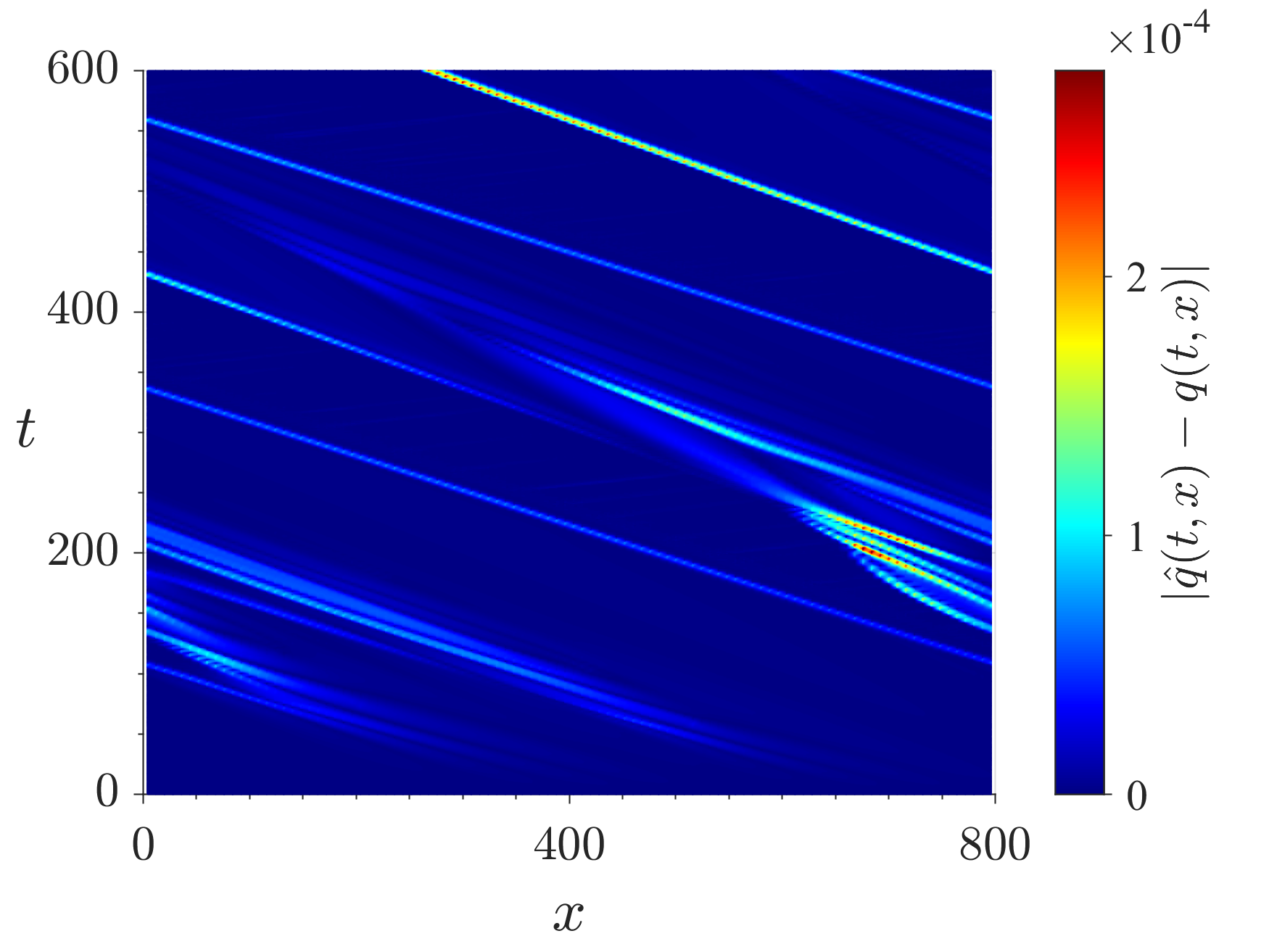}}
    \subfigure[$P(\rho)$ vs GoRINNs $\mathcal{N}(\rho)$]{\includegraphics[width=0.32\textwidth]{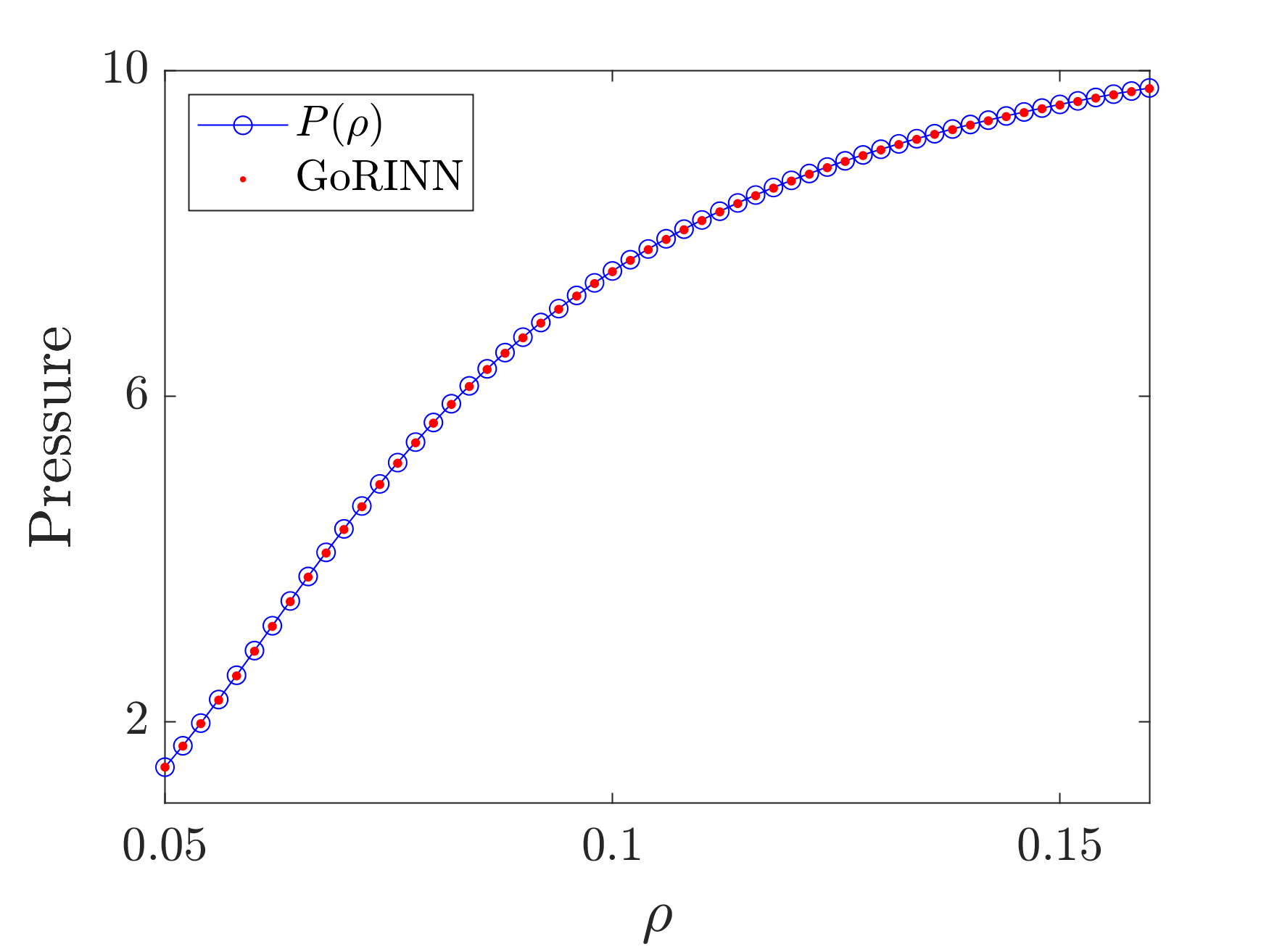}}
    \caption{Numerical accuracy of GoRINNs, assuming a pressure closure of the form $\mathcal{N}(\rho)$, for the PW equations.~(a,b) Absolute errors of the numerical solution provided by the GoRINNs learned equations in \cref{eq:unPW} for $\mathcal{N}(\rho)$ ($\hat{\rho}$ and $\hat{q}$) vs the one provided by PW equations in \cref{eq:PWcon} ($\rho$ and $q$).~(c) Analytically known pressure closure $P(\rho)$ in \cref{eq:PWpres} vs the $\mathcal{N}(\rho)$ functional learned with GoRINNs.}
    \label{fig:GNN_PW1}
\end{figure}
\Cref{fig:GNN_PW1}a,b display the absolute errors between the numerical solution of \cref{eq:unPW} with the learned by the GoRINNs density-dependent flux functional $\mathcal{N}(\rho)$ and the one computed from the PW equations in \cref{eq:PWcon}; the latter numerical solution is displayed in \cref{fig:solsSW_PW}d,e.~Very high numerical accuracy is provided by the GoRINNs  for both $\rho$ and $q$ variables, even near the shock wave and rarefaction regions.~In these regions, the error was 2-3 orders higher when considering GoRINNs for a density $\mathcal{N}(\rho,q)$; see \cref{fig:GNN_PW2}a,b.~Finally, the excellent agreement of the analytically known non-linear pressure closure in \cref{eq:PWpres}, $P(\rho)$, and the GoRINNs learned functional, $\mathcal{N}(\rho)$ is shown in \cref{fig:GNN_PW1}c.

\section{Conclusions}
\label{sec:Con}
Hyperbolic partial differential equations (PDEs) that model conservation laws play a central role in explaining complex phenomena across diverse disciplines such as fluid mechanics, aerospace and aerodynamics engineering, acoustics, astrophysics, electromagnetics, traffic flow, and crowd dynamics, to name just a few. The numerical solution of the forward problem for such PDEs is challenging due to the presence of discontinuities, such as shock waves and rarefactions, which can cause classical numerical methods to fail or yield inaccurate solutions. The inverse problem —especially that of learning the physical closures/potentials — poses even greater difficulties because of the inherent nonlinearities, ill-posedness, and emerging discontinuities. Thus, relatively small approximation errors can lead to instabilities and violation of the conservation laws. \par 
In the last few years, SciML based on DNNs, such as PINNs \cite{pang2019fpinns,yang2020physics,meng2020composite,mao2020physics,patel2022thermodynamically,jagtap2020conservative} and neural operators  \citep{lu2021learning,li2020fourier,wen2022u,kovachki2023neural,thodi2024fourier}, have shown great potential in dealing with the challenging inverse problem of parameter inference
in conservation laws. One critical issue is the lack of explainability \cite{peyvan2024riemannonets,doncevic2024recursively,fabiani2024randonet} and the computational complexity pertaining to their training, thus results to rather moderate approximation accuracy (see also the discussion and some demonstrations in \cite{fabiani2024randonet}). By construction, most of the until now proposed schemes, set the conservation law as a soft constraint in the loss function, thus failing to respect accurately the conservation property. By explicitly preserving the conservation laws, hybrid SciML methods infer fluxes of conservative FV schemes \cite{morand2024deep,chen2024learning,kim2024approximating}, which however provide representation of the numerical-approximated quantities rather than the physical ones per se.~These methods may introduce significant inaccuracies for coarsely discretized data and are developed for scalar conservation laws. 

GoRINNs offer explainable ``intelligently'' designed shallow feedforward NNs for tackling the above problems. Unlike other hybrid SciML methods that approximate the numerical fluxes using DNNs, GoRINNs solve the inverse problem enhancing explainability, and accuracy, thus respecting explicitly the underlying conservation law as approximate Riemann solvers that satisfy the Rankine-Hugoniot condition. This approach accommodates systems of conservation laws and allows for the inclusion of source terms—whether known or unknown—in the associated hyperbolic PDEs. As we show, GoRINNs are able to approximate with high accuracy the closures of four benchmark problems, which exhibit traveling shocks that emerge at finite times. 
We show that GoRINNs result to a high numerical approximation accuracy, preserving numerical stability and the conservation laws. 

The training and performance of GoRINNs, can also incorporate partial knowledge of the physics underlying the missing terms (see \cref{sb:PW_IP}, where GoRINNs converge with the ``correct'' predictor selection).~If no physical insight is available, then techniques such as symbolic regression \cite{udrescu2020ai,billard2002symbolic}, manifold learning (e.g., parsimonious diffusion maps) \cite{coifman2006diffusion,dsilva2018parsimonious,della2024learning,galaris2022numerical,papapicco2022neural}, and SINDy \cite{brunton2016discovering,champion2019data} can be exploited to learn a parsimonious set of variables that can parametrize the unknown closures.~Furthermore, GoRINNs can be extended to analyze real-world data derived from microscopic simulations/observations, facilitating the discovery of macroscopic hyperbolic PDEs for the emergent dynamics.~This application will further require parameter inference, a task that can be addressed by including physical parameters as additional GoRINNs predictors.~Finally, using GoRINNs to discover macroscopic closures connects directly to control theory, which can be exploited for enabling distributed control in large-scale microscopic models (see e.g., \citep{maffettone2022continuification,maffettone2024mixed}), given GoRINNs-derived information at the macroscopic level.

Importantly, what we propose and highlight here is the need of developing new SciML resource-bounded algorithms that can provide high accuracy at a low computational cost. That is develop new ``intelligently'' designed  schemes which are not only physics-informed but also numerical analysis-informed in order to deal with the time and space complexity for training DNNs. Towards this aim, the recently introduced blended inverse-PDE networks (BiPDE-Nets) \cite{pakravan2021solving}, first discover unknown terms of PDEs and in a latter stage carry out numerical operations similar to those in traditional solvers. 
In \cite{mistani2023jax}, the authors show that ``accurate'' optimization of shallow NNs can yield superior results compared to training DNNs for PDEs. Recently, RandONets \cite{fabiani2024randonet} have been introduced to learn linear and nonlinear operators based on random projections, shallow NNs, and iterative linear algebra methods for (large-scale) ill-posed problems. There, it has been shown that RandONets outperform DNNs by several orders of magnitude both in terms of computational cost and accuracy, reaching also machine-precision accuracy for linear operators. 
Our hybrid/blended scheme presented here, contributes exactly towards this important direction of research in SciML, focusing on the solution of the inverse problem for hyperbolic conservation laws with important applications across many fields.

\bibliographystyle{apalike}
\bibliography{references}  

\clearpage
\newpage
\section*{APPENDIX}
\appendix
\renewcommand{\theequation}{A.\arabic{equation}}
\renewcommand{\thefigure}{A.\arabic{figure}}
\setcounter{equation}{0}
\setcounter{figure}{0}
\section{Benchmark problems and data acquisition}
\label{app:Bench}

Here, we present the four benchmark problems upon which the performance of GoRINNs is evaluated on, namely the inviscid Burgers' equation, the LWR equation of traffic flow, the SW equations of fluid dynamics, and the PW equations of traffic flow.~We additionally provide the details for the solution of the forward problem with the high-resolution Godunov-type FV scheme in \cref{sb:FP_god}, which are required for the collection of training/validation/testing data sets.

\subsection{Burger's equation}
\label{sb:B_FP}
The inviscid Burgers' equation is one of the simplest problems of non-linear scalar conservation laws.~Considering a homogeneous PDE with the non-linear flux term $f(u)=u^2/2$, the Burgers' equation is written in the form of \cref{eq:genPDE} as:
\begin{equation}
    \partial_t u + \partial_x (u^2/2) = 0. \label{eq:Burg}
\end{equation}
Due to the simplicity of the Burgers' equation, the Riemann problem can be solved exactly \citep{leveque2002finite,ketcheson2020riemann}.~Here, we solve the Burgers' equation numerically in the 1-dim. spatial domain $x\in[-1,1]$ with periodic boundary conditions, along the time interval $t\in[0,3]$.~In particular, we discretize the spatial domain in $N=100$ cells and consider a time step $dt=0.005$ that satisfies the CFL condition.~As already discussed in \cref{sb:imp}, for generating the numerical solutions, 4 different initial conditions were considered, all of which are following Gaussian profiles, with $u(0,x)=\mu ~exp(-x^2/(2 \sigma^2))$ for a constant $\sigma=0.2$ and a uniformly varying $\mu\in[1,2]$ with $d\mu=1/3$.

In spite of the smooth initial conditions, contact discontinuities are formed within the selected time interval, in particular shock waves and rarefactions; the latter are not transonic due to the positive initial conditions.~Thus, for the numerical solution of the forward problem via the Godunov-scheme in \cref{sb:FP_god}, it is sufficient to employ the Roe solver, which leads to the scalar, in this case, Roe matrix $\hat{a}_{i-1/2} = \bar{u}$, where $\bar{u}=(u_l+u_r)/2$; the derivation is provided in \cref{app:RoeBurg}.~As already discussed in \cref{sb:imp}, the Van-Leer flux-limiter function is employed.~The resulting numerical solution of the Burgers' equation for the Gaussian initial data with $\mu=2$ and $\sigma=0.2$ is depicted in \Cref{fig:solsBR_LWR}a,b, where it is clearly shown that the initial bump quickly forms a right going shock wave, followed by a rarefaction spread out towards the left; for comparison to the analytic solution, see Chapter 4 of \cite{ketcheson2020riemann}.~From the 4 numerical solutions of the Burgers' equation, we randomly split the 15-15-70\% of the time-series data to form the training/validation/testing data sets of the GoRINNs, as discussed in \cref{sb:imp}; the training samples collected from the solution with initial data with $\mu=2$ and $\sigma=0.2$ are denoted with black circles in \Cref{fig:solsBR_LWR}a.
\begin{figure}[!h]
    \centering
    \subfigure[Burgers' equation, $u(t,x)$]{\includegraphics[width=0.45\textwidth]{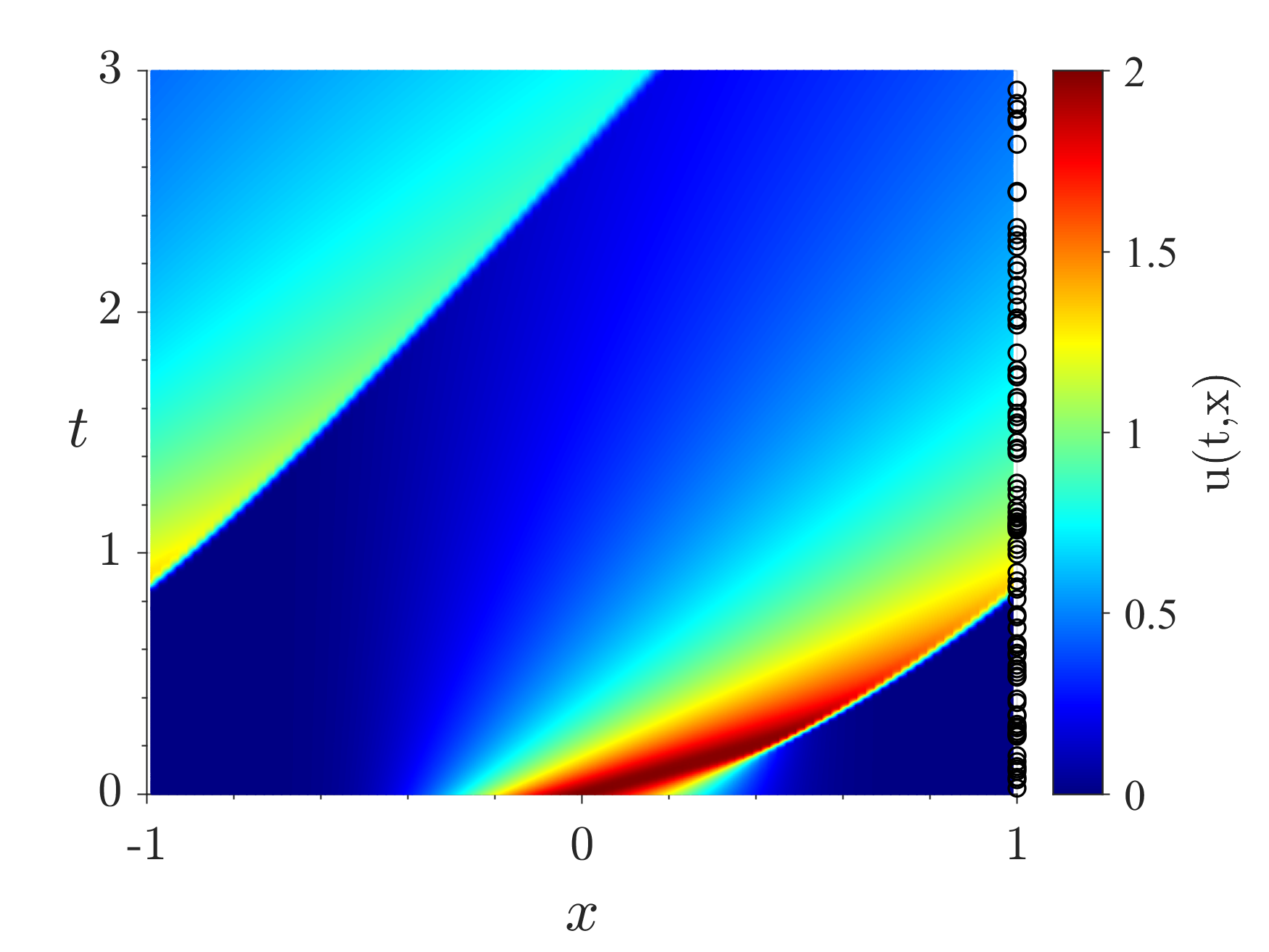}}
    \subfigure[Burgers' equation, $u(0,x)$ and $u(3,x)$]{\includegraphics[width=0.45\textwidth]{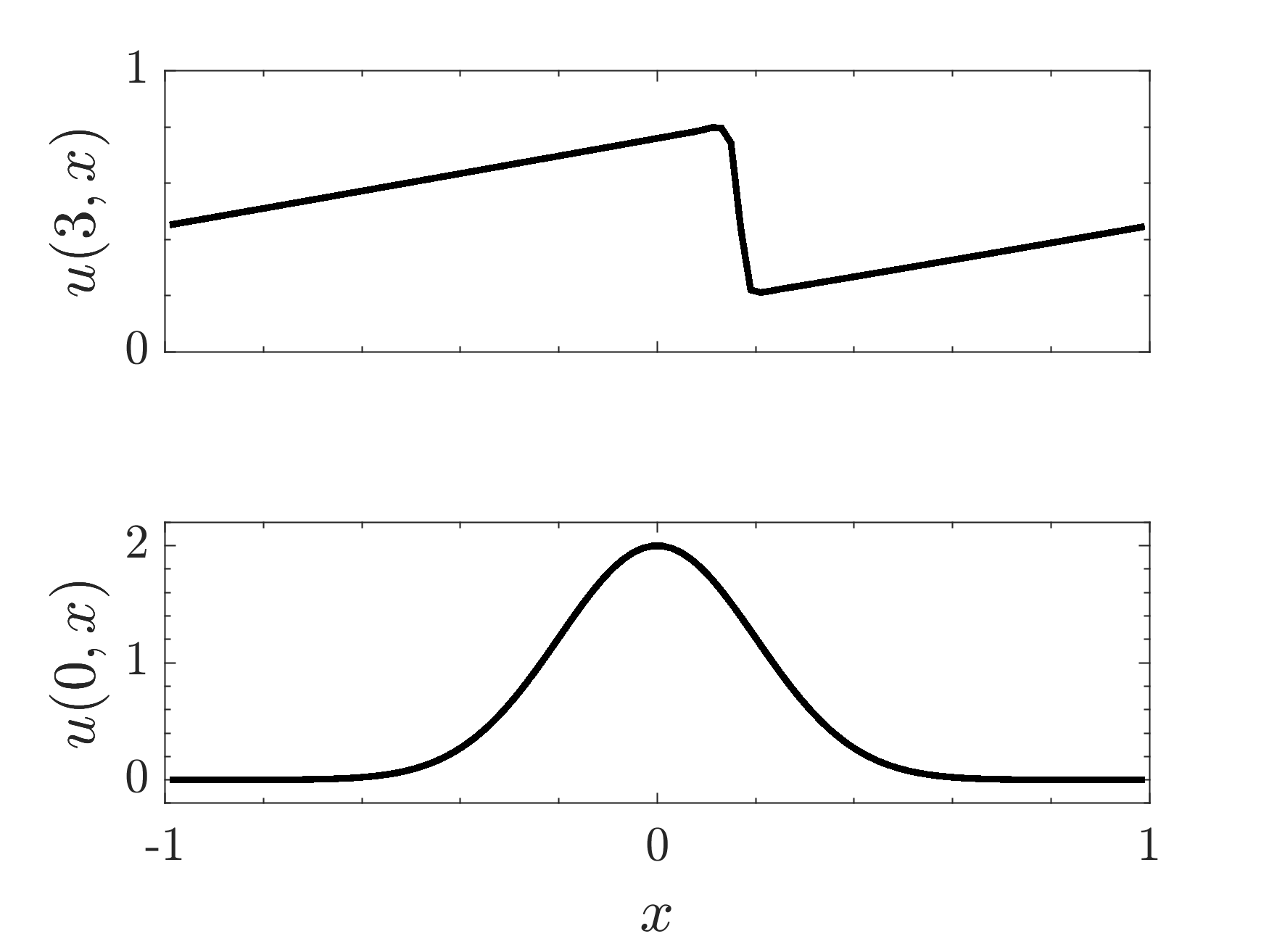}} \\
    \subfigure[LWR equation, $\rho(t,x)$]{\includegraphics[width=0.45\textwidth]{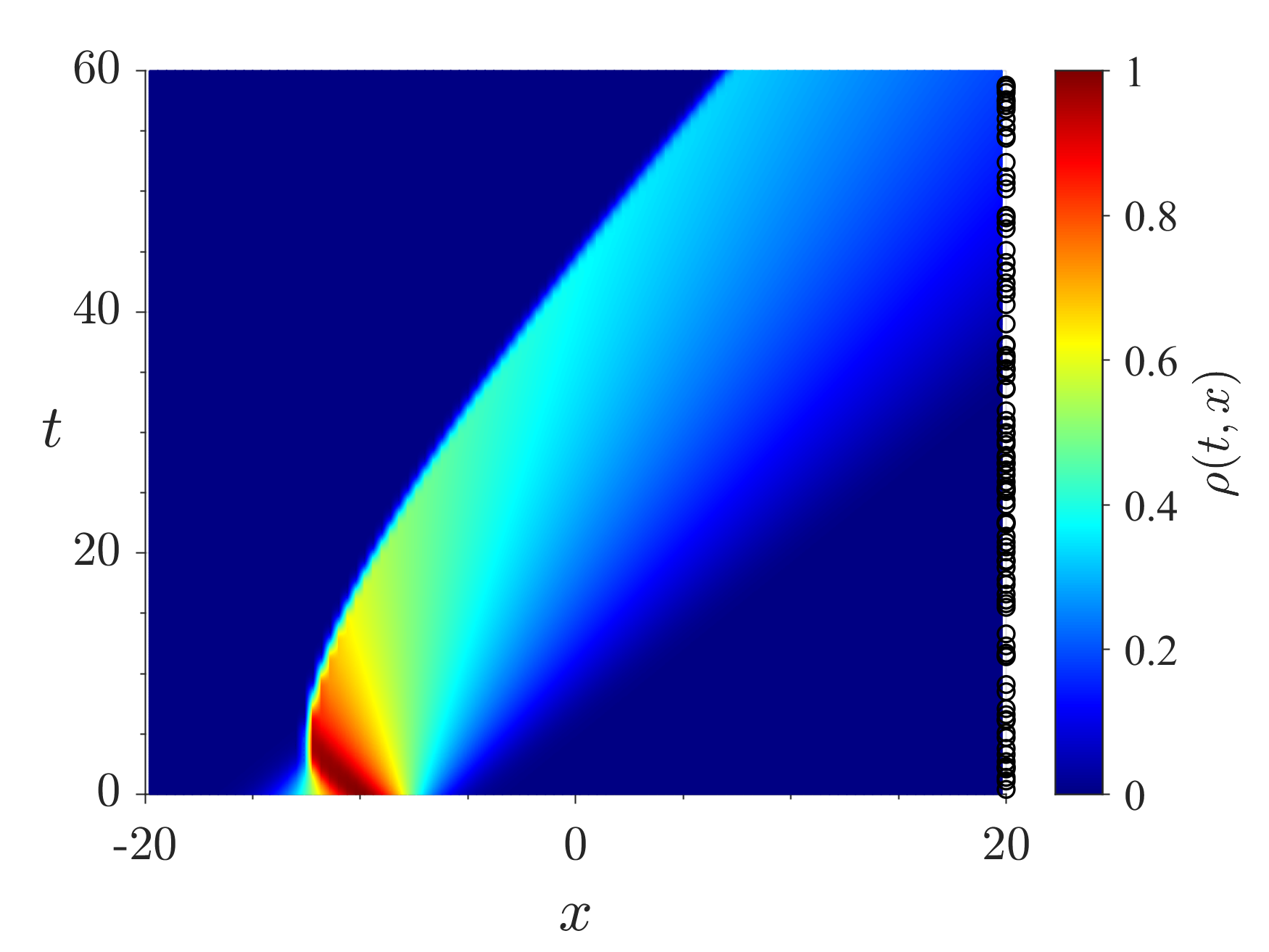}}
    \subfigure[LWR equation, $\rho(0,x)$ and $\rho(60,x)$]{\includegraphics[width=0.45\textwidth]{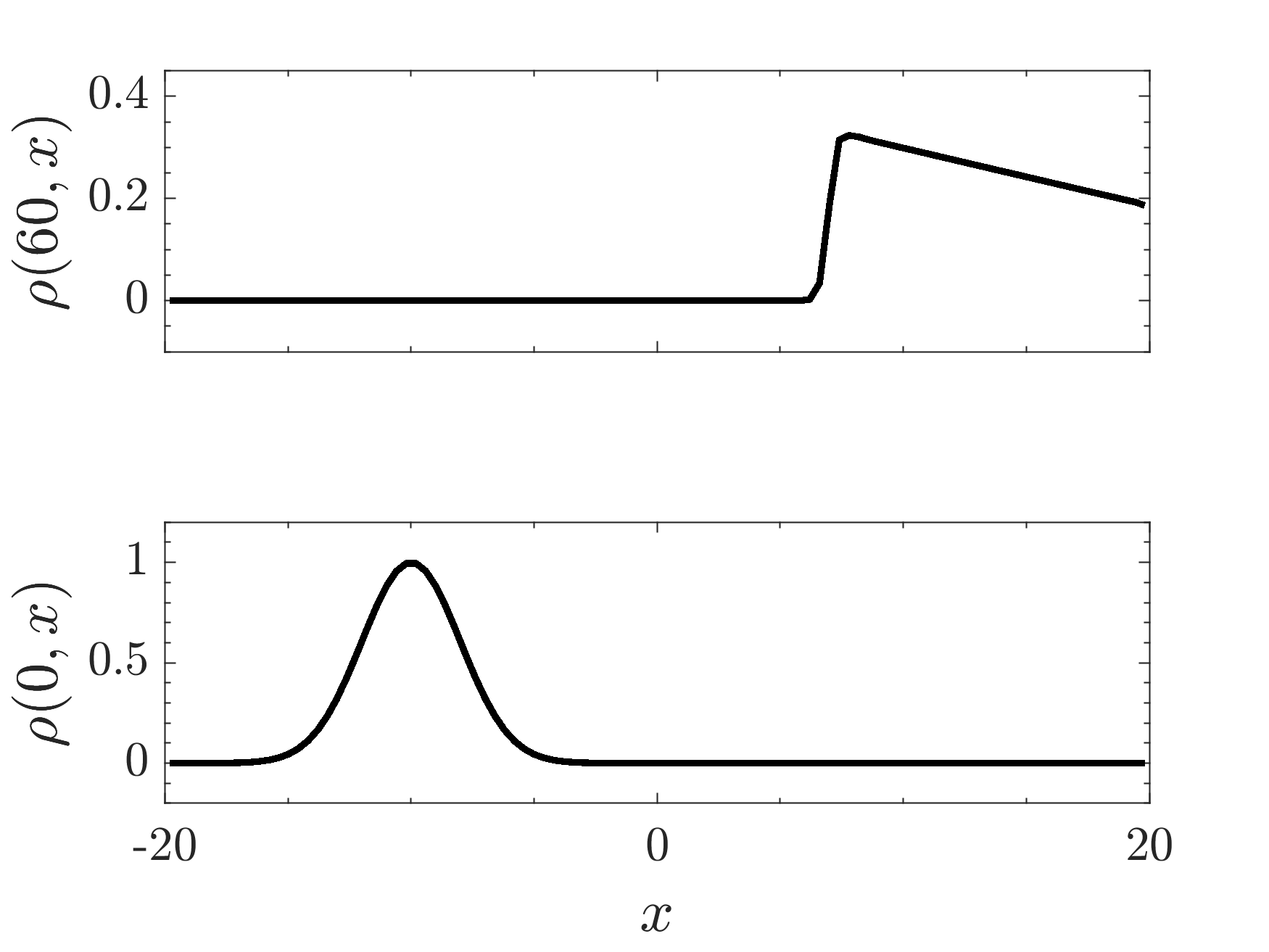}} 
    \caption{Numerical solution of the Burgers' and LWR equations obtained with the high-resolution Godunov-scheme in \cref{sb:FP_god} for (a,b) $x\in[-1,1]$, $t\in[0,3]$, with the initial data $u(0,x)=2 ~exp(-x^2/(2\cdot0.2^2))$ shown in panel (b), and (c,d) $x\in[-20,20]$, $t\in[0,60]$ with the $\rho(0,x)=exp(-(x+10)^2/(2\cdot2^2))$ shown in panel (d).~Periodic boundary conditions are used for the Burgers' equation and outflow ones for the LWR equation.~The black circles on (a,c) denote the time steps $n_t$, collected from the depicted numerical solutions, to form the training set of GoRINNs.}
    \label{fig:solsBR_LWR}
\end{figure}

\subsection{LWR equation of traffic flow}
\label{sb:LWR_FP}
The LWR equation \cite{lighthill1955kinematic,richards1956shock} is a non-linear scalar conservation law describing the evolution of the macroscopic density $\rho(t,x)$ of cars in a road.~The traffic flow is modelled by the LWR equation in the form of \cref{eq:genPDE} as:
\begin{equation}
    \partial_t \rho + \partial_x  \left(\rho v(\rho)\right) = 0 \label{eq:LWR},
\end{equation}
where the car density $\rho(t,x) \in [0,1]$ is moving with a density-dependent velocity:
\begin{equation}
    v(\rho) = v_{max} (1-\rho),
\end{equation}
where $v_{max}>0$ is the maximum velocity, implying that $v(\rho) \in [0,v_{max}]$.

Similarly to the Burgers' equation, the Riemann problem for the LWR equation can be solved exactly \citep{leveque2002finite,ketcheson2020riemann}.~Here, we solve the LWR equation in \cref{eq:LWR} numerically in the 1-dim. spatial domain $x\in[-20,20]$ with outflow boundary conditions (allowing cars to exit the domain), along the time interval $t\in[0,60]$.~In particular, we set $v_{max}=0.7$ and consider  $N=100$ cells and a time step $dt=0.1$ that satisfies the CFL condition.~As discussed in \cref{sb:imp}, for generating the numerical solutions, 4 different initial conditions were considered, again following  Gaussian profiles, this time forming an initial bump around $x_0=-10$; that is, we set $\rho(0,x)=\mu ~exp(-(x-x_0)^2/(2 \sigma^2))$ for a constant $\mu=1$ and a uniformly varying $\sigma\in[1,2.5]$ with $d\sigma = 0.5$.

Similarly to the Burgers' equation, in spite of the smooth initial conditions, contact discontinuities are formed within the selected time interval.~For the numerical solution via the Godunov-scheme in \cref{sb:FP_god}, it is sufficient to employ the Roe solver.~The scalar, for the LWR equation, Roe matrix takes the form $\hat{a}_{i-1/2} = v_{max} (1-2\bar{\rho})$, where $\bar{\rho}=(\rho_l+\rho_r)/2$; see the derivation in \cref{app:RoeLWR}.~The resulting numerical solution of the LWR equation, using the Van-Leer flux-limiter function, for the Gaussian initial conditions with $\sigma=2$ is depicted in \Cref{fig:solsBR_LWR}c,d.~In contrast to the Burgers' equation, the non-linear flux term of the LWR equation $f(\rho)=v_{max}\rho (1-\rho)$ is not convex, but concave.~Because of that, as shown in \Cref{fig:solsBR_LWR}c, the initial smooth profiles quickly form a right going shock wave, which is now following the right-spreading rarefaction (instead of preceding it, as shown in \Cref{fig:solsBR_LWR}a for the Burgers' equation).~This behavior is well aligned with the relaxation of traffic jams reported for agent-based car-following models; see Chapter 11 of \cite{leveque2002finite} for a direct comparison.~From the 4 numerical solutions of the LWR equation, we randomly split the 15-15-70\% of the time-series data to form the training/validation/testing data sets of the GoRINNs, as discussed in \cref{sb:imp}; the training samples collected from the solution with initial data with $\mu=1$ and $\sigma=2$ are denoted with black circles in \Cref{fig:solsBR_LWR}c.

\subsection{The SW equations}
\label{sb:SW_FP}
The SW equations is a system of $D=2$ non-linear conservation laws describing the flow of an incompressible and inviscid fluid below a pressure surface.~The conservation of mass and momentum in 1 dimension is then modeled by the SW equations in the form of \cref{eq:genPDE} as:
\begin{align}
    \partial_t \begin{bmatrix}
        h \\ q 
    \end{bmatrix} + \partial_x \begin{bmatrix}
        q \\ \dfrac{q^2}{h} + P(h)
    \end{bmatrix} = \begin{bmatrix}
        0 \\ 0 
    \end{bmatrix} \label{eq:SWcon}
\end{align}
where $h(t,x)$ and $q(t,x)$ are the depth and momentum conservative state variables.~The pressure is assumed to be hydrostatic, thus taking the depth-dependent form $P(h)= (g h^2)/2$ where $g>0$ is a rescaled parameter of the gravitational constant.

Although the Riemann problem for the SW equations can be solved exactly \citep{ketcheson2020riemann,leveque2002finite}, here we numerically solve the system in \cref{eq:SWcon} in the 1-dim. spatial domain $x\in[-5,5]$ with periodic boundary conditions, along the time interval $t\in[0,3]$.~In particular, we set $g=1$ and consider $N=200$ cells and a, CFL satisfying, time step $dt=0.01$.~For generating the data, we considered 4 initial conditions, following again Gaussian profiles for $h(0,x)$; we set $h(0,x)=\mu ~exp(-x^2/(2 \sigma^2))$ for a constant $\mu=0.5$ and a uniformly varying $\sigma\in[0.2,0.8]$ with $d\sigma = 0.2$, and $q(0,x)=0$.

Starting from the aforementioned smooth initial data, the SW equations exhibit contact discontinuities within the selected time interval.~Since they are not transonic rarefactions, but only shock waves and rarefactions, it is sufficient to employ the Roe solver for obtaining the  numerical solution via the Godunov-scheme in \cref{sb:FP_god}; for transonic rarefactions the HLLE solver can be employed instead.~The Roe linearized matrix for the SW equations is: 
\begin{equation}
    \hat{\mathbf{A}}_{i-1/2} = \begin{bmatrix}
        0 & 1 \\ -\dfrac{\bar{q}^2}{\bar{h}^2} + g \bar{h} & 2\dfrac{\bar{q}}{\bar{h}}
    \end{bmatrix}, \qquad \text{where} \qquad \bar{h}=\dfrac{h_l+h_r}{2},\qquad \bar{q}=\bar{h} \dfrac{q_l h_l^{-1/2} + q_r h_r^{-1/2}}{h_l^{1/2}+h_r^{1/2}},
    \label{eq:SW_Roe1}
\end{equation}
the latter being the so-called \textit{Roe average}.~For the derivation of $\hat{\mathbf{A}}_{i-1/2}$, which is presented in detail in \cref{app:RoeSW}, we followed the general approach introduced by \cite{roe1981approximate}, resulting in the same Roe matrix as the ones derived in \citep{leveque2002finite,ketcheson2020riemann}.~The resulting numerical solution of the SW equations, using the Van-Leer flux-limiter function, for the Gaussian initial data with $\sigma=0.4$ is depicted in \Cref{fig:solsSW_PW}a,b,c for $h$ and $q$, respectively.~It is therein shown that after a fast transient period, the initially smooth data form one left and one right going shock waves, which are followed by rarefactions spreading towards the opposite direction of the shock wave movement.~In fact, $h$ is symmetric along the spatial domain, while $q$ is antisymmetric; see \Cref{fig:solsSW_PW}c.~From the 4 numerical solutions of the SW equations, we randomly split the 15-15-70\% of the time-series data to form the training/validation/testing data sets of the GoRINNs, as discussed in \cref{sb:imp}; the training samples collected from the solution with initial data with $\mu=0.5$ and $\sigma=0.4$ are denoted with black circles in \Cref{fig:solsSW_PW}a,b.
\begin{figure}[!h]
    \centering
    \subfigure[SW equations, $h(t,x)$]{\includegraphics[width=0.32\textwidth]{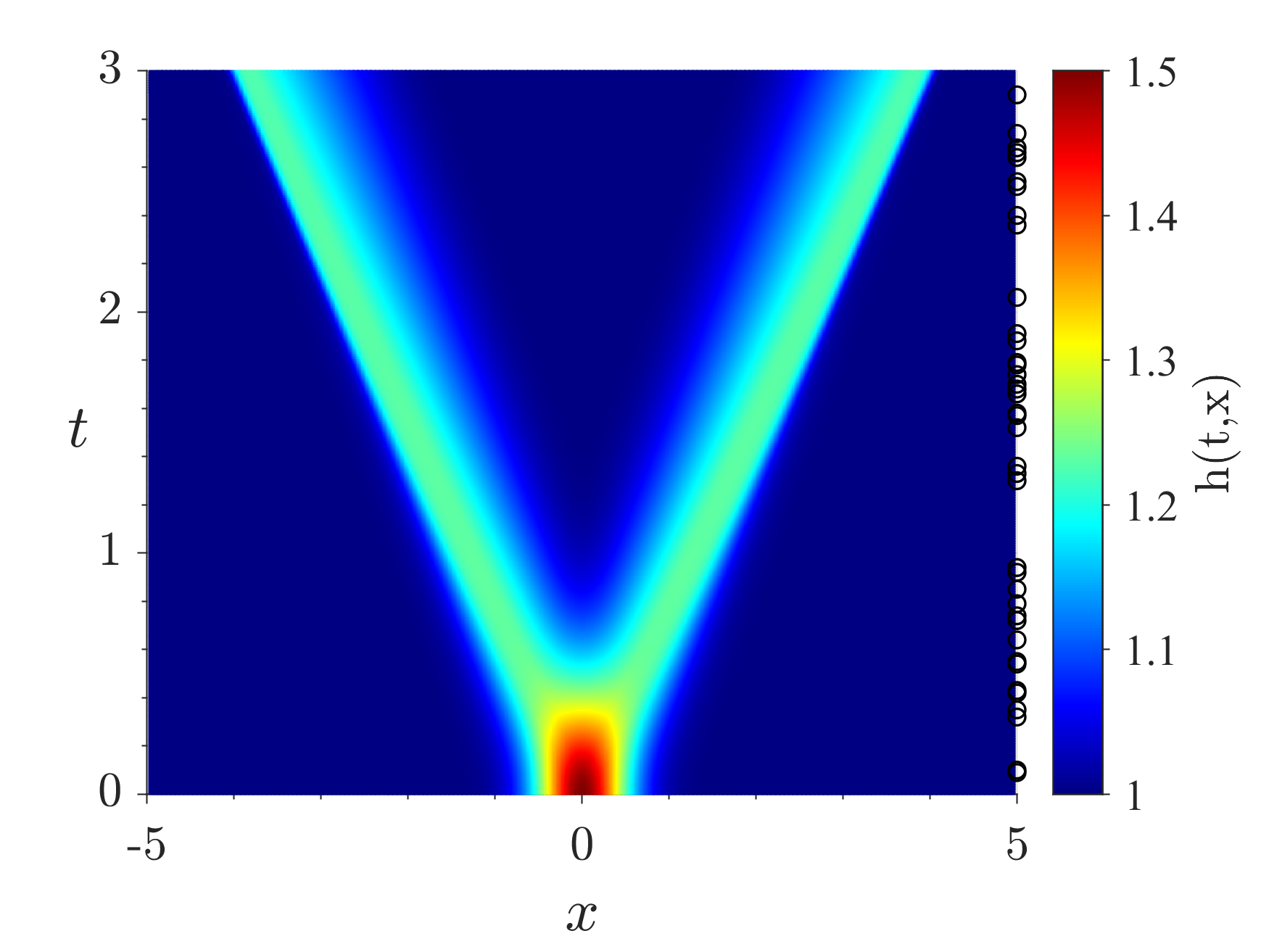}}
    \subfigure[SW equations, $q(t,x)$]{\includegraphics[width=0.32\textwidth]{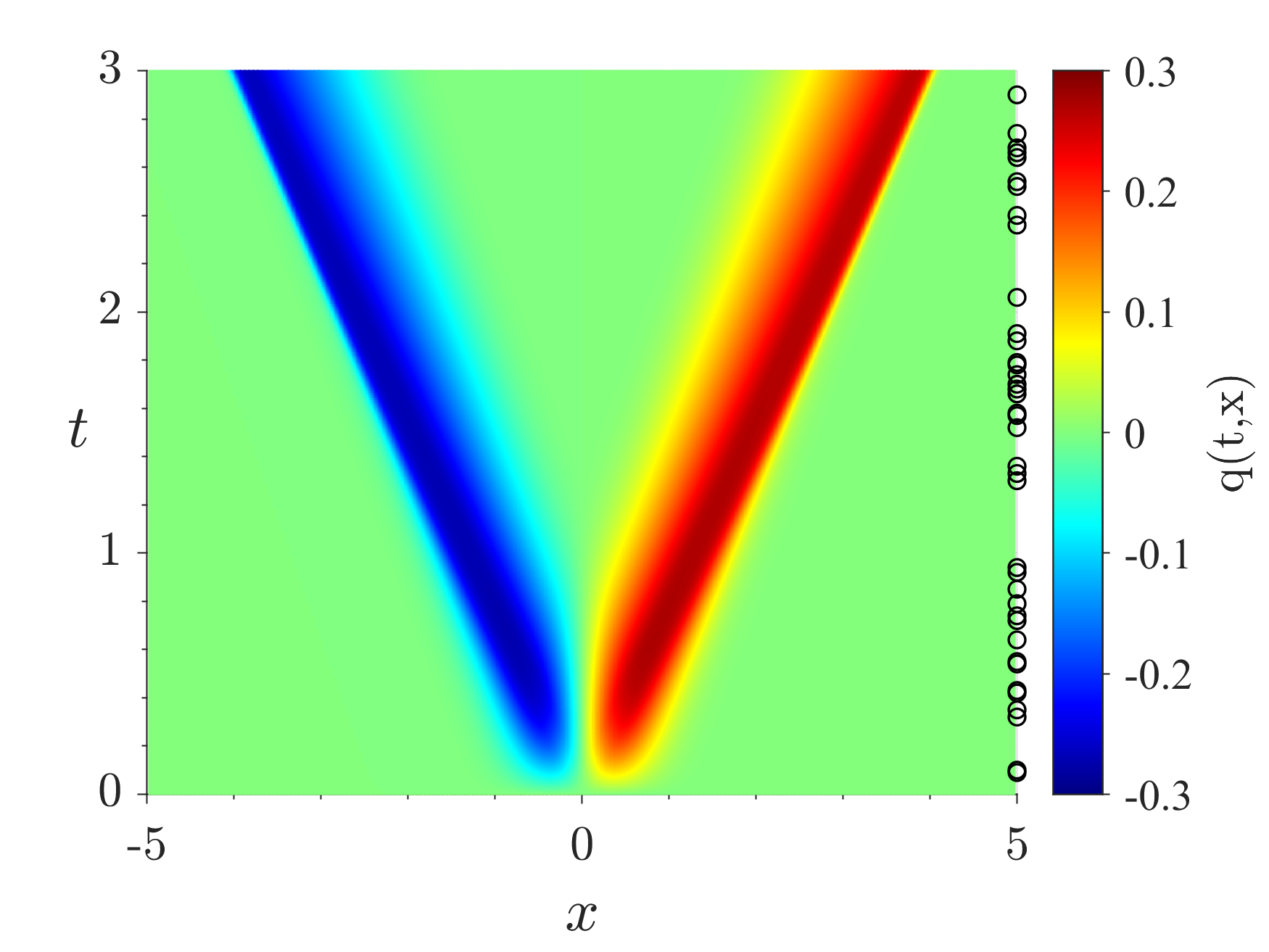}}
    \subfigure[SW equations, $h(0,x)$, $q(0,x)$, $h(3,x)$ and $q(3,x)$]{\includegraphics[width=0.32\textwidth]{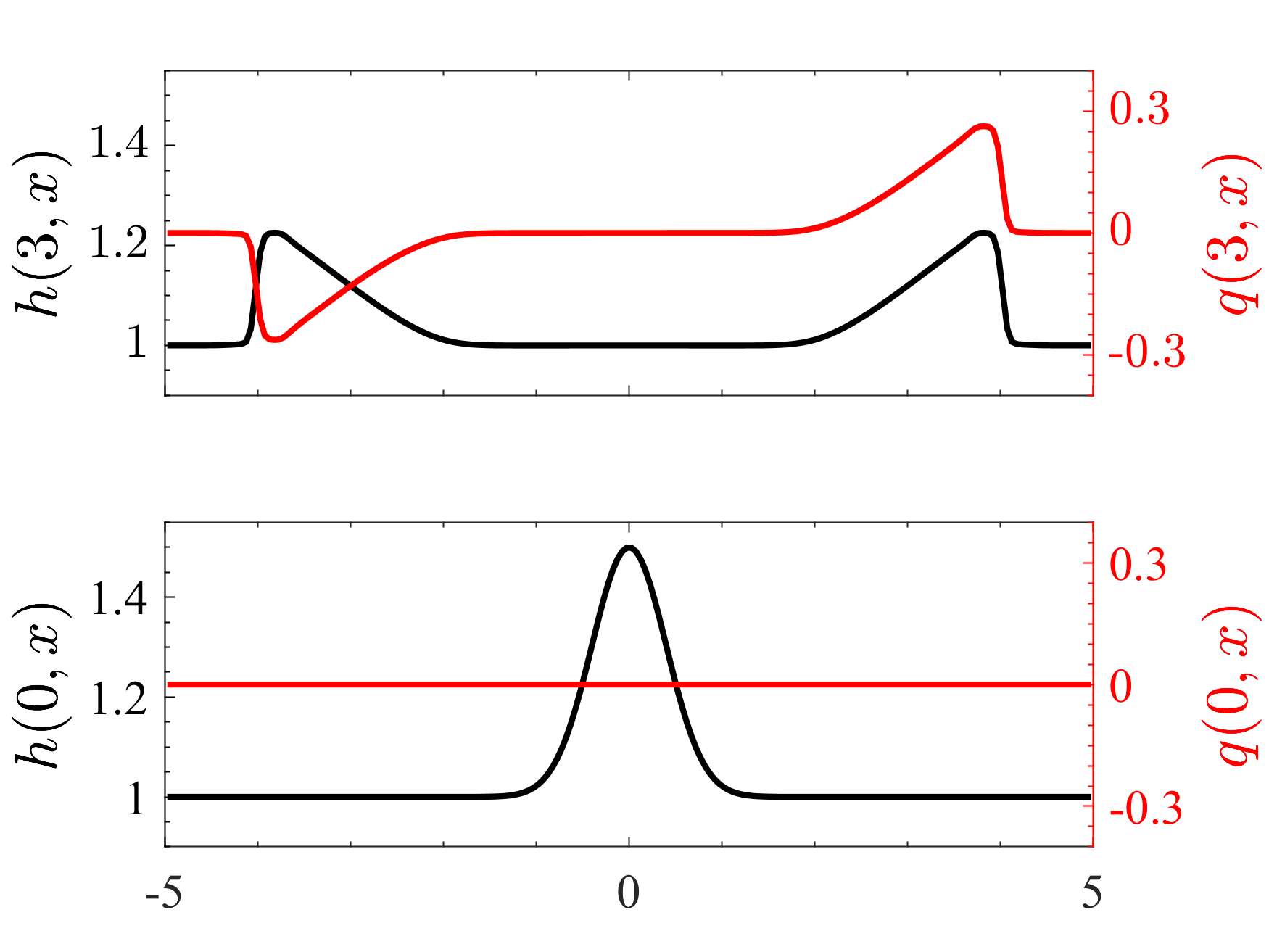}} \\
    \subfigure[PW equations, $\rho(t,x)$]{\includegraphics[width=0.32\textwidth]{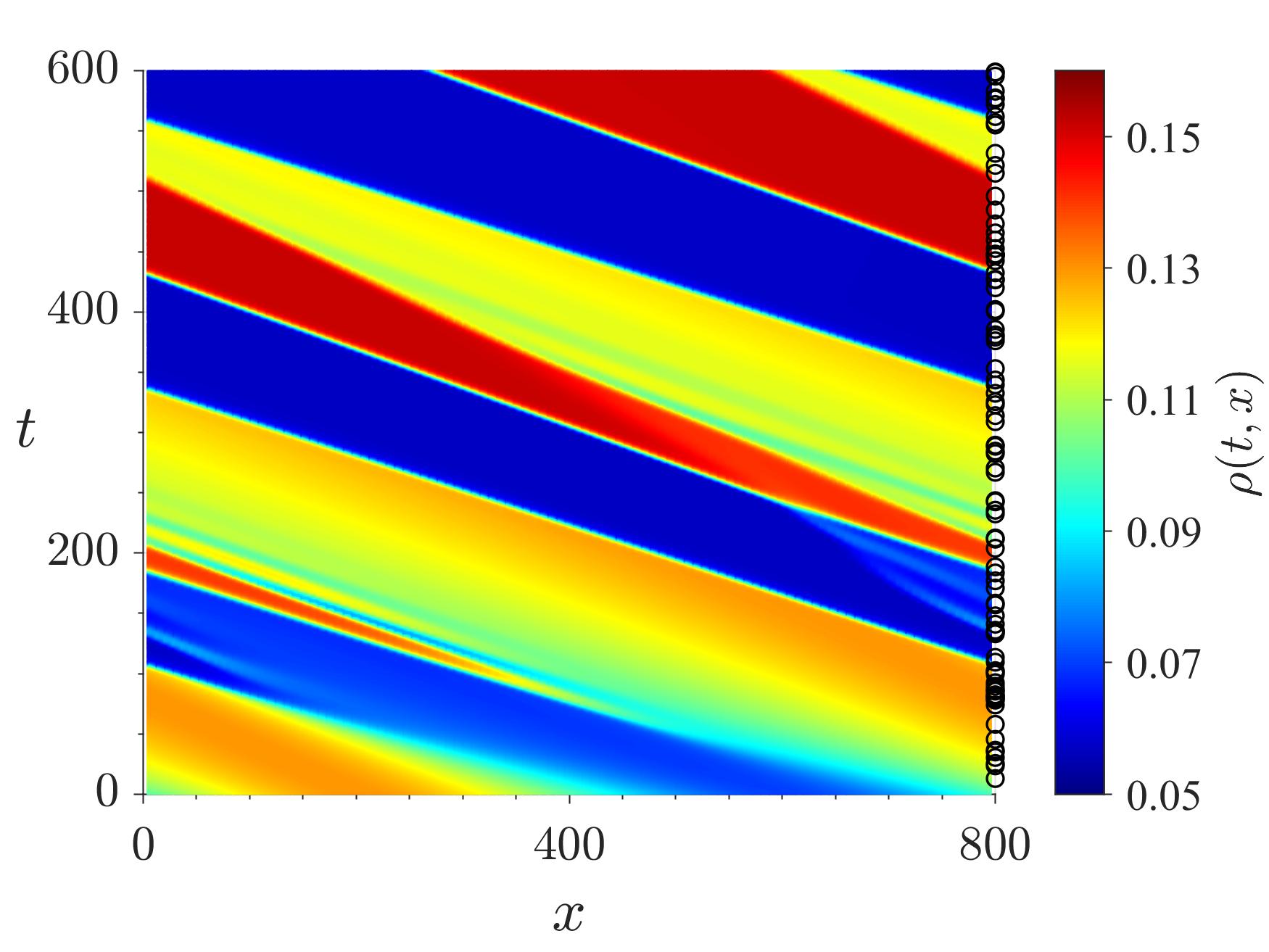}}
    \subfigure[PW equations, $q(t,x)$]{\includegraphics[width=0.32\textwidth]{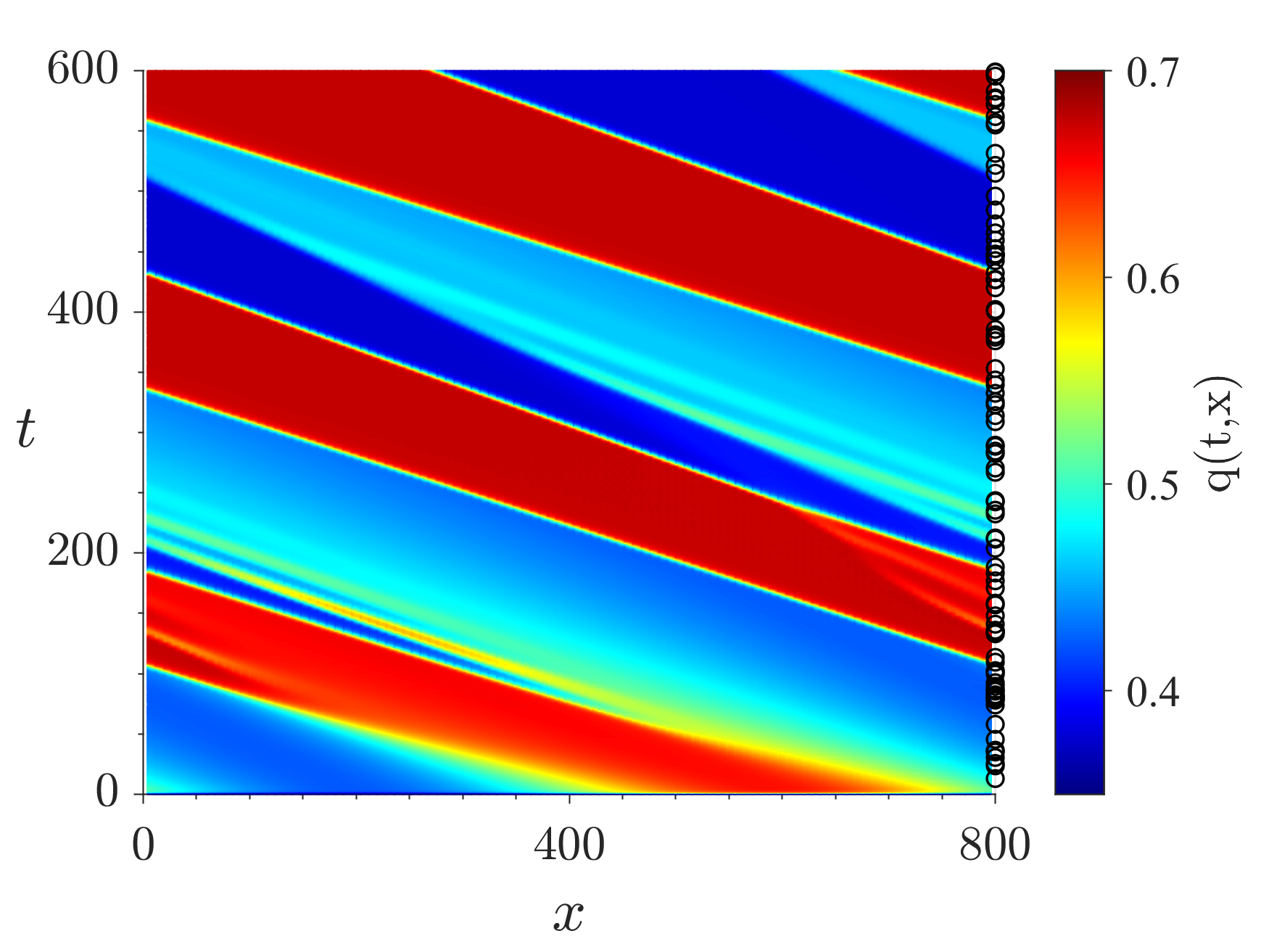}}
    \subfigure[PW equations, $\rho(0,x)$, $q(0,x)$, $\rho(600,x)$ and $q(600,x)$]{\includegraphics[width=0.32\textwidth]{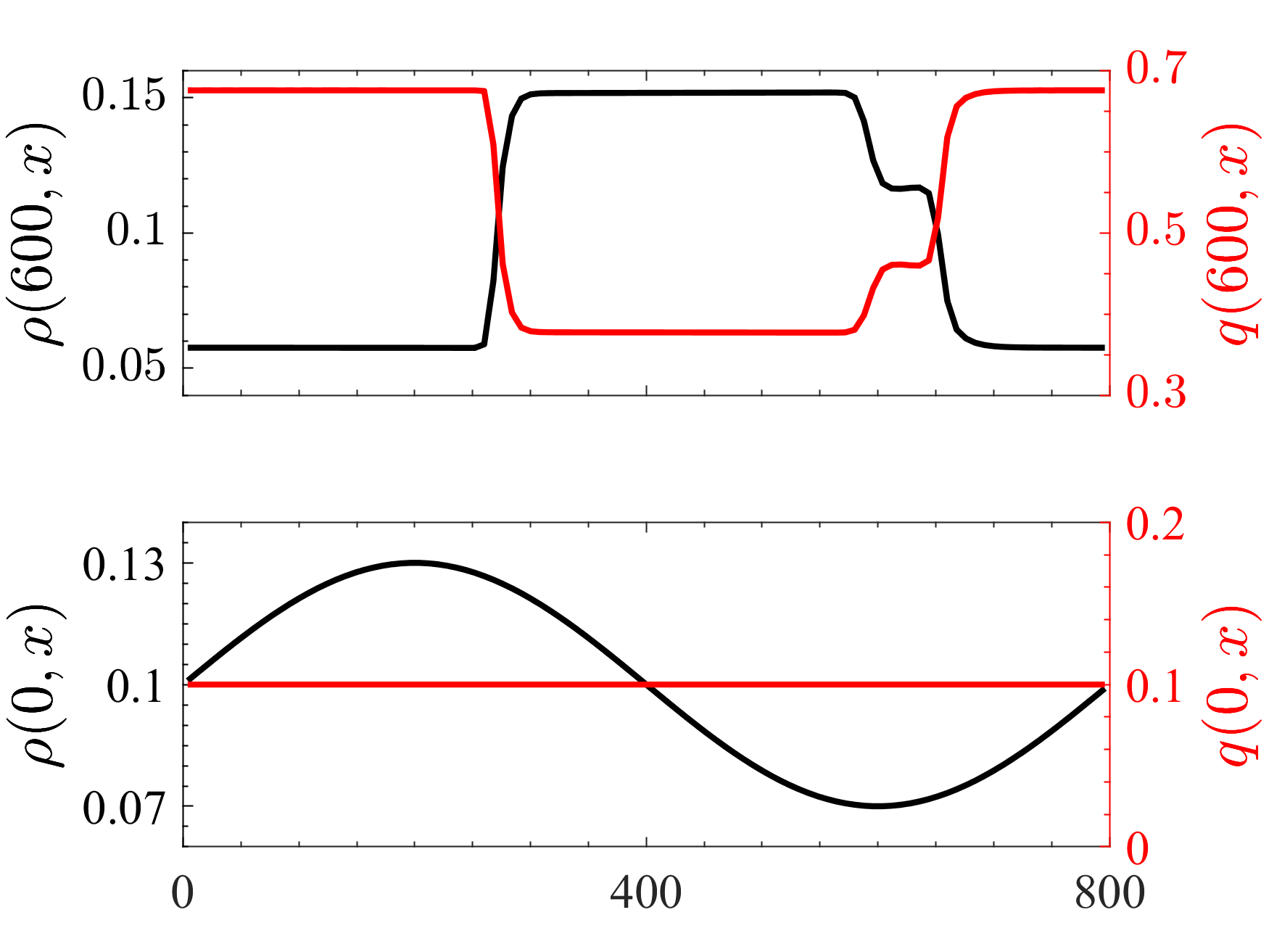}} 
    \caption{Numerical solution of the SW (depth, momentum) and PW (density, momentum) equations obtained with the high-resolution Godunov-scheme in \cref{sb:FP_god} for (a,b,c) $x\in[-5,5]$, $t\in[0,3]$, with the initial data $h(0,x)=0.5 ~exp(-x^2/(2 \cdot 0.4^2))$ and $q(0,x)=0$ shown in panel (c), and (d,e,f) $x\in[0,800]$, $t\in[0,600]$, with the initial data $\rho(0,x)=0.1~(1+0.3sin(2\pi x/L))$ and $q(0,x)=0.1$ shown in panel (f).~Periodic boundary conditions are used for both SW and PW equations.~The black circles on (a,b,c,d) denote the time steps $n_t$, collected from the depicted numerical solutions, to form the training set of GoRINNs.}
    \label{fig:solsSW_PW}
\end{figure}

\subsection{The PW equations of traffic flow}
\label{sb:PW_FP}
The PW equations \citep{payne1971model,whitham2011linear} is a system of $D=2$ non-linear conservation laws, describing a class of macroscopic second-order traffic models.~The traffic flow is modelled by the PW equations by a non-homogeneous system in the form of \cref{eq:genPDE} as:
\begin{align}
    \partial_t \begin{bmatrix}
        \rho \\ q
    \end{bmatrix} + \partial_x \begin{bmatrix}
        q \\ \dfrac{q^2}{\rho} + P(\rho)
    \end{bmatrix} = \begin{bmatrix}
        0 \\ \dfrac{\rho V_e(\rho)-q}{\tau},
    \end{bmatrix} \label{eq:PWcon}
\end{align}
where $\rho(t,x)$ is the car density, as in the LWR model; see \cref{eq:LWR}.~However, here the velocity $v(t,x)$ is not only density-dependent, and is instead modelled by the second PW equation through the momentum $q(t,x)=\rho(t,x)v(t,x)$; we use the momentum equation to obtain a system in a conservative form.~The PW equations in \cref{eq:PWcon} include a source term and a density-dependent traffic pressure term $P\equiv P(\rho)$, which is given by:
\begin{equation}
    P(\rho) = \dfrac{V_0-V_e(\rho)}{2\tau}, \label{eq:PWpres}
\end{equation}
where $\tau$ expresses the speed relaxation time and $V_0/(2\tau)$ is a constant term, such that the pressure is zero when $\rho\rightarrow 0$.

The PW equations in \cref{eq:PWcon,eq:PWpres} consist a class of macroscopic second-order traffic models which take into account the individual behavior of the driver \cite{treiber2013traffic,helbing1998generalized}.~This is incorporated by the selection of microscopic velocity function $V_e(\rho)$, which is usually determined over experimental data from individual cars \citep{treiber2013traffic}.~Here, we consider the optimal velocity (OV) function \citep{bando1995dynamical}, reading:
\begin{equation}
    V_e(\rho) = V(\rho^{-1}) = v_0 \dfrac{tanh(\gamma \rho^{-1}-\beta) + tanh(\beta)}{1+tanh(\beta)},
    \label{eq:OVF}
\end{equation}
where $v_0$, $\gamma$ and $\beta$ are parameters of the individualistic car description, expressing the desired speed, the transition width and the form factor, respectively \citep{treiber2013traffic}.~Selecting the OV function implies that $V_0=v_0$ in the pressure term of \cref{eq:PWpres}.

For our simulations, we consider a city scenario with $\tau=0.65~s$, $v_0=15~m/s$, $\gamma=1/8~m^{-1}$ and $\beta=1.5$ \citep{treiber2013traffic}, in a road of length $L=800~m$.~First, we rescale the PW equations, so that $\rho=\mathcal{O}(1)$, and then we numerically solve them in the 1-dim. spatial domain $x\in[0,L]$ with periodic boundary conditions (a ring road), along the time interval $t\in[0,600]$.~In particular, we consider $N=100$ cells and a, CFL satisfying, time step $dt=0.5$.~To generate the data, we consider 4 different initial conditions following a sinusoidal perturbation $\rho(0,x)=\rho^*(1+\mu~sin(2\pi x/L))$ of a free flow density $\rho^*=0.1$, for a uniformly varying $\mu\in[0.1,0.4]$ with $d\mu = 0.1$; the momentum is set to $q(0,x)=0.1$.

Starting from the aforementioned smooth initial conditions, the PW equations develop travelling waves for the parameter set considered; such a behavior has been reported in agent-based, OV function considering, car-following models \citep{patsatzis2023data,marschler2014implicit}.~During the transition to the travelling waves, the formed contact discontinuities include transonic rarefactions.~Hence, for obtaining the numerical solution via the Godunov-scheme in \cref{sb:FP_god}, the Roe solver is insufficient and thus we employed the HLLE one.~The latter solver also requires the Roe matrix which, for the PW equations, is:
\begin{equation}
    \hat{\mathbf{A}}_{i-1/2}=\begin{bmatrix}
        0 & 1 \\ -\dfrac{\bar{q}^2}{\bar{\rho}^2} - \bar{V} & \dfrac{2 \bar{q}}{\bar{\rho}}
    \end{bmatrix}, ~~ \text{where} ~~ \bar{\rho}= \dfrac{\rho_l+\rho_r}{2}, ~~ \bar{q} =\bar{\rho} \dfrac{q_l \rho_l^{-1/2} + q_r \rho_r^{-1/2}}{\rho_l^{1/2} + \rho_r^{1/2}}, ~~ \bar{V} = \dfrac{V(\rho_l^{-1})-V(\rho_r^{-1})}{2 \tau (\rho_l-\rho_r)}.
    \label{eq:PW_Roe1}
\end{equation}
For the derivation of $\hat{\mathbf{A}}_{i-1/2}$, we followed the general approach introduced by \cite{roe1981approximate}, as presented in detail in \cref{app:RoePayne}.~The resulting numerical solution of the PW equations, using the Van-Leer flux-limiter function, for the sinusoidal initial data with $\mu=0.3$ is depicted in \Cref{fig:solsSW_PW}d,e,f for $\rho$ and $q$, respectively.~It is therein shown that the initial pulse gradually forms a left going shock wave, shown by the abrupt difference of low (high) values of $\rho$ ($q$) to high (low)  ones.~Trailing to the right of the travelling shock wave, fluctuations develop which after $t>500$ lead to the development of another left going shock, with the opposite behavior; abrupt differences from high (low) values of $\rho$ ($q$) to low (high) ones.~From the 4 numerical solutions of the PW equations, we randomly split the 7.5-7.5-85\% of the time-series data to form the training/validation/testing data sets of the GoRINNs, as discussed in \cref{sb:imp}; the training samples collected from the solution with initial data with $\mu=0.3$ are denoted with black circles in \Cref{fig:solsSW_PW}d,e.
\renewcommand{\theequation}{B.\arabic{equation}}
\renewcommand{\thefigure}{B.\arabic{figure}}
\renewcommand{\theproposition}{B.\arabic{proposition}}
\setcounter{equation}{0}
\setcounter{figure}{0}
\setcounter{proposition}{0}
\section{Approximate Riemann solvers: Roe and HLLE solvers}
\label{app:linRiem}
As discussed in \cref{sec:meth}, to locally approximate the solution of the Riemann problem with a contact discontinuity between the left and right states $\mathbf{q}_l$ and $\mathbf{q}_r$, a linearized matrix $\hat{\mathbf{A}}_{i-1/2}\equiv\hat{\mathbf{A}}_{i-1/2}(\mathbf{q}_l,\mathbf{q}_r)$ should be derived for the non-linear system of PDEs in \cref{eq:genPDE}, with the following properties \citep{leveque2002finite,roe1981approximate}:
\begin{enumerate}[label=(\roman*)]
    \item $\hat{\mathbf{A}}_{i-1/2}(\mathbf{q}_l,\mathbf{q}_r)\rightarrow \partial_{\mathbf{u}} \mathbf{f}(\hat{\mathbf{q}})$ as $\mathbf{q}_l, \mathbf{q}_r \rightarrow \hat{\mathbf{q}}$, for consistency with the original non-linear system of PDEs,
    \item $\hat{\mathbf{A}}_{i-1/2}(\mathbf{q}_l,\mathbf{q}_r)$ diagonalizable with real eigenvalues, for hyperbolicity of the linearized system in \cref{eq:linRiem}, and
    \item $\hat{\mathbf{A}}_{i-1/2} (\mathbf{q}_l,\mathbf{q}_r) \cdot (\mathbf{q}_r-\mathbf{q}_l) = \mathbf{f}(\mathbf{q}_r) -\mathbf{f}(\mathbf{q}_l)$, for ensuring that a Godunov-type numerical scheme, when employed to the linearized system in \cref{eq:linRiem}, is conservative (RH condition for systems).
\end{enumerate}
The first two are essential for any linearized Riemann solver, and the latter one characterizes a Roe solver; when all three properties are satisfied, $\hat{\mathbf{A}}_{i-1/2}$ is called Roe matrix.

For the construction of a Roe matrix \cite{roe1981approximate}, one approach is to integrate the Jacobian matrix $\partial_\mathbf{u} \mathbf{f}(\mathbf{u})$ over a suitable path between $\mathbf{q}_l$ and $\mathbf{q}_r$, since it satisfies the first and the third properties.~By considering the line path $\mathbf{u}(\xi) = \mathbf{q}_l + (\mathbf{q}_r-\mathbf{q}_l)\xi$ for $0<\xi<1$, the flux function difference can be written as \citep{leveque2002finite}:
\begin{equation}
     \mathbf{f}(\mathbf{q}_r) -\mathbf{f}(\mathbf{q}_l) = \left[ \int_0^1 \partial_{\mathbf{u}} \mathbf{f}(\mathbf{u}(\xi)) d\xi \right] \left(\mathbf{q}_r - \mathbf{q}_l\right),
\end{equation}
and thus define $\hat{\mathbf{A}}_{i-1/2}=\int_0^1  \partial_{\mathbf{u}} \mathbf{f}(\mathbf{u}(\xi))~d\xi$.~However, this selection doesn't necessarily satisfy the second property of hyperbolicity.~More importantly, the integrals of the Jacobian matrix cannot be computed in a closed form.~To overcome these obstacles, average functions of the left and right states $\mathbf{q}_r$ and $\mathbf{q}_l$ are frequently used (the simplest one being the arithmetic average) for defining an average state $\mathbf{\bar{q}}$ on the interface $i-1/2$, upon which the Roe matrix is computed.~This idea was first introduced by Roe in \citep{roe1981approximate}, resulting in the linearization approach briefly described below.

\subsection{Roe solver}
\label{app:RoeLin}
Roe \cite{roe1981approximate} introduced a change of state variables $\mathbf{z}(\mathbf{u})$ that leads to simpler integral expressions.~In addition, this mapping is invertible, allowing us to derive $\mathbf{u}(\mathbf{z})$, and thus cast the flux function as $\mathbf{f}(\mathbf{u}(\mathbf{z}))=\mathbf{f}(\mathbf{z})$.~Then, one can integrate along the line path $\mathbf{z}(\xi) = \mathbf{z}_l+ (\mathbf{z}_r-\mathbf{z}_l)\xi$ for $0<\xi<1$, where $\mathbf{z}_{l,r} = \mathbf{z}(\mathbf{q}_{l,r})$.~Since $\partial_\xi \mathbf{z}(\xi) = \mathbf{z}_r-\mathbf{z}_l$, the flux function difference now takes the form \citep{leveque2002finite}:
\begin{equation}
     \mathbf{f}(\mathbf{q}_r) -\mathbf{f}(\mathbf{q}_l) = \left[ \int_0^1 \partial_{\mathbf{z}} \mathbf{f}(\mathbf{z}(\xi)) d\xi \right] \left(\mathbf{z}_r - \mathbf{z}_l\right) = \hat{\mathbf{C}}_{i-1/2} \left(\mathbf{z}_r - \mathbf{z}_l\right),
     \label{eq:Roe1}
\end{equation}
where $\hat{\mathbf{C}}_{i-1/2}\in\mathbb{R}^{D\times D}$ includes the integrals of each component of $\partial_{\mathbf{z}} \mathbf{f}(\mathbf{z}(\xi))$.~A parsimonious selection of $\mathbf{z}(\mathbf{u})$ makes these integrals computable in a closed form.~Now, given that $\mathbf{z}(\mathbf{u})$ is invertible, one easily derives:
\begin{equation}
     \mathbf{q}_r -\mathbf{q}_l = \left[ \int_0^1 \partial_{\mathbf{z}} \mathbf{u}(\mathbf{z}(\xi)) d\xi \right] \left(\mathbf{z}_r - \mathbf{z}_l\right) = \hat{\mathbf{B}}_{i-1/2} \left(\mathbf{z}_r - \mathbf{z}_l\right),
     \label{eq:Roe2}
\end{equation}
where $\hat{\mathbf{B}}_{i-1/2}\in\mathbb{R}^{D\times D}$ includes the integrals of each component of $\partial_{\mathbf{z}} \mathbf{u}(\mathbf{z}(\xi))$.~Following \cref{eq:Roe1,eq:Roe2}, the desired matrix is computed as $\hat{\mathbf{A}}_{i-1/2}=\hat{\mathbf{C}}_{i-1/2} \hat{\mathbf{B}}_{i-1/2}^{-1}$, and satisfies the third property that $\hat{\mathbf{A}}_{i-1/2}$ should possess.~Usually the first property is also satisfied, while the second one (regarding hyperbolicity) holds, under specific conditions of the entropy function \citep{harten1983upstream}, which are not explored in this work.

Alternatively to the above general derivation of $\hat{\mathbf{A}}_{i-1/2}$ \citep{leveque2002finite}, a ``reverse engineering" approach was introduced in \citep{ketcheson2020riemann}, which finds some \textit{average} state $\mathbf{\bar{q}}$ as a function of $\mathbf{q}_r$ and $\mathbf{q}_l$, such that $\hat{\mathbf{A}}_{i-1/2}(\mathbf{q}_l,\mathbf{q}_r)=\partial_{\mathbf{u}} \mathbf{f}(\mathbf{\bar{q}})$.~This selection satisfies by definition the first property, and also the third one, when the average state $\mathbf{\bar{q}}$ is computed via:
\begin{equation}
    \partial_{\mathbf{u}} \mathbf{f}(\mathbf{\bar{q}}) \cdot (\mathbf{q}_r-\mathbf{q}_l) = \mathbf{f}(\mathbf{q}_r) -\mathbf{f}(\mathbf{q}_l)
    \label{eq:RoeAlt}
\end{equation}
This technique usually results in the same matrices $\hat{\mathbf{A}}_{i-1/2}$ as the ones derived by \cref{eq:Roe1,eq:Roe2} in \citep{leveque2002finite}, as was the case for all the benchmark problems considered in this work, except from the PW equations, where the later alternative techinique cannot be employed due to the non-linear flux function.~We hereby note that the Roe linearized matrix for the Burger's, LWR and SW system of PDEs was cross-validated with expressions provided in \citep{ketcheson2020riemann,leveque2002finite}, while for the PW system, it was derived in this work.

\subsection{HLLE solver}
\label{app:HLLELin}

In cases where the Roe linearization fails, the HLLE solver, introduced in \citep{harten1983upstream,einfeldt1988godunov}, can be used as an alternative approximate Riemann solver.~According to the HLLE solver, the Riemann problem is approximated by only two waves, propagating with the smallest and largest speeds:
\begin{equation}
    s^1_{i-1/2} = \min_{p}(\min(\lambda_{i-1}^p,\hat{\lambda}_{i-1/2}^p)), \qquad 
    s^2_{i-1/2} = \max_{p}(\max(\lambda_i^p,\hat{\lambda}_{i-1/2}^p)) 
    \label{eq:HLLEs}
\end{equation}
where $\lambda_{i-1}^p$ and $\lambda_i^p$ are the $p$-th eigenvalues of the Jacobian $\partial \mathbf{f}(\mathbf{u})/\partial \mathbf{u}$ at $\mathbf{q}_l$ and $\mathbf{q}_r$ respectively, while $\hat{\lambda}_{i-1/2}^p$ is the $p$-th eigenvalue of the $\hat{\mathbf{A}}_{i-1/2}$ derived via Roe linearization.~Using the above selection, the middle state of the Riemann problem's approximate solution is then computed as:
\begin{equation}
    \mathbf{q}_m = \dfrac{\mathbf{f}(\mathbf{q}_r) - \mathbf{f}(\mathbf{q}_l)-s^2_{i-1/2} \mathbf{q}_r+s^1_{i-1/2} \mathbf{q}_l}{s^1_{i-1/2}-s^2_{i-1/2}},
    \label{eq:HLLEms}
\end{equation}
that is also conservative.~In turn, the two waves are approximated as:
\begin{equation}
    \mathbf{W}^1_{i-1/2} = \mathbf{q}_m - \mathbf{q}_l, \qquad
    \mathbf{W}^2_{i-1/2} = \mathbf{q}_r - \mathbf{q}_m. 
    \label{eq:HLLEwaves}
\end{equation}

The HLLE solver performs similarly to the Roe solver, without however requiring the entropy fixes that the latter requires when facing rarefaction waves (see \citep{leveque2002finite}).~However, the HLLE solver includes only $M_w=2$ waves, which may affect the resolution of systems with more than two variables.~In this work, this solver was only employed for the PW system of hyperbolic PDEs.


\renewcommand{\theequation}{C.\arabic{equation}}
\renewcommand{\thefigure}{C.\arabic{figure}}
\renewcommand{\theproposition}{C.\arabic{proposition}}
\setcounter{equation}{0}
\setcounter{figure}{0}
\setcounter{proposition}{0}
\section{Roe solvers for the benchmark problems}
\label{app:RS4bp}
Here, we provide the expressions of the Roe matrices and the resulting speeds and waves used in the - Roe/HLLE solver based - FV high-resolution Godunov-type method for each benchmark problem considered.

\subsection{The Burger's equation}
\label{app:RoeBurg}
Recall that the Burger's equation in \cref{eq:Burg} is written in the form of a scalar PDE in \cref{eq:genPDE}, where
\begin{equation}
    \mathbf{u} = u, \quad 
    \mathbf{f}(\mathbf{u}) = u^2/2, \quad 
    \partial_{\mathbf{u}} \mathbf{f}(\mathbf{u}) =  u.
\end{equation}
Since this is a scalar hyperbolic PDE, the Roe matrix is scalar, say $\hat{a}_{i-1/2}$, and is provided by the following proposition:
\begin{proposition}
    \cite{leveque2002finite} The Roe matrix of the Burger's equation in \cref{eq:Burg} is the simple arithmetic average: 
    \begin{equation}
        \hat{a}_{i-1/2}=(u_r+u_l)/2. \label{eq:BurgEv}
    \end{equation}
\end{proposition}
\begin{proof}
Following the ``reverse engineering'' technique presented in \cref{app:RoeLin}, we solve \cref{eq:RoeAlt} to define $\bar{\mathbf{q}} = \bar{u}$, which implies:
\begin{equation}
    \bar{u} (u_r-u_l) = (u_r^2-u_l^2)/2 \Rightarrow \bar{u} = (u_r+u_l)/2. \label{eq:RHcBurg}
\end{equation} 
The resulting scalar $\hat{a}_{i-1/2} = \bar{u}= (u_r+u_l)/2$ satisfies all three properties of a Roe matrix, as discussed in \cref{app:RoeLin}, since (i) $\hat{a}_{i-1/2}(u_l,u_r)\rightarrow \partial_{u} \mathbf{f}(\hat{u})$ as $u_l, u_r \rightarrow \hat{u}$, (ii) $\hat{a}_{i-1/2}\in \mathbb{R}$ and (iii) the RH condition is satisfied by construction.
\end{proof}
To perform the FV update in \cref{eq:HRupdate} with the Roe solver, we set $u_l= \mathbf{Q}^n_{i-1}$ and $u_r = \mathbf{Q}^n_i$ at the $n$-th time step and $i$-th cell, and compute, from the Roe matrix $\hat{a}_{i-1/2}$ in \cref{eq:BurgEv}, the speed $s^p_{i-1/2}=\hat{a}_{i-1/2}$ and the wave $\mathbf{W}^p_{i-1/2}=u_r-u_l$ for $p=1$.

\subsection{The LWR equation of traffic flow}
\label{app:RoeLWR}
Recall that the scalar PDE of the LWR model in \cref{eq:LWR} is written in the form of \cref{eq:genPDE}, where
\begin{equation}
    \mathbf{u} = \rho, \quad 
    \mathbf{f}(\mathbf{u}) = \rho v(\rho), \quad 
    \partial_{\mathbf{u}} \mathbf{f}(\mathbf{u}) =  v(\rho)+\rho v'(\rho),
\end{equation}
for $v(\rho)=v_{max}(1-\rho)$ and $v'(\rho)=-v_{max}$.~Since the LWR model is a scalar hyperbolic PDE, the Roe matrix is scalar as well, say $\hat{a}_{i-1/2}$, and is provided by the following proposition:
\begin{proposition}
    \cite{leveque2002finite} The Roe matrix of the LWR equation in \cref{eq:LWR} is:
    \begin{equation}
        \hat{a}_{i-1/2} = v_{max}(1-\rho_r-\rho_l). \label{eq:LWREv}
    \end{equation}
\end{proposition}
\begin{proof}
Following the ``reverse engineering'' technique presented in \cref{app:RoeLin}, we solve \cref{eq:RoeAlt} to define $\bar{\mathbf{q}} = \bar{\rho}$, which implies:
\begin{equation}
    \left(v(\bar{\rho}) + \bar{\rho} v'(\bar{\rho}) \right) (\rho_r-\rho_l) = \rho_r v(\rho_r)-\rho_l v(\rho_l)  \label{eq:RHcLWR}
\end{equation}
After some calculations, the above equation results to the common average $\bar{\rho} = (\rho_r+\rho_l)/2$ for any $v_{max}>0$ and thus, $\hat{a}_{i-1/2}$ is given by
\begin{equation}
\hat{a}_{i-1/2}=v(\bar{\rho}) + \bar{\rho} v'(\bar{\rho}) = v_{max}(1-2\bar{\rho}). 
\end{equation}
The resulting $\hat{a}_{i-1/2}$ satisfies all three properties of a Roe matrix, as discussed in \cref{app:RoeLin}, since (i) $\hat{a}_{i-1/2}(\rho_l,\rho_r)\rightarrow \partial_{\rho} \mathbf{f}(\hat{\rho})$ as $\rho_l, \rho_r \rightarrow \hat{\rho}$, (ii) $\hat{a}_{i-1/2}\in \mathbb{R}^N$, and (iii) the RH condition is satisfied by construction.
\end{proof}

Similarly to the Burger's equation, to perform the FV update in \cref{eq:HRupdate} with the Roe solver, we set $\rho_l= \mathbf{Q}^n_{i-1}$ and $\rho_r = \mathbf{Q}^n_i$ at the $n$-th time step and $i$-th cell, and compute, from the Roe matrix $\hat{a}_{i-1/2}$ in \cref{eq:LWREv}, the speed $s^p_{i-1/2}=\hat{a}_{i-1/2}$ and the wave $\mathbf{W}^p_{i-1/2}=\rho_r-\rho_l$ for $p=1$.

\subsection{The SW equations}
\label{app:RoeSW}
As a reminder, recall that the SW system of PDEs in \cref{eq:SWcon} is written in the form of \cref{eq:genPDE}, where
\begin{equation}
    \mathbf{u} =  \begin{bmatrix}
        h \\ q
    \end{bmatrix}, \quad \mathbf{f}(\mathbf{u}) = \begin{bmatrix}
        q \\ \dfrac{q^2}{h} + \dfrac{1}{2} g h^2  \end{bmatrix}, \quad \partial_{\mathbf{u}} \mathbf{f}(\mathbf{u}) =  \begin{bmatrix}
            0 & 1 \\ -\dfrac{q^2}{h^2} + g h & \dfrac{2q}{h}
        \end{bmatrix}.
\end{equation}
The Roe matrix of the SW system is provided by the following proposition:
\begin{proposition}
    \cite{leveque2002finite} The Roe matrix of the SW equations in \cref{eq:SWcon} is:
\begin{equation}
    \hat{\mathbf{A}}_{i-1/2} = \begin{bmatrix}
        0 & 1 \\ -\dfrac{\bar{q}^2}{\bar{h}^2} + g \bar{h} & 2\dfrac{\bar{q}}{\bar{h}}
    \end{bmatrix}, \qquad \text{where} \qquad \bar{h}=\dfrac{h_l+h_r}{2},\qquad \bar{q}=\bar{h} \dfrac{q_l h_l^{-1/2} + q_r h_r^{-1/2}}{h_l^{1/2}+h_r^{1/2}},
    \label{eq:AmatSW}
\end{equation}
\end{proposition}
\begin{proof}
    Following the technique described in \cref{app:RoeLin}, as introduced in \cite{roe1981approximate}, we employ the change of variables $\mathbf{z}(\mathbf{u})=h^{-1/2}\mathbf{u}$.~This mapping is reversible, resulting to:
\begin{equation}
    \mathbf{z}(\mathbf{u}) = \begin{bmatrix} z^1 \\ z^2 \end{bmatrix} = \begin{bmatrix} h^{1/2} \\ q h^{-1/2} \end{bmatrix} , \qquad \mathbf{u}(\mathbf{z}) = \begin{bmatrix} (z^1)^2 \\ z^1 z^2 \end{bmatrix} \Rightarrow \partial_{\mathbf{z}} \mathbf{u}(\mathbf{z})= \begin{bmatrix}
        2 z^1 & 0 \\ z^2 & z^1
    \end{bmatrix}.
    \label{eq:SWroe1}
\end{equation}  
Now, the flux function and its Jacobian is written in terms of $\mathbf{z}$ as:
\begin{equation}
    \mathbf{f}(\mathbf{z}) = \begin{bmatrix}
        z^1 z^2 \\ (z^2)^2 + \dfrac{1}{2} g (z^1)^4  \end{bmatrix}, \quad \partial_{\mathbf{z}} \mathbf{f}(\mathbf{z}) =  \begin{bmatrix}
            z^2 & z^1 \\ 2 g (z^1)^3  & 2 z^2
        \end{bmatrix}.
        \label{eq:SWroe2}
\end{equation}
Next, setting the line path $z^p(\xi) = z^p_l+ (z^p_r-z^p_l)\xi$ for $p=1,2$, the integration of all elements of the matrices in \cref{eq:SWroe1,eq:SWroe2} is straightforward, resulting to:
\begin{equation}
    \mathbf{\hat{B}}_{i-1/2} = \begin{bmatrix}
        2 \bar{Z}^1 & 0 \\ \bar{Z}^2 & \bar{Z}^1 
    \end{bmatrix}, \qquad \mathbf{\hat{C}}_{i-1/2} = \begin{bmatrix}
        \bar{Z}^2 & \bar{Z}^1 \\ 2 g\bar{Z}^1 \bar{h} & 2\bar{Z}^2 
    \end{bmatrix},
\end{equation}
where $\bar{Z}^p=(z^p_l+z^p_r)/2$ and $\bar{h} = (h_l+h_r)/2$.~Thus, the matrix $\mathbf{\hat{A}}_{i-1/2}$ takes the form:
\begin{equation}
    \hat{\mathbf{A}}_{i-1/2}=\hat{\mathbf{C}}_{i-1/2} \hat{\mathbf{B}}_{i-1/2}^{-1}=\begin{bmatrix}
        0 & 1 \\ -\dfrac{\bar{q}^2}{\bar{h}^2} + g \bar{h} & \dfrac{2 \bar{q}}{\bar{h}}
    \end{bmatrix},
    \label{eq:AmatSW1}
\end{equation}
where 
\begin{equation}
    \bar{q} =\bar{h} \dfrac{q_l h_l^{-1/2} + q_r h_r^{-1/2}}{h_l^{1/2} + h_r^{1/2}}.
\end{equation}
For the computation of $\mathbf{\hat{A}}_{i-1/2}$, one can alternatively follow the ``reverse engineering'' technique presented in \cref{app:linRiem} and solve \cref{eq:RoeAlt} for $\bar{\mathbf{q}} = (\bar{h}, \bar{q})$, as done in \citep{ketcheson2020riemann}.~This procedure implies the same average states, $\bar{h}$ and $\bar{q}$, found above, and thus results to the same $\mathbf{\hat{A}}_{i-1/2}$ matrix of \cref{eq:AmatSW1}.

The resulting $\mathbf{\hat{A}}_{i-1/2}$ in \cref{eq:AmatSW1} is the Jacobian $\partial_{\mathbf{u}} \mathbf{f}(\mathbf{u})$ evaluated in $(\bar{h},\bar{q})$.~This implies that the property (i) in \cref{app:linRiem} is satisfied; i.e., when $(h_l,q_l)$,$(h_r,q_r)\rightarrow (\hat{h},\hat{q})$, then $\bar{h}\rightarrow\hat{h}$ and $\bar{q}\rightarrow\hat{q}$, so that $\mathbf{\hat{A}}_{i-1/2}((h_l,q_l),(h_r,q_r))\rightarrow \partial_{(h,q)} \mathbf{f}(\hat{h},\hat{q})$.~In addition, the matrix $\mathbf{\hat{A}}_{i-1/2}$ in \cref{eq:AmatSW1} has always real eigenvalues, thus satisfying property (ii) in \cref{app:linRiem}.~This is because:
\begin{equation}
Tr(\mathbf{\hat{A}}_{i-1/2})^2-4 Det(\mathbf{\hat{A}}_{i-1/2}) = 4 g \bar{h}\geq 0, 
\end{equation}
since $g>0$ and $\bar{h}\geq 0$.~Finally, as shown in \cref{app:RoeLin}, the matrix $\mathbf{\hat{A}}_{i-1/2}$ also satisfies the RH condition by construction.~Hence, all three properties in \cref{app:linRiem} are satisfied, defining $\mathbf{\hat{A}}_{i-1/2}$ in \cref{eq:AmatSW1} a Roe matrix.
\end{proof}

The eigenvalues $\hat{\lambda}^{1,2}_{i-1/2}$ and the eigenvectors $\mathbf{\hat{r}}^{1,2}_{i-1/2}$ of the matrix $\mathbf{\hat{A}}_{i-1/2}$ can be computed analytically as:
\begin{equation}
    \hat{\lambda}^{1,2}_{i-1/2} = \dfrac{\bar{q}}{\bar{h}} \pm \sqrt{g \bar{h}}, \qquad \mathbf{\hat{r}}^{1,2}_{i-1/2} = \begin{bmatrix}
        1 \\ \dfrac{\bar{q}}{\bar{h}} \pm \sqrt{g \bar{h}}
    \end{bmatrix}.
    \label{eq:SWev}
\end{equation}

Hence, to perform the FV update in \cref{eq:HRupdate} with the Roe solver, we set $(h_l,q_l)^\top = \mathbf{Q}^n_{i-1}$ and $(h_r,q_r)^\top = \mathbf{Q}^n_i$ at the $n$-th time step and $i$-th cell, and compute (i) the speeds $s^p_{i-1/2}$ by the eigenvalues $\hat{\lambda}^p_{i-1/2}$ in \cref{eq:SWev} and, (ii) the waves $\mathbf{W}^p_{i-1/2}$ (computed via the eigenvectors in \cref{eq:SWev}) by the expressions:
\begin{equation}
    \mathbf{W}^p_{i-1/2} = a^p_{i-1/2} \mathbf{\hat{r}}^p_{i-1/2}, \qquad \text{where} \quad \mathbf{a}_{i-1/2}=\begin{bmatrix}
        a^1_{i-1/2} \\ a^2_{i-1/2}
    \end{bmatrix} = \begin{bmatrix}
        \mathbf{\hat{r}}^1_{i-1/2} & \mathbf{\hat{r}}^2_{i-1/2}
    \end{bmatrix}^{-1} \begin{bmatrix}
        h_r-h_l \\ q_r-q_l 
    \end{bmatrix} 
\end{equation}
for $p=1,2$.

\subsection{The PW equations}
\label{app:RoePayne}

Recall that the system of PDEs provided by the PW model in \cref{eq:PWcon} with the use of the optimal velocity function is written in the form of \cref{eq:genPDE} where
\begin{equation}
    \mathbf{u} = \begin{bmatrix}
        \rho \\ q
    \end{bmatrix}, \quad \mathbf{f}(\mathbf{u}) = \begin{bmatrix}
        q \\ \dfrac{q^2}{\rho} +  \dfrac{V_0 - V(\rho^{-1})}{2\tau}  \end{bmatrix}, \quad \partial_{\mathbf{u}} \mathbf{f}(\mathbf{u}) = \begin{bmatrix}
            0 & 1 \\ -\dfrac{q^2}{\rho^2} + \dfrac{V'(\rho^{-1})}{2\tau \rho^2} & \dfrac{2q}{\rho}
        \end{bmatrix}
\end{equation}
where $V'(\cdot)$ denotes the derivative of the optimal velocity function.~The respective Roe matrix is provided by the following proposition:
\begin{proposition}
    The Roe matrix of the PW equations in \cref{eq:PWcon} is: 
\begin{equation}
    \hat{\mathbf{A}}_{i-1/2}=\begin{bmatrix}
        0 & 1 \\ -\dfrac{\bar{q}^2}{\bar{\rho}^2} - \bar{V} & \dfrac{2 \bar{q}}{\bar{\rho}}
    \end{bmatrix}, ~~ \text{where} ~~ \bar{\rho}= \dfrac{\rho_l+\rho_r}{2}, ~~ \bar{q} =\bar{\rho} \dfrac{q_l \rho_l^{-1/2} + q_r \rho_r^{-1/2}}{\rho_l^{1/2} + \rho_r^{1/2}}, ~~ \bar{V} = \dfrac{V(\rho_l^{-1})-V(\rho_r^{-1})}{2 \tau (\rho_l-\rho_r)}.
    \label{eq:AmatP}
\end{equation}
\end{proposition}
\begin{proof}
    Following the technique described in \cref{app:RoeLin}, as introduced in \cite{roe1981approximate}, we employ the change of variables $\mathbf{z}(\mathbf{u})=\rho^{-1/2}\mathbf{u}$, inspired by the shallow water Roe linearization.~This mapping is reversible, resulting to:
\begin{equation}
    \mathbf{z}(\mathbf{u}) = \begin{bmatrix} z^1 \\ z^2 \end{bmatrix} = \begin{bmatrix} \rho^{1/2} \\ q \rho^{-1/2} \end{bmatrix} , \qquad \mathbf{u}(\mathbf{z}) = \begin{bmatrix} (z^1)^2 \\ z^1 z^2 \end{bmatrix} \Rightarrow \partial_{\mathbf{z}} \mathbf{u}(\mathbf{z}) = \begin{bmatrix}
        2 z^1 & 0 \\ z^2 & z^1
    \end{bmatrix}.
    \label{eq:Proe1}
\end{equation}
Now, the flux function and its Jacobian is written in terms of $\mathbf{z}$ as:
\begin{equation}
    \mathbf{f}(\mathbf{z}) =  \begin{bmatrix}
        z^1 z^2 \\ (z^2)^2 + \dfrac{V_0 - V((z^1)^{-2})}{2\tau}
    \end{bmatrix}, \quad \partial_{\mathbf{z}}  \mathbf{f}(\mathbf{z})= \begin{bmatrix}
        z^2 & z^1 \\  \dfrac{V'((z^1)^{-2})}{\tau (z^1)^3} & 2 z^1
    \end{bmatrix}.
    \label{eq:Proe2}
\end{equation}
Next, setting the line path $z^p(\xi) = z^p_l+ (z^p_r-z^p_l)\xi$ for $p=1,2$, the integration of all elements of the matrices in \cref{eq:Proe1,eq:Proe2} is straightforward, except only from the bottom left element of the matrix in \cref{eq:Proe2}, which is integrated as:
\begin{align}
    \int_0^1 \dfrac{V'((z^1(\xi))^{-2})}{\tau (z^1(\xi))^3} d \xi = \int_0^1 -\dfrac{1}{2\tau}\dfrac{d V((z^1(\xi))^2)}{dz^1(\xi)} d\xi = -\dfrac{1}{2\tau} \int_0^1  \left( \dfrac{d V((z^1(\xi))^2)}{d\xi}\right) \left( \dfrac{dz^1(\xi)}{d\xi}\right)^{-1} d\xi = \nonumber \\
    = -\dfrac{1}{2\tau (z_l^1-z^1_r)} \int_0^1  \dfrac{d V((z^1(\xi))^2)}{d\xi} d\xi = -\dfrac{V((z^1_r)^2)-V((z^1_l)^2)}{2\tau (z_r^1-z^1_l)}\equiv - \hat{V},
\end{align}
where the chain rule is employed at the third step of the above calculation.~Then, the matrices in \cref{eq:Proe1,eq:Proe2} are computed as:
\begin{equation}
    \mathbf{\hat{B}}_{i-1/2} = \begin{bmatrix}
        2 \bar{Z}^1 & 0 \\ \bar{Z}^2 & \bar{Z}^1 
    \end{bmatrix}, \qquad \mathbf{\hat{C}}_{i-1/2} = \begin{bmatrix}
        \bar{Z}^2 & \bar{Z}^1 \\ -\hat{V} & 2\bar{Z}^2 
    \end{bmatrix},
\end{equation}
where $\bar{Z}^p=(Z^p_{i-1}+Z^p_i)/2$.~Returning to the original variables $\mathbf{u}$, the matrix $\mathbf{\hat{A}}_{i-1/2}$ takes the form:
\begin{equation}
    \hat{\mathbf{A}}_{i-1/2}=\hat{\mathbf{C}}_{i-1/2} \hat{\mathbf{B}}_{i-1/2}^{-1}=\begin{bmatrix}
        0 & 1 \\ -\dfrac{\bar{q}^2}{\bar{\rho}^2} - \bar{V} & \dfrac{2 \bar{q}}{\bar{\rho}}
    \end{bmatrix},
    \label{eq:AmatP1}
\end{equation}
where 
\begin{equation}
    \bar{\rho}= \dfrac{\rho_l+\rho_r}{2}, \qquad \bar{q} =\bar{\rho} \dfrac{q_l \rho_l^{-1/2} + q_r \rho_r^{-1/2}}{\rho_l^{1/2} + \rho_r^{1/2}}, \qquad \bar{V} = \dfrac{V(\rho_l^{-1})-V(\rho_r^{-1})}{2 \tau (\rho_l-\rho_r)}.
\end{equation}
Note that $\bar{V}$ is not ill-conditioned and, in fact, 
\begin{equation}
    \bar{V} \rightarrow - \dfrac{1}{2\tau \hat{\rho}^2} \dfrac{\partial V(\hat{\rho}^{-1})}{\partial \rho}  \in\mathbb{R}, \qquad \text{as} \qquad \rho_l,\rho_r\rightarrow\hat{\rho}.
    \label{eq:OVMprop}
\end{equation}

In contrast to the SW system of PDEs, here the matrix $\mathbf{\hat{A}}_{i-1/2}$ in \cref{eq:AmatP} is not the Jacobian evaluated in $(\bar{\rho},\bar{q})$.~It is in fact impossible to derive the Roe matrix following the alternative technique presented in \cref{app:linRiem}, due to the non-linearity of the optimal velocity function, that hinders the derivation of $(\bar{\rho},\bar{q})$.

Nonetheless, the matrix $\mathbf{\hat{A}}_{i-1/2}$ in \cref{eq:AmatP1} satisfies  property (i) in \cref{app:linRiem}; i.e., when $(\rho_l,q_l)$,$(\rho_r,q_r)\rightarrow (\hat{\rho},\hat{q})$, then $\bar{\rho}\rightarrow\hat{\rho}$, $\bar{q}\rightarrow\hat{q}$ and $\bar{V} \rightarrow - \dfrac{V'(\hat{\rho}^{-1})}{2\tau \hat{\rho}^2}$ as shown in \cref{eq:OVMprop}, implying that  $\mathbf{\hat{A}}_{i-1/2}((\rho_l,q_l),(\rho_r,q_r))\rightarrow \partial_{(\rho,q)} \mathbf{f}(\hat{\rho},\hat{q})$.~In addition, the matrix $\mathbf{\hat{A}}_{i-1/2}$ in \cref{eq:AmatP1} has always real eigenvalues, thus satisfying property (ii) in \cref{app:linRiem}.~This is because:
\begin{equation}
Tr(\mathbf{\hat{A}}_{i-1/2})^2-4 Det(\mathbf{\hat{A}}_{i-1/2}) = \dfrac{2v_0\left(tanh(\beta-\gamma \rho_r^{-1})-tanh(\beta-\gamma \rho_l^{-1})\right)}{\tau(1+tanh(\beta))(\rho_r-\rho_l)} \geq 0
\end{equation}
for any $\rho_r\neq\rho_l\geq0$, since $\gamma>0$, $v_0/(\tau(1+tanh(\beta)))>0$ and $tanh(\cdot)$ is a monotonically increasing function.~Note that when $\rho_r=\rho_l$, one should use the limit in \cref{eq:OVMprop} to construct the matrix $\mathbf{\hat{A}}_{i-1/2}$.~Finally, as shown in \cref{app:RoeLin}, the matrix $\mathbf{\hat{A}}_{i-1/2}$ also satisfies the RH condition by construction.~Hence, all three properties in \cref{app:linRiem} are satisfied, making $\mathbf{\hat{A}}_{i-1/2}$ in \cref{eq:AmatP1} a Roe matrix.
\end{proof}

The eigenvalues $\hat{\lambda}^{1,2}_{i-1/2}$ and the eigenvectors $\mathbf{\hat{r}}^{1,2}_{i-1/2}$ of the matrix $\mathbf{\hat{A}}_{i-1/2}$ can be computed analytically as:
\begin{equation}
    \hat{\lambda}^{1,2}_{i-1/2} = \dfrac{\bar{q}}{\bar{\rho}} \pm \sqrt{-\bar{V}}, \qquad \mathbf{\hat{r}}^{1,2}_{i-1/2} = \begin{bmatrix}
        1 \\ \dfrac{\bar{q}}{\bar{\rho}} \pm \sqrt{ -\bar{V}}
    \end{bmatrix}.
    \label{eq:PayneEv}
\end{equation}

Hence, to perform the FV update in \cref{eq:HRupdate} with the Roe solver, (i) the speeds $s^p_{i-1/2}$ are provided by the eigenvalues $\hat{\lambda}^p_{i-1/2}$ in \cref{eq:PayneEv} and, (ii) the waves $\mathbf{W}^p_{i-1/2}$ (computed via the eigenvectors in \cref{eq:PayneEv}) are provided by the expressions:
\begin{equation}
    \mathbf{W}^p_{i-1/2} = a^p_{i-1/2} \mathbf{\hat{r}}^p_{i-1/2}, \qquad \text{where} \quad \mathbf{a}_{i-1/2}=\begin{bmatrix}
        a^1_{i-1/2} \\ a^2_{i-1/2}
    \end{bmatrix} = \begin{bmatrix}
        \mathbf{\hat{r}}^1_{i-1/2} & \mathbf{\hat{r}}^2_{i-1/2}
    \end{bmatrix}^{-1} \begin{bmatrix}
        \rho_r-\rho_l \\ q_r-q_l 
    \end{bmatrix} 
\end{equation}
for $p=1,2$.~For the FV update in \cref{eq:HRupdate} with the HLLE solver, only the eigenvalues $\hat{\lambda}^p_{i-1/2}$ in \cref{eq:PayneEv} are required, as shown in \cref{app:HLLELin}.

\clearpage
\newpage
\renewcommand{\theequation}{D.\arabic{equation}}
\renewcommand{\thefigure}{D.\arabic{figure}}
\setcounter{equation}{0}
\setcounter{figure}{0}
\section{Supplemental Figures for PW equation when learning $\mathcal{N}(\rho,q)$}

\begin{figure}[!h]
    \centering
    \subfigure[$\lvert \hat{\rho}(t,x)-\rho(t,x)\rvert$]{\includegraphics[width=0.45\textwidth]{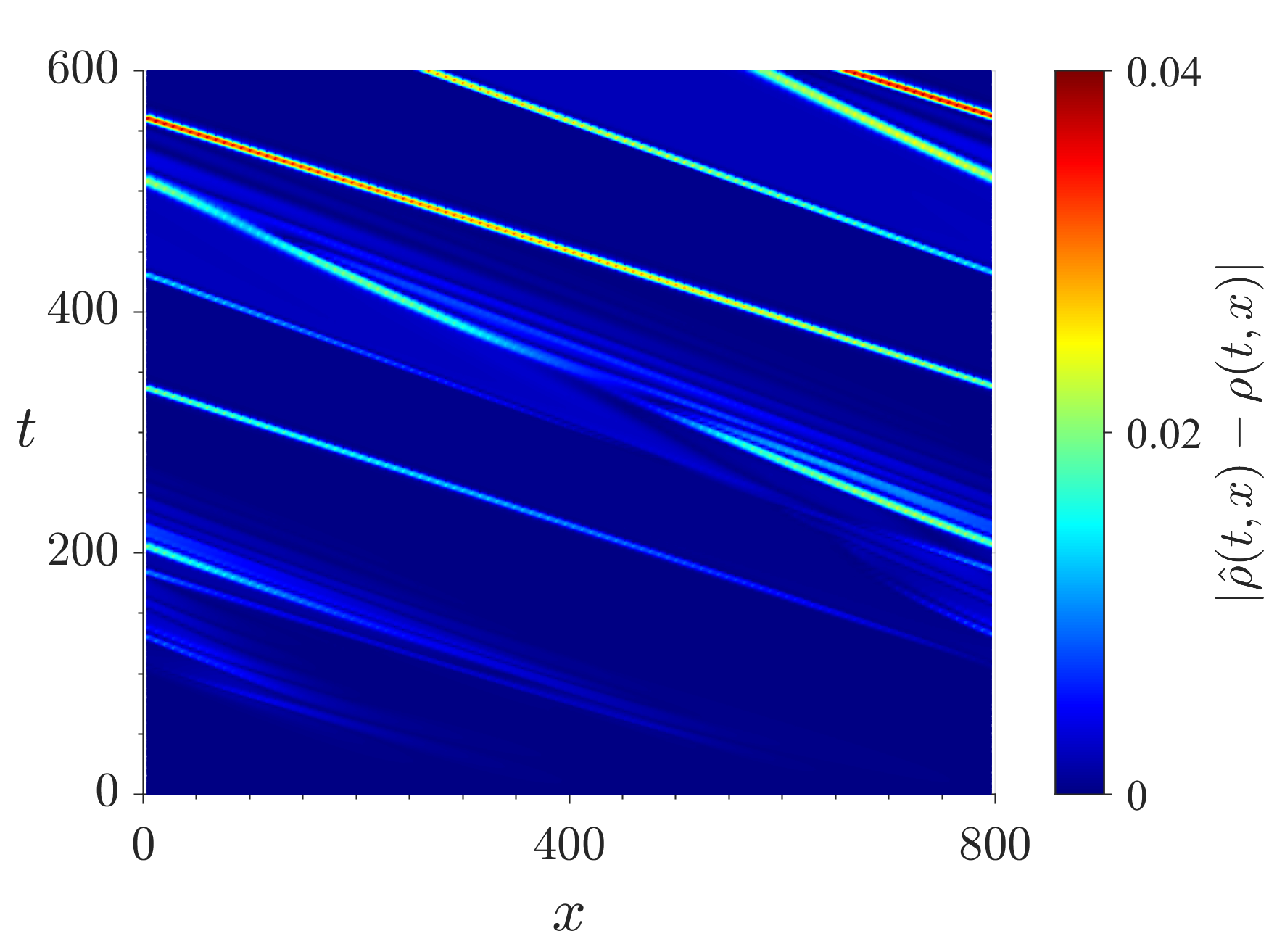}}
    \subfigure[$\lvert \hat{q}(t,x)-q(t,x)\rvert$]{\includegraphics[width=0.45\textwidth]{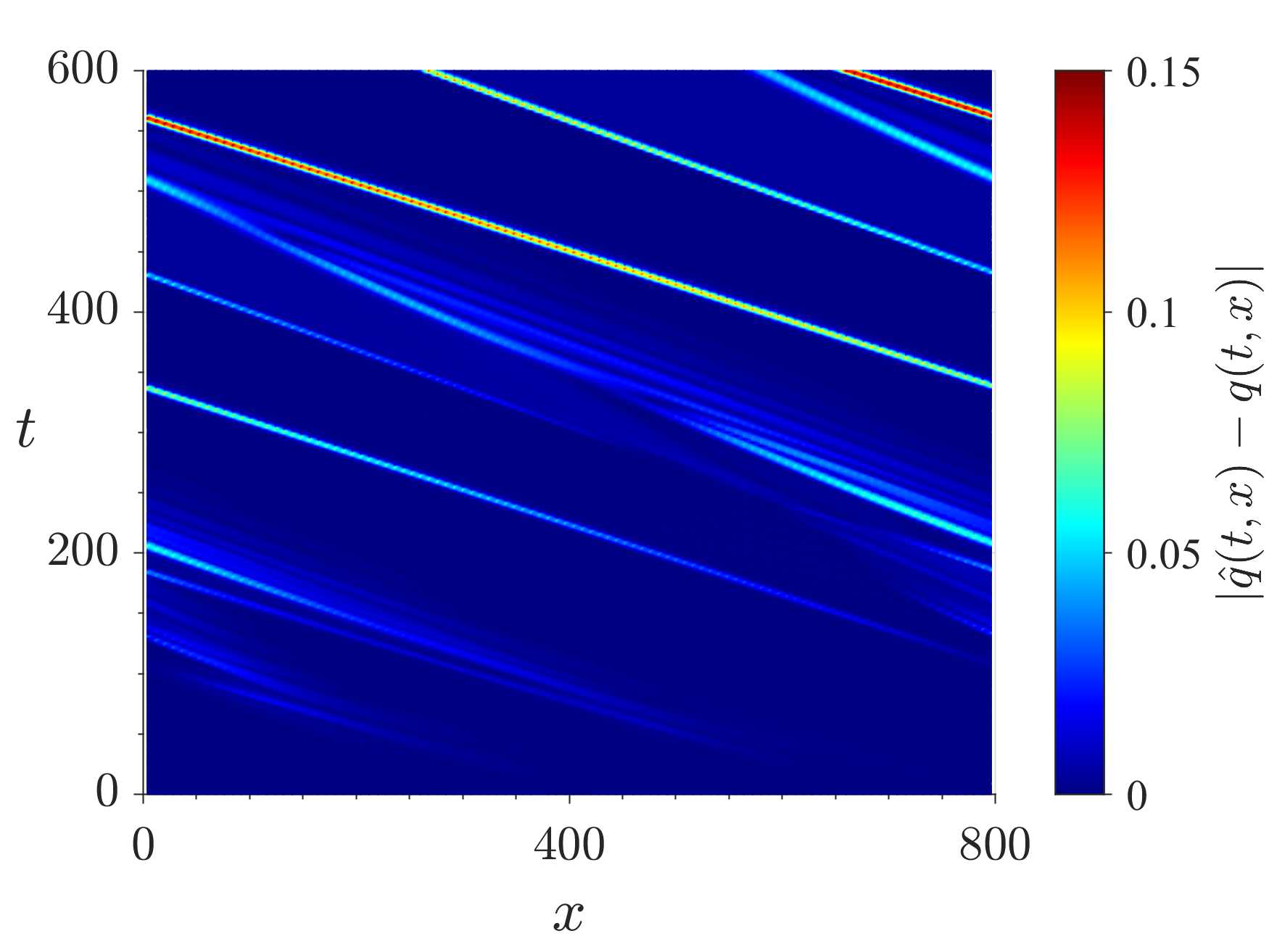}} \\
    \subfigure[$\rho(500,x)$ vs $\hat{\rho}(500,x)$]{\includegraphics[width=0.45\textwidth]{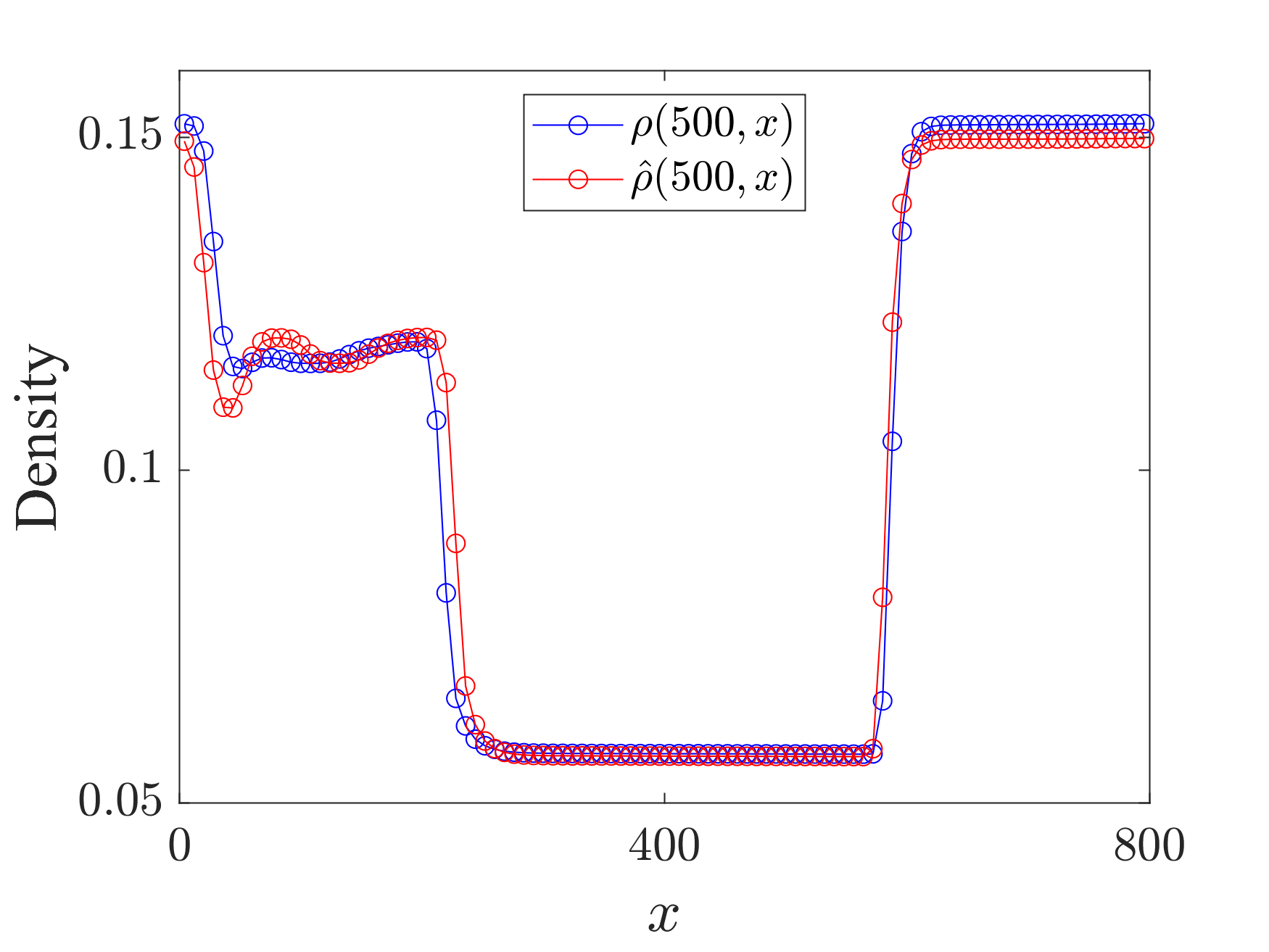}}
    \subfigure[$P(\rho)$ vs GoRINNs $\mathcal{N}(\rho,q)$]{\includegraphics[width=0.45\textwidth]{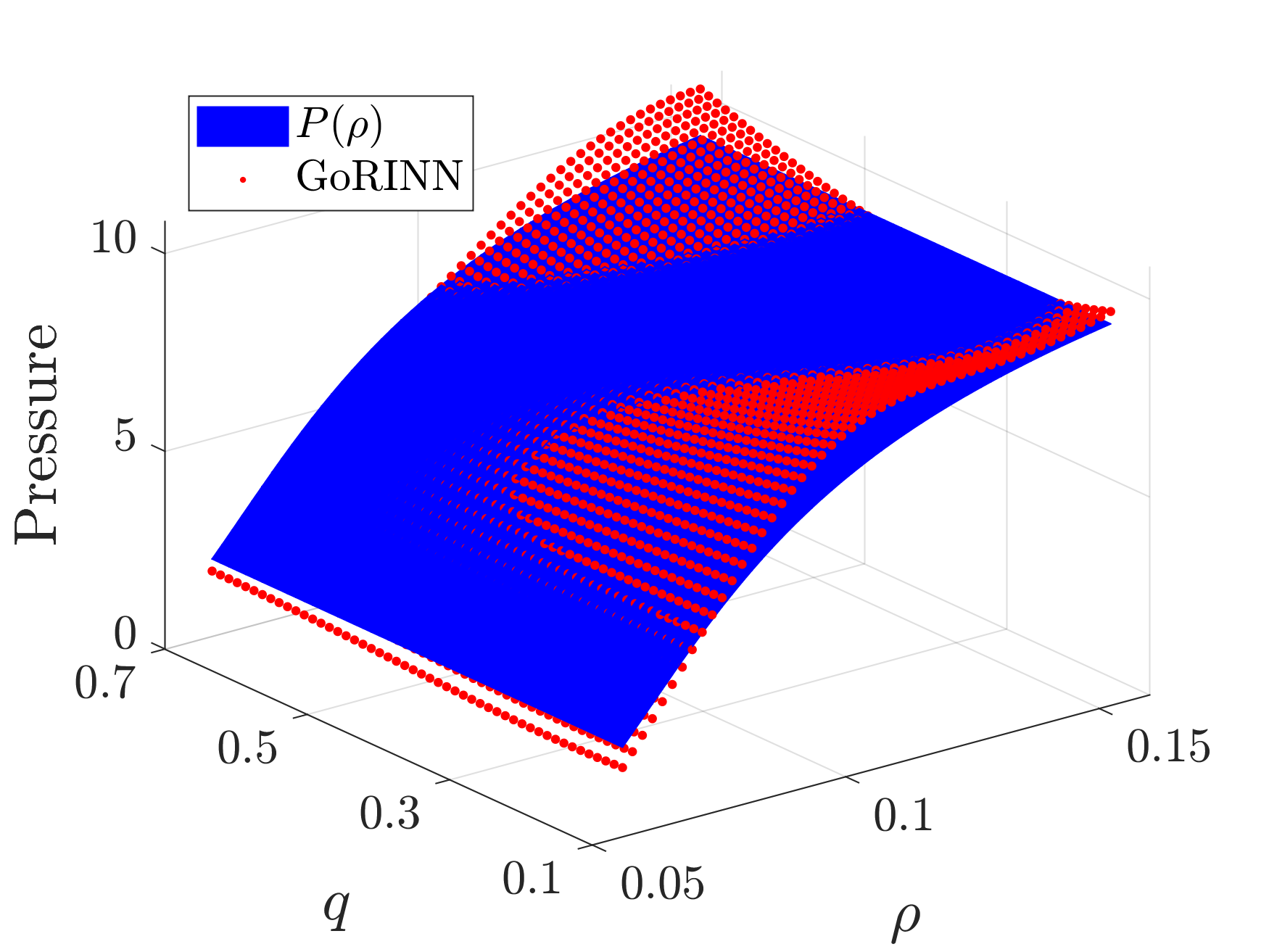}}
    \caption{Numerical accuracy of GoRINNs, assuming a pressure closure of the form $\mathcal{N}(\rho,q)$, for the PW equations.~(a,b) Absolute errors of the numerical solution provided by the GoRINNs learned equations in \cref{eq:unPW} ($\hat{\rho}$ and $\hat{q}$) vs the one provided by the PW equations in \cref{eq:PWcon} ($\rho$ and $q$).~(c) Solutions compared in panel (a) at $t=500$.~(d) Analytically known pressure closure $P(\rho)$ vs the $\mathcal{N}(\rho,q)$ functional learned with GoRINNs.}
    \label{fig:GNN_PW2}
\end{figure}

%
%
%

\end{document}